\documentclass[10pt,a4paper,openany,oneside]{article}

\usepackage{amsfonts}
\usepackage{fancyhdr}
\usepackage{verbatim}
\usepackage{graphicx,color}
\usepackage{authblk}

\usepackage{times}
\usepackage[latin1]{inputenc}
\usepackage[english]{babel}
\usepackage{mathrsfs}
\usepackage{stmaryrd}
\usepackage{amssymb,amsmath,amsthm}
\usepackage{graphicx}
\usepackage{subcaption}
\usepackage{eurosym}
\usepackage{bigints}
\usepackage{upgreek}
\usepackage{mathtools}
\usepackage{thmtools}

\usepackage{color}
\usepackage{enumitem}
\usepackage[arrow, matrix, curve]{xy}
\usepackage{tikz}
\usetikzlibrary{decorations.pathreplacing,angles,quotes}
\usetikzlibrary{patterns}

\usepackage{stmaryrd}
\SetSymbolFont{stmry}{bold}{U}{stmry}{m}{n}

\allowdisplaybreaks
\newtheoremstyle{dotless}{}{}{\itshape}{}{\bfseries}{}{\newline}{}
\theoremstyle{dotless}
\newtheorem{them}{Theorem}[section]	

\newtheorem{kor}[them]{Corollary}
\newtheorem{defi}[them]{Definition}
\newtheorem{lem}[them]{Lemma}
\newtheorem{remark}[them]{Remark}

\newtheorem{bsp}[them]{Example}

\renewenvironment{proof}{{\scshape Proof: \hspace{0.01 cm}}}{\qed}

\newcommand{\beq}{\begin{equation}}
\newcommand{\eeq}{\end{equation}}

\newcommand{\R}{\mathbb{R}}
\newcommand{\N}{\mathbb{N}}

\newcommand{\C}{\mathbb{C}}

\def\cH{{\mathcal H}}
\def\cE{{\mathcal E}}
\def\cD{{\mathcal D}}
\def\cG{{\mathcal G}}

\newcommand{\cW}{\mathcal{W}}
\newcommand{\eps}{\varepsilon}
\newcommand{\Uans}{\bs{U}_{\rm ans}}
\newcommand{\Uext}{\bs{U}_{\rm ext}}
\newcommand{\Umod}{\bs{U}_{\rm mod}}
\newcommand{\Uanshat}{\widehat{\bs{U}}_{\rm ans}}
\newcommand{\Uexthat}{\widehat{\bs{U}}_{\rm ext}}
\newcommand{\Umodhat}{\widehat{\bs{U}}_{\rm mod}}
\newcommand{\Uansone}{U_{{\rm ans},1}}
\newcommand{\Uanstwo}{U_{{\rm ans},2}}
\newcommand{\Uansthree}{U_{{\rm ans},3}}
\newcommand{\Uextone}{U_{{\rm ext},1}}
\newcommand{\Uexttwo}{U_{{\rm ext},2}}

\newcommand{\Ahat}{\widehat{A}}
\newcommand{\Abarhat}{\widehat{\overline{A}}}
\newcommand{\Ktil}{\widetilde{K}}
\newcommand{\Reshat}{\widehat{\bs{\operatorname{Res}}}}
\newcommand{\Res}{\bs{\operatorname{Res}}}

\newcommand{\Trho}{T_{\rho_0,\eps_0}}
\newcommand{\Jr}{J_{\rho_0,\eps_0}}

\newcommand{\ie}{\mathrm{i}}
\newcommand{\e}{\mathrm{e}}
\newcommand{\norm}[1]{\left\lVert #1 \right\rVert}
\newcommand{\sgn}[1]{\operatorname{sgn} \left( #1 \right)}
\newcommand{\bs}[1]{\boldsymbol{#1}}
\newcommand{\cc}{\mathrm{c.c.}}
\newcommand{\mv}{:}

\newcommand{\dist}{\operatorname{dist}}

\newcommand{\pa}{\partial}

\DeclareMathOperator\Hdiv{H_{div}}
\DeclareMathOperator\Hcurl{H_{curl}}

\makeatletter
\newcommand\Label[1]{&\refstepcounter{equation}\left(\theequation\right)\ltx@label{#1}&}
\makeatother

\newcommand\widewidehat[1]{\arraycolsep=0pt\relax%
\begin{array}{c}
\stretchto{
\scaleto{
\scalerel*[\widthof{\ensuremath{#1}}]{\kern-.5pt\bigwedge\kern-.5pt}
{\rule[-\textheight/2]{1ex}{\textheight}} 
}{\textheight} %
}{0.5ex}\\			
#1\\				
\rule{-1ex}{0ex}
\end{array}
}

\pgfdeclarepatternformonly{mydots}{\pgfqpoint{-1pt}{-1pt}}{\pgfqpoint{1pt}{1pt}}{\pgfqpoint{10pt}{10pt}}
{%
\pgfpathcircle{\pgfqpoint{0pt}{0pt}}{.5pt}%
\pgfusepath{fill}%
}%

\begin{document}
	
\author{Tom\'{a}\v{s} Dohnal \\
Institute of Mathematics, Martin Luther University Halle-Wittenberg,\\
06099 Halle (Saale), Germany. Email: tomas.dohnal@mathematik.uni-halle.de \and
Roland Schnaubelt
\\Department of Mathematics,
Karlsruhe Institute of Technology, \\ 76128 Karlsruhe, Germany. Email: schnaubelt@kit.edu \and
Daniel P. Tietz \\
Analysis and PDE Unit, Okinawa Institute of Science and Technology, \\
904-0495 Onna, Japan. Email: daniel.tietz@oist.jp}
\title{Rigorous Envelope Approximation for Interface Wave-Packets in Maxwell's Equations with 2D Localization}
\date{\today}
\maketitle

\begin{abstract}
We study transverse magnetic (vector valued) wave-packets in the time dependent Kerr nonlinear Maxwell's equations at the interface of two inhomogeneous dielectrics with an instantaneous material response. The resulting model is quasilinear. The problem is solved on each side of the interface and the fields are coupled via natural interface conditions. The wave-packet is localized at the interface and propagates in the tangential direction. For a slowly modulated envelope approximation the nonlinear Schr\"odinger equation is formally derived as an amplitude equation for the envelope. We rigorously justify the approximation in a Sobolev space norm on the corresponding asymptotically large time intervals. The well-posedness result for the quasilinear Maxwell problem builds on the local theory of [R. Schnaubelt and M. Spitz, Local wellposedness of quasilinear Maxwell equations with conservative interface conditions, \textit{Commun. Math. Sci.}, 2022] and extends this to asymptotically large time intervals for small data using an involved bootstrapping argument.
\end{abstract}

{\bf 2020 MSC:} 35Q61, 35C07, 35L50.

{\bf Key words:} Maxwell's equations, Kerr nonlinearity, quasilinear, interface, envelope approximation, traveling pulse.

\section{Introduction}

Propagation of electromagnetic wave-packets at interfaces is of interest for applications in modern and future optical components. A typical example is surface plasmons (SPs). These electromagnetic waves propagate  at the interface of a conductor and a dielectric, are strongly localized perpendicular to the interface and are closely linked to nonlinear optical effects, see e.g. \cite{maier2007plasmonics}. From an engineering point of view their advantage is mainly in the higher level of localization compared to wave-packets in bulk media. If the involved media feature a nonlinear material response, new phenomena and thus new functionality of SPs are produced \cite{Kauranen2012}. Also interfaces of two dielectrics, e.g. photonic crystal waveguides \cite{JAMOIS20031,Kivshar:02,PRL-2008}, are interesting from the applied point of view.

We consider wave-packets at the interface of two generally inhomogeneous and Kerr nonlinear dielectrics. The problem is modeled by time dependent quasilinear Maxwell's equations in two spatial dimensions (assuming homogeneity of the material in the third direction). We study spatio-temporal wave-packets propagating in the direction tangential to the interface. They are broad in the propagation direction, have a small amplitude and are slowly modulated in time. In particular, we analyze their approximation via a slowly varying envelope. The equation governing the envelope dynamics is the one dimensional nonlinear Schr\"odinger equation (NLS).

The NLS is well known to approximate the dynamics of wave-packets in dispersive problems with a single carrier wave \cite{Kalyakin1989,KSM1992}. The formal derivation of the NLS for spatio-temporal wave-packets in quasilinear Maxwell's equations exists for a number of scenarios ranging from pulses in optical fibers \cite{agrawal2013} over photonic crystals \cite{BF05} to surface plasmons, e.g. \cite{Li:89, Davydova:15}. Spatial (time harmonic) surface plasmon wave-packets have been also formally approximated by the NLS (or more generally by the complex Ginzburg-Landau equation), e.g. in \cite{Davoyan:09,MS2010,CO11}, where one of the spatial variables plays the role of an evolution variable.

Formally derived asymptotic models for wave-packet envelopes can fail to produce a valid approximation of the original nonlinear problem, see e.g. \cite{Schneider-95,SSZ15}. Hence, a rigorous justification analysis with an error estimate must be performed. In the case of the semilinear wave equation with periodic coefficients this justification was carried out in \cite{BSTU06} for the one dimensional case and in \cite{DR20} in $d$ dimensions. The time dependent Maxwell's equations for nonlinear materials are quasilinear and the NLS approximation has been justified only in the case of fields leading to a scalar equation. Examples are \cite{LesSchnei2012} and \cite{SU03}. In \cite{LesSchnei2012} wave-packets in 2D photonic crystal waveguides are described by a quasilinear wave equation. In \cite{SU03} an approximation via a complex Ginzburg-Landau equation is proved for the quasilinear wave equation in one spatial dimension and with a time delayed material response (memory effect modeled by the coupling to an ODE system). In \cite{dull2018existence} the authors study a quadratic quasilinear dispersive equation allowing for resonances. These play no role in our analysis, which focuses on difficulties arising from the interface condition and the system character.

We work with vector valued Maxwell's equations for transverse magnetic (TM) polarized fields and reduce the problem to a system for the components $E_1, E_2$ and $H_3$. We restrict here to the instantaneous material response (as relevant for dielectrics) and avoid time delayed terms (relevant for metals). To our knowledge there are no directly applicable results on long time existence in full quasilinear Maxwell's equations with time delay on unbounded domains. For the instantaneous case we use local well-posedness results from \cite{schnaubelt2018local}. The working function space for each component is $\mathcal{G}^3(\R^2\times J):=\bigcap_{j=0}^3 C^j(\overline{J},\mathcal{H}^{3-j}(\R^2))$, where $\mathcal{H}^{s}(\R^2)$ consists of functions which are $H^s$ on each half-space defined by the interface and $J$ is a time interval. This high regularity is enforced by the quasilinear term. The two half-spaces are coupled by interface conditions out of which only the continuity of $E_2$ and $H_3$ needs to be enforced in the time evolution. As the approximation result needs to be proved on asymptotically large time intervals, we use a sophisticated bootstrap argument to extend the local existence to such asymptotically long time scales for small data. Here one estimates space-time differentiated solutions of the problem. If no normal derivatives occur, we can apply higher-order energy inequalities. Otherwise, the boundary conditions are violated and we have to use the equation itself and the divergence condition iteratively to bound the differentiated solutions in Gronwall arguments, see Section~\ref{S:bootstrap}.

Maxwell's equations in the whole space in the absence of free currents read
\begin{equation}
	\left\{
	\begin{aligned}
		\partial_t \bs{\mathcal{D}} &= \nabla \times \bs{\mathcal{H}}, \\
		\upmu_0 \partial_t \bs{\mathcal{H}} &= -\nabla \times \bs{\mathcal{E}}, \\
		\nabla \cdot \bs{\mathcal{D}} &= \varrho_0, \\
		\nabla \cdot \bs{\mathcal{H}} &= 0, 
	\end{aligned} \right. 
	\label{E:Maxw}
\end{equation}
for $\bs{x} \in \R^3$ and $t>0$, where $\bs{\mathcal{H}} = \bs{\mathcal{H}}\left(\bs{x},t\right)$ is the 
magnetic field, $\bs{\mathcal{E}} = \bs{\mathcal{E}}\left(\bs{x},t\right)$ is the electric field, $\varrho_0 = \varrho_0(\bs{x})$ is the volume charge density and
$\upmu_0$ is the permeability of free space, see e.g. \cite{feynman1979feynman}. We consider an electric displacement field $\bs{\mathcal{D}} = \bs{\mathcal{D}}(\bs{\mathcal{E}})$ 
given by the instantaneous material law
\begin{equation}
	\bs{\mathcal{D}}\left(\bs{x},t\right) = \upepsilon_0 \bs{\mathcal{E}}\left(\bs{x},t\right) + \bs{\mathcal{P}}\left(\bs{x},t\right).
	\label{E:D}
\end{equation}
Here $\upepsilon_0$ is the permittivity of free space and $\bs{\mathcal{P}} = \bs{\mathcal{P}}(\bs{\mathcal{E}})$ is the electric polarization modeling 
an $x_1$-dependent Kerr nonlinear material, i.e.,
\begin{equation}
	\bs{\mathcal{P}}\left(\bs{x},t\right) = \upepsilon_0 \left( \chi_1\left(x_1\right) \bs{\mathcal{E}}\left(\bs{x},t\right) + \chi_3\left(x_1\right) \left(\bs{\mathcal{E}}\left(\bs{x},t\right) \cdot \bs{\mathcal{E}}\left(\bs{x},t\right)\right)\bs{\mathcal{E}}\left(\bs{x},t\right) \right) 
	\label{E:P}
\end{equation}
with the linear and cubic susceptibilities $\chi_1, \chi_3:\R \to \R$, see e.g. \cite{boyd2020nonlinear}. For simplicity, the susceptibilities are scalar quantities, meaning that the material is isotropic.

In order to describe an interface, we allow $\chi_1$ and $\chi_3$ to have a jump at $x_1=0$ and denote
\begin{align*} 
	\chi_1\left(x_1\right) = \begin{cases}
		\chi_1^-(x_1), & x_1 < 0, \\
		\chi_1^+(x_1), & x_1 > 0,
	\end{cases} \qquad
	\chi_3\left(x_1\right) = \begin{cases} 
		\chi_3^-(x_1), & x_1 < 0,\\
		\chi_3^+(x_1), & x_1 > 0,
	\end{cases}
\end{align*}
for $\chi_1^\pm, \chi_3^\pm: \R_\pm := \left\{x_1 \in \R \mv \pm x_1 > 0 \right\} \to \R$. We also define
$$\epsilon_{1}:= \upepsilon_0(1+\chi_1),\ \ \epsilon_3:= \upepsilon_0\chi_3\quad\text{ and } \quad 
\epsilon_{1}^\pm:= \upepsilon_0\left(1+\chi_1^\pm\right), \ \ \epsilon_{3}^\pm:= \upepsilon_0\chi_3^\pm.$$

We investigate a two-dimensional setting with all fields independent of $x_3$, i.e.,
$$\left(\bs{\mathcal{D}}, \bs{\mathcal{E}}, \bs{\mathcal{H}}, \bs{\mathcal{P}} \right) = \left( \bs{\mathcal{D}}, \bs{\mathcal{E}}, \bs{\mathcal{H}}, \bs{\mathcal{P}} \right)\left(x_1,x_2,t\right).$$
Hence, the problem can be reduced to $\R^2$. From now on the variable $\bs{x}$ lies in $\R^2$. The two resulting half-spaces are denoted by $\R^2_-:=\!\left\{\bs{x}\in \R^2 \mv x_1<0 \right\}$ and $\R^2_+:=\!\left\{\bs{x}\in \R^2 \mv x_1>0 \right\}$ and the interface is $\Gamma:=\left\{\bs{x}\in \R^2 \mv x_1=0 \right\}$.
The aim of this paper is to describe the propagation of wave-packets localized near the interface $\Gamma$ and propagating in the $x_2$-direction, see Figure \ref{F:pulse-schem}.
\begin{figure}[t]
	\centering
	\begin{center}
		\includegraphics[width=0.8\linewidth]{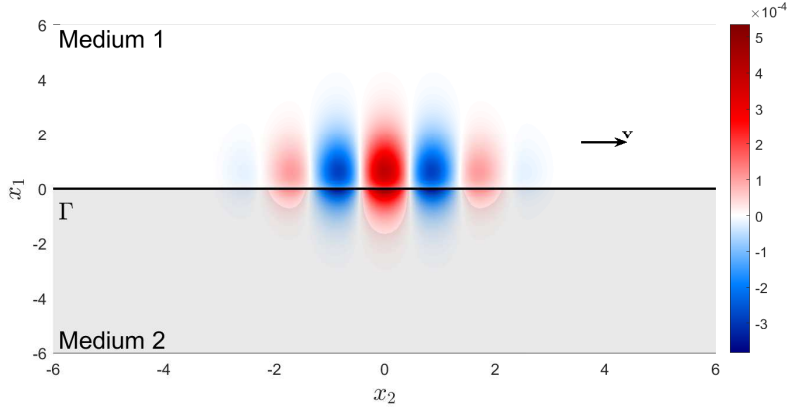}
	\end{center}
	\caption{Schematic of a pulse propagating in the direction $\textbf{v}=(0,1)^\top$, i.e., along the interface.}
	\label{F:pulse-schem}
\end{figure}

We also introduce the (time independent) surface charge density $\varrho_\Gamma:\Gamma\to \R$. Using Maxwell's equations in integral form, one can formally derive the jump conditions for solutions
\begin{equation}
	\begin{aligned}
		\llbracket \mathcal{D}_1 \rrbracket\left(\bs{x},t\right) = \varrho_\Gamma(\bs{x}),\ \ \llbracket \mathcal{E}_2 \rrbracket\left(\bs{x},t\right) = \llbracket \mathcal{E}_3 \rrbracket\left(\bs{x},t\right) &= 0, & \forall \bs{x} &\in \Gamma, t\geq 0, \\ \llbracket \mathcal{H}_1 \rrbracket\left(\bs{x},t\right) = \llbracket \mathcal{H}_2 \rrbracket\left(\bs{x},t\right) = \llbracket \mathcal{H}_3 \rrbracket\left(\bs{x},t\right) &= 0, & \forall \bs{x} &\in \Gamma, t \geq 0,
	\end{aligned}
	\label{E:IFC}
\end{equation}
in the absence of surface currents, see also Section I.4.2.4 of \cite{DL1990}. Here for $f:\R^2 \to \R$ continuous on $\overline{\R^2_-}$ and $\overline{\R^2_+}$ and for each $\bs{x} \in \Gamma$ we define
$$\llbracket f \rrbracket\left(\bs{x}\right):=\lim_{\bs{y}\to\bs{x}, \bs{y}\in \R^2_+}f(\bs{y})-\lim_{\bs{y}\to\bs{x}, \bs{y}\in \R^2_-}f(\bs{y}).$$
For $\varrho_\Gamma=0$ conditions \eqref{E:IFC} hold in trace sense for any $\bs{\mathcal{E}}(\cdot,t) \in \Hcurl(\R^2),$ $\bs{\mathcal{D}}(\cdot,t) \in \Hdiv(\R^2)$ and $\bs{\mathcal{H}}(\cdot,t) \in H^1(\R^2)^3$, see e.g. an appendix in \cite{BDPW21}.

We study specific solutions satisfying the reduction
\beq\label{E:reduce}
\cE_3\equiv \cH_1 \equiv \cH_2 \equiv 0.
\eeq 
This is a TM reduction since the fields propagate in the $(x_1,x_2)-$plane. One of the motivations for studying the TM case is that if $\epsilon_1^\pm$ is constant, all eigenfunctions of the linear eigenvalue problem for time harmonic fields (namely \eqref{E:ev-prob}) have to satisfy \eqref{E:reduce}, see e.g. \cite{BDPW21}. We study this reduced type of solutions also in the nonlinear case with non-constant $\epsilon_1^\pm$. Hence, we set
\beq\label{E:U}
\bs{U}:=(\cE_1,\cE_2,\cH_3)^\top
\eeq
and further define 
\begin{equation*}
	\bs{U}^\pm :=\bs{U}|_{\bs{x}\in \R^2_\pm}.
\end{equation*}
Throughout this paper we will always use $f^\pm$ to indicate the restriction of a function $f: \R^n \rightarrow \R$ to $\R^n_\pm$ with $n=1,2$.
We also write $\bs{U}_E := \left(U_1,U_2,0\right)^\top$ to denote the part of $\bs{U}$ corresponding to the electric field.
With \eqref{E:U} the first two equations in \eqref{E:Maxw} reduce to a system of three instead of six scalar equations. 
Note that the problem is indeed compatible with this reduction since the form of the nonlinearity implies that $\cD_j\equiv 0$ if $\cE_j\equiv 0$. 

Regarding the interface conditions, note that for time-independent surface charges we have $\llbracket \cD_1 \rrbracket (\bs{x},t) = \varrho_\Gamma(\bs{x})$ for all $t>0$ if $\llbracket \cD_1 \rrbracket (\bs{x},0) = \varrho_\Gamma(\bs{x})$.
This can be derived from the first component of the first equation in \eqref{E:Maxw}. Indeed, we get $\pa_t \llbracket \cD_1 \rrbracket (\bs{x},t) = \pa_{x_2}\llbracket \cH_3 \rrbracket (\bs{x},t)=0$ for all $\bs{x}\in \Gamma$ and $t > 0$. Also the divergence condition $\nabla\cdot \bs{\mathcal{D}}= \varrho_0$ needs to be checked only at $t=0$ as follows from the first equation in \eqref{E:Maxw}. For our specific solutions $\nabla \cdot \bs{\mathcal{H}} = 0$ is always satisfied, since the only non-trivial component $\mathcal{H}_3$ is independent of $x_3$. Therefore, the equations $\llbracket \cD_1 \rrbracket = \varrho_\Gamma$, $\nabla\cdot \bs{\mathcal{D}}= \varrho_0$, and $\nabla \cdot \bs{\mathcal{H}} = 0$ play no role in our analysis. Only the fact that $\nabla\cdot \bs{\mathcal{D}}$ equals a time independent quantity, sufficiently smooth in each half space, is used in the bootstrapping argument in Section~\ref{S:pf-main}.

Let $T^* >0$. From now on we study the initial value problem on the interval $(0,T^*)$ with initial data $\bs{U}^{(0)}: \R^2 \rightarrow \R^3$. 
With the above reduction the Maxwell problem \eqref{E:Maxw}, \eqref{E:D}, \eqref{E:P}, and \eqref{E:IFC} becomes
\begin{equation}
	\begin{pmatrix}
		\epsilon_1^\pm &0 & 0 \\ 0 & \epsilon_1^\pm & 0 \\ 0 & 0 &\upmu_0
	\end{pmatrix}
	\pa_t \bs{U}^\pm + \epsilon_3^\pm \pa_t 	
	\begin{pmatrix}
		\left(U_1^{\pm^2}+U_2^{\pm^2}\right)U_1^{\pm}\\
		\left(U_1^{\pm^2}+U_2^{\pm^2}\right)U_2^{\pm}\\
		0
	\end{pmatrix}
	+
	\begin{pmatrix}
		-\pa_{x_2}U_3^\pm\\
		\pa_{x_1}U_3^\pm\\
		\pa_{x_1}U_2^\pm- \pa_{x_2}U_1^\pm
	\end{pmatrix}=\bs{0}
	\label{E:Maxw-red}
\end{equation}
on $ \R^2_\pm \times (0,T^*)$ with 
\begin{equation} \label{E:IC}
	\bs{U}^\pm(\cdot,0)=\bs{U}^{(0),\pm} \qquad \text{on \ } \R^2_\pm,
\end{equation}
and the interface conditions
\beq
\llbracket U_2 \rrbracket=\llbracket U_3 \rrbracket=0 \quad \text{on } \Gamma \times [0,T^*).\label{E:IFC-red-E2H3}
\eeq
System \eqref{E:Maxw-red}, \eqref{E:IC}, and \eqref{E:IFC-red-E2H3} is the problem treated by our approximation result.

If, in addition, the sought solutions are to fit a prescribed volume charge density $\varrho_0$ and a prescribed surface charge density $\varrho_\Gamma$, then the initial condition $\bs{U}_E^{(0)}$ must be chosen such that the divergence condition
\begin{equation} \label{E:div-red}
	\begin{aligned}
		\pa_{x_1}\mathcal{D}_1\left(\bs{U}_E^{(0),\pm}\right) + \pa_{x_2}\mathcal{D}_2 \left(\bs{U}_E^{(0),\pm}\right) &= \pa_{x_1}\left(\epsilon_1^\pm U_1^{(0),\pm}+ \epsilon_{3}^\pm(U_1^{(0),\pm^2}+U_2^{(0),\pm^2})U_1^{(0),\pm}\right) \\
		&\quad + \pa_{x_2}\left(\epsilon_1^\pm U_2^{(0),\pm} + \epsilon_{3}^\pm(U_1^{(0),\pm^2}+U_2^{(0),\pm^2})U_2^{(0),\pm}\right)\\
		&= \varrho_0 \qquad \text{on \ } \R^2_\pm
	\end{aligned}
\end{equation}
and the interface condition
\begin{align}
	\left\llbracket \mathcal{D}_1\left(\bs{U}_E^{(0)} \right)\right\rrbracket = \left\llbracket \epsilon_1 U_1^{(0)}+ \epsilon_3\left(U_1^{(0)^2}+U_2^{(0)^2}\right)U_1^{(0)} \right\rrbracket&= \varrho_\Gamma\quad \text{on } \Gamma\label{E:IFC-red-D1}
\end{align}
are satisfied.

We study wave-packets based on the carrier wave
$$\bs{m}\left(x_1\right) \e^{\ie\left(k_0 x_2 - \nu_0 t \right)}, \qquad (\bs{x},t) \in (\R^2 \setminus \Gamma ) \times [0,\infty),$$
which solves the linear Maxwell equations corresponding to \eqref{E:Maxw-red}, i.e., with $\epsilon_3=0$, and with $\varrho_0 \equiv 0$ and $\varrho_\Gamma \equiv 0$. Here $k_0 \in \R$ is a fixed wave-number and $\bs{m}\left(x_1\right)$ is a localized profile (an eigenfunction) of the resulting eigenvalue problem in $x_1$, and $\nu_0 \in \R \setminus \{0\}$ is the corresponding eigenvalue, see Section \ref{S:lin-problem} for details. We use the classical formal asymptotic ansatz of a wave-packet
\begin{equation}
	\Uans\left(\bs{x}, t\right) =
	\begin{pmatrix}
		\mathcal{E}_{\mathrm{ans},1}\left(\bs{x}, t\right) \\ \mathcal{E}_{\mathrm{ans},2}\left(\bs{x}, t\right) \\ \mathcal{H}_{\mathrm{ans},3}\left(\bs{x}, t\right) \end{pmatrix} := \varepsilon A\left(\varepsilon\left(x_2 - \nu_1 t\right),\varepsilon^2t\right) \bs{m}\left(x_1\right) \e^{\ie\left(k_0 x_2 - \nu_0 t \right)} + \cc
	\label{E:Uans}
\end{equation}
with the complex envelope $A: \R \times [0,\infty) \to \C$, a small parameter $0 < \eps \ll 1$, the group velocity $\nu_1\in \R$ at the wave-number $k_0$ as defined in \eqref{E:nus}, and $\cc$ denoting the complex conjugate of the previous term. The envelope travels with the velocity $\nu_1$, depends slowly on the moving frame variable $x_2-\nu_1 t$ and is modulated in time on an even slower scale. 
If $A(\cdot, \eps^2 t)$ is localized, then $\Uans$ describes a wave-packet localized in both $x_1$ and $x_2$ and propagating in the $x_2$ direction. As explained in Section \ref{S:asymp}, the ansatz (after a suitable correction via higher-order terms) produces a small residual in the Maxwell system \eqref{E:Maxw-red} only if $A$ satisfies a nonlinear Schr\"odinger equation, see \eqref{E:NLS}. This NLS possesses localized solutions, see e.g. \cite{sulem1999nonlinear}, and is an effective macroscopic description of the wave-packet dynamics. 

We make the assumptions
\begin{equation} \label{E:ass-eps1}
	\epsilon_1^\pm \in C^3(\R_\pm) \cap W^{3,\infty}(\R_\pm), \quad \epsilon_1^\pm\geq \epsilon_{1,m}^\pm\in (0,\infty), \tag{A1} 
	\eeq
	\beq\label{E:ass-eps1-conv}
	\epsilon_1^\pm(x_1) \to \epsilon_1^{\pm\infty} \in \left[\epsilon_{1,m}^\pm, \infty\right) \text{ \ as } x_1 \to\pm\infty, \tag{A2}
	\eeq
	\beq\label{E:ass-eps3}
	\epsilon_3^\pm \in C^3(\R_\pm) \cap W^{3,\infty}(\R_\pm), \quad \epsilon_{3,m}^\pm \leq \epsilon_3^\pm \leq\epsilon_{3,M}^\pm \text{ \ with } \epsilon_{3,m}^\pm, \epsilon_{3,M}^\pm\in \R, \tag{A3} 
 	\eeq
\beq\label{E:ass-eps3-conv}
	\epsilon_3^\pm(x_1) \to \epsilon_3^{\pm\infty} \in \left[\epsilon_{3,m}^\pm, \epsilon_{3,M}^\pm\right] \text{ \ as } x_1 \to\pm\infty. \tag{A4}
\end{equation}

Our main result shows that the asymptotic wave-packet ansatz \eqref{E:Uans} is close to a true solution on a time interval of length $\mathcal{O}(\eps^{-2})$. This is the natural time scale for our approximation problem since $A$ depends on $\varepsilon^2 t$. Hence, changes in the envelope are observed only on this long time scale.
\begin{them}[Approximation Theorem]
	Assume \eqref{E:ass-eps1}, \eqref{E:ass-eps1-conv}, \eqref{E:ass-eps3}, \eqref{E:ass-eps3-conv}, and the conditions \eqref{E:ass-eigvexist}, \eqref{E:ass-esssp}, and \eqref{E:nres-cond}
	stated in Section~\ref{S:lin-problem}  and let $A \in \bigcap_{k=0}^4 C^{4-k}([0,T_0],$ $H^{3+k}(\R))$ be a solution of the effective nonlinear Schr\"odinger equation \eqref{E:NLS} for some $T_0 > 0$.  Assume that the initial value $\bs{U}^{(0)}:=\bs{U}(\cdot,0)\in\mathcal{H}^3(\R^2)^3$ satisfies the nonlinear compatibility conditions of order $3$, see Definition \ref{D:compat}. There exist constants $\eps_0 >0$ small enough and $C>0$ such that if $\eps \in (0,\eps_0)$ and if  $\bs{U}^{(0)}$ fulfills
	\begin{equation}
		\norm{\bs{U}^{(0)}-\Uans(\cdot,0)}_{\cH^3(\R^2)^3} \leq c \eps^{\frac{3}{2}},
		\label{E:difference_initial_values}
	\end{equation}
	with $c > 0$, then there exists a solution $\bs{U} \in \cG^3(\R^2\times (0,T_0\eps^{-2}))^3$  
	of \eqref{E:Maxw-red}, \eqref{E:IC}  and \eqref{E:IFC-red-E2H3} 
	such that
	\begin{equation}\label{E:main-est}
		\norm{\bs{U} - \Uans}_{\cG^3(\R^2\times (0,T_0\eps^{-2}))^3} \leq C \eps^{\frac{3}{2} - \delta}
	\end{equation}
	for all $\delta >0$. (The space $\cG^3$ is introduced below.) If, in addition, $\bs{U}_E^{(0)}$ satisfies  \eqref{E:div-red} and \eqref{E:IFC-red-D1}, then we have $\nabla \cdot \bs{\mathcal{D}}(\bs{U}_E) = \varrho_0$ on $(\R^2 \setminus \Gamma) \times (0,T_0 \varepsilon^{-2})$ and $\llbracket \mathcal{D}_1(\bs{U}_E)\rrbracket = \varrho_\Gamma$ on $\Gamma \times (0,T_0 \varepsilon^{-2})$.
	\label{T:main}
\end{them} 
\begin{remark}
	1. The existence of initial data $\bs{U}^{(0)}$ which satisfy \eqref{E:difference_initial_values} and the nonlinear compatibility conditions of order 3 is an open problem. Similarly, the existence of initial data $\bs{U}^{(0)}$ which satisfy \eqref{E:difference_initial_values}	 as well as \eqref{E:div-red} and \eqref{E:IFC-red-D1} for given $\varrho_0$ and $\varrho_\Gamma$ is an open problem. 
	
	For the case $\varrho_0=0$ and $\varrho_\Gamma=0$ this problem was considered in \cite{dohnal2022quasilinear}, where 
	initial data were found in the form $\bs{U}^{(0)}= \Uans(\cdot,0) + \nabla \phi$ with a correction function $\phi$, such that \eqref{E:difference_initial_values} holds with an exponent $a < 1$ instead of $\frac{3}{2}$. Note that our ansatz $\Uans$ naturally fits the choice $\varrho_0=0$ and $\varrho_\Gamma=0$ because $\epsilon_{1}\bs{m}\e^{\ie k_0x_2}$ is divergence free on $\R^2_\pm$ and $\epsilon_{1}m_1$ is continuous at $x_1=0$, see Remark~\ref{rem:lin-div}. As a result one can easily show that $\|\nabla\cdot \bs{\cD}(\bs{U}_{\mathrm{ans},E})\|_{L^2(\R^2)} \leq c\eps^{3/2}$ and $\left\llbracket \mathcal{D}_1\left(\bs{U}_{\mathrm{ans},E}(\cdot,0) \right)\right\rrbracket\leq c\eps^3$ for any bounded continuous $A$, where $\bs{U}_{\mathrm{ans},E} := \left(U_{\mathrm{ans},1},U_{\mathrm{ans},2},0\right)^\top$.
	
	2. Due to their high regularity the components $\bs{\mathcal{E}}:=(U_1,U_2,0)^\top$ and $\bs{\mathcal{H}}:=(0,0,U_3)^\top$
	of the solution $\bs{U}$ of Theorem \ref{T:main} satisfy \eqref{E:Maxw}, \eqref{E:D}, \eqref{E:P}, and \eqref{E:IFC} on $(\R^2 \setminus \Gamma) \times (0,T_0 \varepsilon^{-2})$ in the classical sense. 
	
	3. In the case $\varrho_\Gamma=0$ the regularity of $\bs{U}$ produced by Theorem \ref{T:main} guarantees that we have $\bs{\mathcal{E}} \in \Hcurl(\R^2),$ $\bs{\mathcal{D}} \in \Hdiv(\R^2)$ and $\bs{\mathcal{H}} \in H^1(\R^2)$ at each point in time. This is because functions $\bs{f}$ with $\bs{f}|_{\R^2_\pm}\in \Hcurl(\R^2_\pm)$ and with the tangential trace being continuous across the interface, are in $\Hcurl(\R^2)$. An analogous statement holds for $\Hdiv$ and the continuity of the normal trace, see \cite{BDPW21}.
\end{remark}	

For $m,n\in \N$, $ p\in [1,\infty]$ and an interval $J\subset \R$ we define
\begin{align*}
	\mathcal{W}^{m,p}(\R^n) &:= \left\{u \in L^p(\R^n) \mv u^- \in W^{m,p}(\R_-^n), \ u^+ \in W^{m,p}(\R_+^n)\right\}, \\
	\norm{u}_{\mathcal{W}^{m,p}(\R^n)} &:= \norm{u^-}_{W^{m,p}(\R_-^n)} + \norm{u^+}_{W^{m,p}(\R_+^n)},\\
	\mathcal{W}^{m,p}(\R^n \times J) &:= \left\{u \in L^p(\R^n \times J) \mv u^- \in W^{m,p}(\R_-^n \times J), \ u^+ \in W^{m,p}(\R_+^n \times J)\right\}, \\
	\norm{u}_{\mathcal{W}^{m,p}(\R^n\times J)} &:= \norm{u^-}_{W^{m,p}(\R_-^n \times J)} + \norm{u^+}_{W^{m,p}(\R_+^n \times J)}
\end{align*}
with the usual Lebesgue spaces $L^p$ and Sobolev spaces $W^{m,p}$. For $\mathcal{H}^m:=\mathcal{W}^{m,2}$ we also set
\begin{align*}
	\mathcal{G}^m(\R^n\times J) &:= \bigcap_{j=0}^m C^j(\overline{J},\mathcal{H}^{m-j}(\R^n)),\ 
	\norm{u}_{\mathcal{G}^m(\R^n\times J)} := \max_{0 \leq j \leq m} \norm{\partial_t^j u}_{L^\infty(J,\mathcal{H}^{m-j}(\R^n))}.
\end{align*}
The norm $\|\cdot\|_{\cG^3}$ in Theorem \ref{T:main} thus contains spatial and temporal derivatives of total degree three.

\begin{remark}
	We will often extend functions $f$ with $f^\pm \in L^p(\R^n_\pm)$ to a function in $L^p(\R^n)$. In general, a function $g \in \mathcal{H}^1(\R^n)$ does not belong to $H^1(\R^n)$ as the weak partial derivatives only exist in the half-spaces, e.g., $\partial_{x_1}g^+ \in L^2(\R_+^n)$ and $\partial_{x_1}g^- \in L^2(\R_-^n)$. Nevertheless, we will often write $\partial_{x_1}g \in L^2(\R^n)$ because the weak derivatives can be extended to a function defined on $\R^n$ by an arbitrary extension on $\R^n\setminus(\R_+^n \cup \R_-^n)$.
\end{remark}

The rest of the paper is organized as follows. In Section \ref{S:lin-problem} the linear spatial eigenvalue problem is studied in order to construct a carrier wave for the wave-packet. Section \ref{S:asymp} provides a formal derivation of the NLS as an amplitude equation. In Section \ref{S:res} we estimate the residual of the asymptotic approximation. In Section \ref{S:exist} we rewrite the reduced quasilinear Maxwell system \eqref{E:Maxw-red}, \eqref{E:IC}, \eqref{E:IFC-red-E2H3} in the form of a hyperbolic system and adapt the local existence results of \cite{schnaubelt2018local} to this problem. The proof of the main approximation result (Theorem \ref{T:main}) is provided in Sec. \ref{S:pf-main}. The proof is based on a bootstrapping argument which extends the local existence from \cite{schnaubelt2018local} to the existence on time intervals of length $\mathcal{O}(\eps^{-2})$ for initial data close to the (small) asymptotic ansatz. The bootstrapping simultaneously provides the error bound \eqref{E:main-est}. Finally in Appendix \ref{S:numerical_eigenvalue} we describe the numerical method for computing eigenvalues (and eigenfunctions) of the linear interface problem. In Appendix \ref{S:resiudal_order_4} the highest order residual terms are provided explicitly. Appendix \ref{A:products} contains estimates on products of functions in the used function spaces. 

\section{Linear time-harmonic eigenvalue problem}\label{S:lin-problem}
\subsection{Linear eigenvalue problem}
We first study the linear part of equation \eqref{E:Maxw-red} and \eqref{E:IFC-red-E2H3}, i.e., with $\epsilon_3=0$, assuming that $\epsilon_1$ satisfies \eqref{E:ass-eps1} and \eqref{E:ass-eps1-conv}.
Using the reduction \eqref{E:reduce}, \eqref{E:U} and the ansatz
\begin{equation*}
	\bs{U}(\bs{x},t)=\e^{\ie (k x_2 -\omega t)} \bs{w}(x_1) + \text{c.c.}, \quad (\bs{x},t) \in (\R^2 \setminus \Gamma ) \times [0,\infty),
\end{equation*}
where $k,\omega \in \R$ and $\bs{w}:\R\to \C^3$, one arrives at the eigenvalue problem 
\begin{equation}
	\begin{aligned}
		L\left(k\right)\bs{w}(x_1)+\omega \Lambda\bs{w}(x_1) &= \bs{0}, & x_1 &\in \R\setminus \{0\}
	\end{aligned} 
	\label{E:ev-prob}
\end{equation}
for the profile $\bs{w}$.
Here for each $k\in \R$ the operators $L(k): D(L(k)) \rightarrow L^2(\R)^3$ and $\Lambda: D(\Lambda) \rightarrow L^2(\R)^3$ are given by 
\begin{equation}
	L\left(k\right)\bs{w} := \begin{pmatrix}
		k w_3 \\
		\ie \partial_{x_1} w_3 \\
		k w_1 + \ie \partial_{x_1}w_2 
	\end{pmatrix}, \qquad	\Lambda\bs{w} := \begin{pmatrix}
		\epsilon_1(x_1) w_1 \\
		\epsilon_1(x_1) w_2 \\
		\upmu_0 w_3
	\end{pmatrix},
	\label{E:L}
\end{equation}
with the domains $D(\Lambda):=L^2(\R)^3$ and
$$
\begin{aligned}
	D(L(k)) := &\left\{ \bs{w}: \R \to \C^3 \mv w_{1} \in L^2(\R), w_{2}, w_{3}\in H^1(\R)\right\}.
\end{aligned}
$$ 
We call $\omega=\omega(k) \in \R$ an \emph{eigenvalue} of \eqref{E:ev-prob} if there exists a function $\bs{w}=\bs{w}(k)\in D(L(k)) \setminus \{\bs{0}\}$ such that \eqref{E:ev-prob} holds. For the eigenfunctions we choose the normalization
\beq\label{E:w-normaliz}
\int_\R \bs{w}^\top  \Lambda \overline{\bs{w}} \, \mathrm{d} x_1 = 1.
\eeq
Because the operator $L(k)$ is self-adjoint and $\Lambda$ is real and diagonal, all eigenvalues $\omega$ (in fact the whole spectrum) are indeed real. 
The interface conditions for $\bs{w}$ corresponding to \eqref{E:IFC-red-E2H3} are
\begin{equation}
	\llbracket w_2 \rrbracket_{\text{1D}} = \llbracket w_3 \rrbracket_{\text{1D}} = 0,
	\label{E:IFC-linear}
\end{equation}
where we define $\llbracket f \rrbracket_{\text{1D}}:=\lim_{x_1\to 0^+}f(x_1)-\lim_{x_1\to 0^-}f(x_1)$ for $f: \R \to \R$.  Solutions $\bs{w}$ of \eqref{E:ev-prob} fulfil these conditions, since $w_2,w_3 \in H^1(\R)$.

\begin{remark}\label{rem:lin-div}
	Let us, in addition, discuss the linear versions of divergence condition \eqref{E:div-red} and interface condition \eqref{E:IFC-red-D1}.
	Solutions $\bs{w}$ of \eqref{E:ev-prob} with $\omega\neq0$ satisfy $\llbracket \epsilon_1 w_1 \rrbracket_{\text{1D}}=0$ since $w_3 \in H^1(\R)$ and so $\epsilon_1 w_1$ is continuous because of $kw_3 + \omega\epsilon_1 w_1 = 0$, i.e., the first line in \eqref{E:ev-prob}.
	The (linear) divergence condition with $\varrho_0 = 0$, i.e.,
	$\partial_{x_1} \mathcal{D}_1(\bs{U}_E) + \partial_{x_2} \mathcal{D}_2(\bs{U}_E) = 0$ on $(\R^2 \setminus \Gamma) \times [0,\infty)$ 
	with $\epsilon_3=0$  is also automatically satisfied. Indeed, we have (for $\epsilon_{3}=0$)
	\begin{equation*}
		\partial_{x_1} \mathcal{D}_1(\bs{U}_E) + \partial_{x_2} \mathcal{D}_2(\bs{U}_E) = \left( \partial_{x_1} \left(\epsilon_1 w_1 \right) + \ie k \epsilon_1 w_2 \right) \e^{\ie (k x_2 - \omega t)} + \cc
	\end{equation*}
	and due to \eqref{E:ev-prob}
	\begin{equation}
		\partial_{x_1} \left(\epsilon_1 w_1 \right) + \ie k \epsilon_1 w_2 
		=-\frac{k}{\omega} \partial_{x_1} w_3 -\ie \frac{k}{\omega} \epsilon_1 \left(\frac{\ie \partial_{x_1} w_3}{\epsilon_1} \right) = 0.
		\label{E:linear_divergence}
	\end{equation}
\end{remark}
\begin{remark}
	Note that the second assumption in \eqref{E:ass-eps1} eliminates the pathological case where each $\omega \in \C$ 
	is an eigenvalue of infinite multiplicity, which is caused by the fact that gradient fields belong in the kernel of the curl operator. Indeed, if e.g.\ $\epsilon_1^+=0$, then $(\pa_{x_1}f, \ie k f, 0)^\top$ is an eigenfunction of \eqref{E:ev-prob} for any $f\in C^\infty_c(\R_+)$. Here, the electric field part $(\pa_{x_1}f, \ie k f)^\top$ corresponds to a gradient field (in the selected ansatz).
\end{remark}
For the construction of the wave-packet we need that near $k = k_0$ there is a unique smooth eigenvalue curve $k \mapsto \omega(k)$ and we set
\beq\label{E:nus}
\nu_0:=\omega(k_0), \quad \nu_1:=\pa_k \omega(k_0), \quad \nu_2:=\pa_k^2 \omega(k_0).
\eeq
This eigenvalue curve defines the so-called dispersion relation. In addition, let us assume that the eigenvalue 
$\nu_0$ is simple and denote the (normalized) eigenfunction by
\[
\bs{m}:=\bs{w}(k_0).
\]
We impose the following assumptions for Theorem \ref{T:main}.\beq\label{E:ass-eigvexist}
\nu_0=\omega(k_0) \text{ is a simple eigenvalue of \eqref{E:ev-prob} isolated from all other eigenvalues at $k=k_0$}. \tag{A5}
\eeq
In order to control the essential spectrum corresponding to \eqref{E:ev-prob}, we also require
\beq\label{E:ass-esssp}
\omega(k_0) \epsilon_1 \neq 0, \quad \omega(3k_0) \epsilon_1 \neq 0 \quad \text{and}
\quad k_0^2 > \omega(k_0)^2 \upmu_0 \epsilon_1^{\pm \infty}. \tag{A6}
\eeq 
In view of \eqref{E:ass-eps1} and \eqref{E:ass-eps1-conv}, the first two conditions in \eqref{E:ass-esssp} just say that $\omega(k_0)$ and $\omega(3k_0)$ do not vanish.
As noted in Corollary~\ref{cor:bsw}, the above assumptions also guarantee that the eigenvalue $\omega(k)$ and the
eigenfunction $\bs{w}(k)$ depend smoothly on $k$ near $k_0$, which is needed below.

Given a solution $\bs{w}(k)$, also $\widetilde{\bs{w}} := \left(\overline{w}_1, - \overline{w}_2, \overline{w}_3 \right)^\top$ solves \eqref{E:ev-prob}. 
We can thus choose the eigenfunction $\bs{w}$ with real valued $w_1,w_3$ and imaginary valued $w_2$, which we do throughout the rest of the paper. 
With this choice the normalization \eqref{E:w-normaliz} means that 
\beq\label{E:w-normaliz2}
\int_\R \left(\epsilon_1 \left( m_1^2 - m_2^2\right) + \upmu_0 m_3^2\right)\,\mathrm{d}x_1=1.
\eeq

For the proof of Theorem \ref{T:main}, i.e., the justification of \eqref{E:Uans} as an asymptotic approximation of a solution of the cubically nonlinear problem, it will be necessary to assume the non-resonance condition
\begin{equation}
	3 \nu_0 \neq \omega(3k_0), \text{ \ i.e., }	3 \nu_0 \text{ is not an eigenvalue of \eqref{E:ev-prob} at } k = 3k_0, \tag{A7}
	\label{E:nres-cond} 
\end{equation}
see \eqref{E:system_h} and the following arguments in Section \ref{S:asymp}.

Since $\epsilon_1$ depends on $x_1$ it is in general not possible to solve \eqref{E:ev-prob} explicitly. We therefore have to calculate solutions numerically and check if the Assumptions \eqref{E:ass-eigvexist}, \eqref{E:ass-esssp}, and \eqref{E:nres-cond} are satisfied. However, as explained above, the first two assumptions in \eqref{E:ass-esssp} describe the generic situation. Also assumption \eqref{E:nres-cond}  (being an inequality) holds generically. Moreover, note that $k\mapsto \omega(k)$ is typically nonlinear.
\begin{figure}[ht!]
	\centering
	\begin{subfigure}[b]{0.45\linewidth}
		\includegraphics[width=\linewidth]{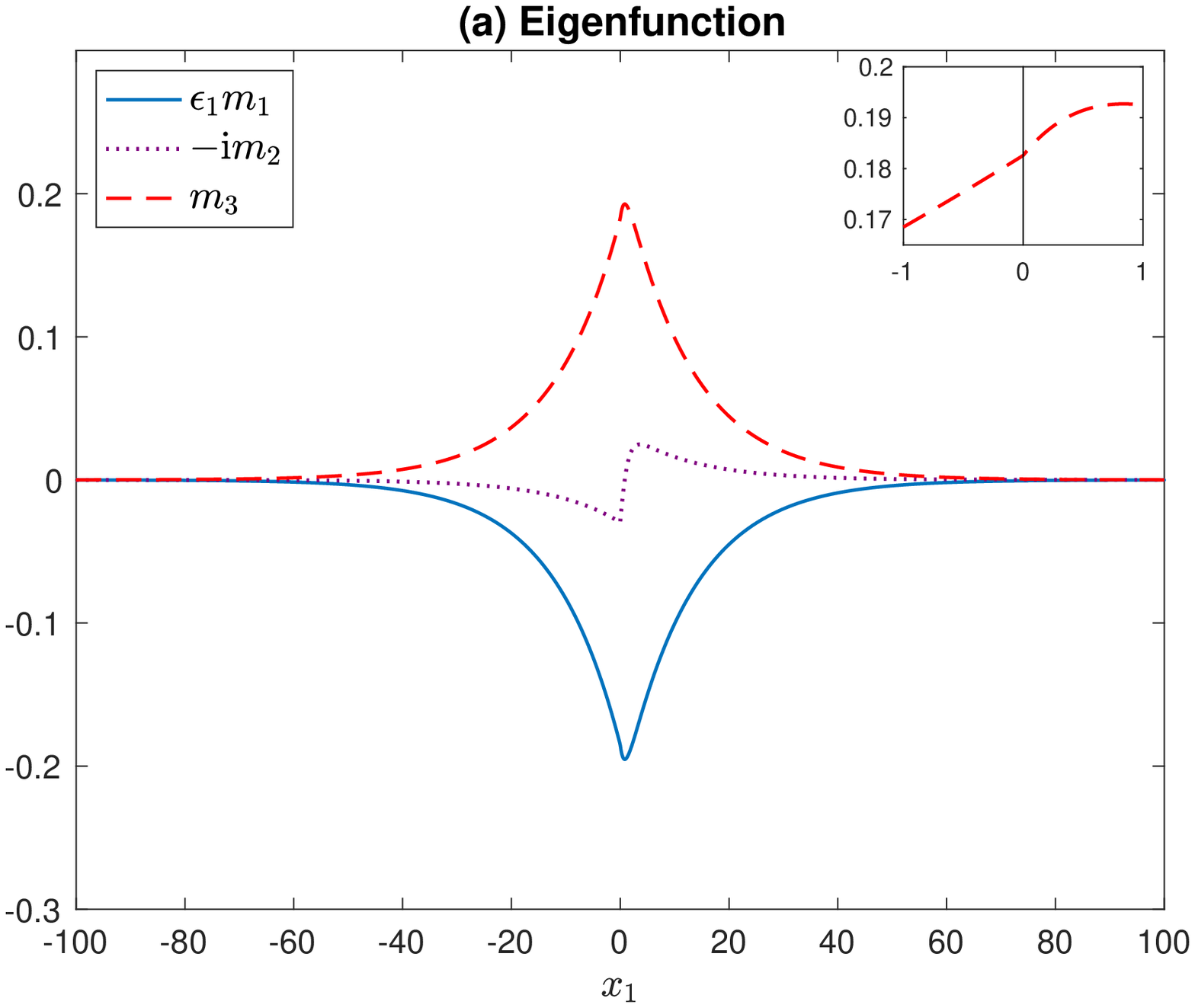}
	\end{subfigure}
	\begin{subfigure}[b]{0.45\linewidth}
		\includegraphics[width=\linewidth]{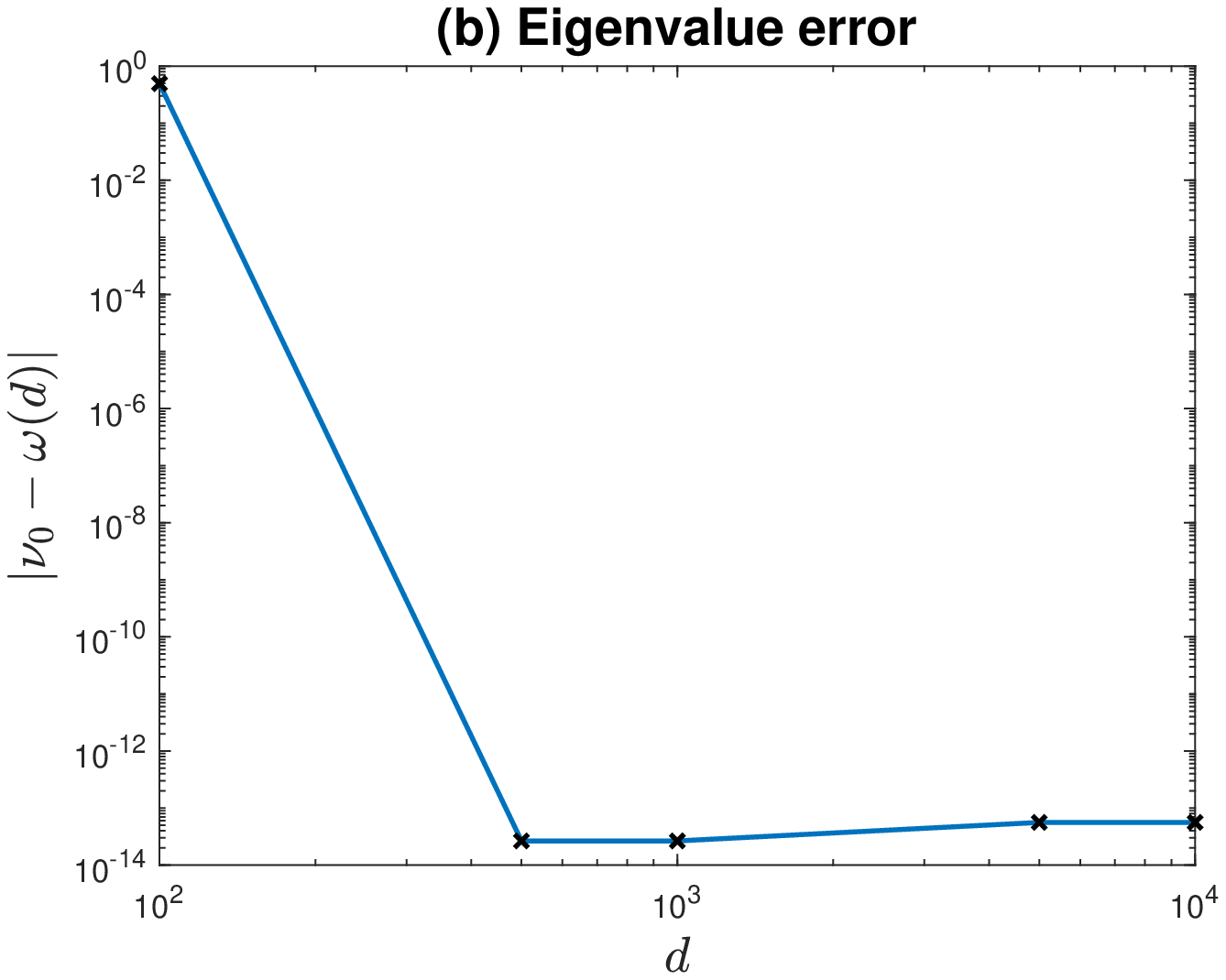}
	\end{subfigure}
	\caption{(for Example \ref{Ex:eval}) (a) The eigenfunction $\bs{m}$ of the linear problem \eqref{E:ev-prob} for $k_0 = 0.5$. (We plot $\epsilon_1 m_1$ to show that the linear interface conditions are satisfied.) The inset shows that the eigenfunction is not symmetric and not $C^1$. (b) Numerical convergence test for the eigenvalue $\omega = \nu_0 \approx 0.494$ of $T_{k_0,\omega}:=L(k_0)+ \omega \Lambda$ for $k_0 = 0.5$ in dependence on the computational box size $d$.}
	\label{F:isol-eval}
\end{figure}
\begin{bsp} \label{Ex:eval}
	For $\epsilon_1(x_1)=1\chi_{\R_-}+\left(1+{\rm e}^{-x_1}\right)\chi_{\R_+}$ and $\upmu_0 = 1$, we compute a numerical solution of \eqref{E:ev-prob} with the method described in Appendix \ref{S:numerical_eigenvalue}. 
	We study the generalized eigenvalue problem $L(k)\bs{w} = -\omega \Lambda \bs{w}$ for $k = k_0$ on the interval $[-d,d]$ and compute all eigenvalues in a neighborhood of $\omega = \nu_0$. For $k_0 = 0.5$, step size $h = 0.01$ and interval length $d = 5 \cdot 10^{4}$ we get the eigenvalue $\omega(k_0) = \nu_0 \approx 0.494$ and no other eigenvalue in a neighborhood of $\nu_0$.
	
	To check the effects of the boundary, we repeated the calculation for different intervals $[-d,d]$ and get the eigenvalue $\omega(d)$ closest to $\nu_0$ in dependence on $d$ . In Figure \ref{F:isol-eval} $\mathrm{(a)}$ we see the calculated eigenfunction $\bs{w}$ and Figure \ref{F:isol-eval} $\mathrm{(b)}$ shows that the error in the calculation of $\omega(d)$ converges to zero for increasing $d$. 
	
	Note that for this example we have $\epsilon_{1,m}^\pm = \epsilon_1^{\pm \infty} = 1$ and one can numerically calculate that the eigenvalue closest to $3\nu_0 \approx 1.481 $ is given by $\omega(3k_0) \approx 1.404$. Therefore, Assumptions \eqref{E:ass-eps1}, \eqref{E:ass-eps1-conv}, \eqref{E:ass-eigvexist}, \eqref{E:ass-esssp}, and \eqref{E:nres-cond} appear to be satisfied. 
\end{bsp}

\subsection{Solution of the inhomogeneous problem}
\label{section_solution_inhomogeneous_problem}

In Section \ref{S:asymp} we also have to solve the inhomogeneous version of the eigenvalue problem
\begin{equation}
	T_{k,\omega} \bs{v}:= \left(L(k) + \omega \Lambda\right) \bs{v} = \begin{pmatrix}
		\omega \epsilon_1 & 0 & k \\
		0 & \omega \epsilon_1	& \ie \partial_{x_1} \\
		k & \ie \partial_{x_1} & \omega \upmu_0 
	\end{pmatrix} \bs{v} = \bs{f},
	\label{E:general_inhomo_system}
\end{equation}
with $\bs{f} \in N(T_{k,\omega})^\perp$ and the kernel $$N(T_{k,\omega})\subset D(T_{k,\omega}):=\left\{\bs{v}\in L^2(\R)^3 \mv v_2,v_3\in H^1(\R)\right\}.$$

\begin{lem}\label{L:inhom-prob}
	Let $ \epsilon_1\in\cW^{1,\infty}(\R)$ satisfy \eqref{E:ass-eps1-conv} and let $k, \omega \in \R$ be such that $k^2 > \omega^2 \epsilon_1^{\pm \infty} \upmu_0$ and $\omega \epsilon_1 \neq 0$. Assume that we are in one of the cases
	\begin{enumerate}[label=\roman*)]
		\item $0$ is a simple eigenvalue of $T_{k,\omega}$ isolated from all other eigenvalues;
		\item $0$ is not an eigenvalue of $T_{k,\omega}$. 
	\end{enumerate}
	If $\bs{f}\in N(T_{k,\omega})^\perp \subset L^2(\R)^3$ ($\bs{f} \in L^2(\R)^3$ if $N(T_{k,\omega})=\{0\}$), then \eqref{E:general_inhomo_system} has a solution $\bs{v}\in D(T_{k,\omega})$.
\end{lem}

\begin{proof}
	Equation \eqref{E:general_inhomo_system} splits for $\omega \epsilon_1 \neq 0$ into the scalar equation 
	\begin{equation*}
		v_1 = \frac{1}{\omega \epsilon_1} \left(f_{1} - k v_3 \right)
	\end{equation*}
	and the reduced problem
	\begin{equation*}
		\widetilde{T}_{k,\omega} \widetilde{\bs{v}} = \widetilde{\bs{f}}
	\end{equation*}
	with
	\begin{equation*}
		\widetilde{T}_{k,\omega} := \begin{pmatrix}
			\omega \epsilon_1 & \ie \partial_{x_1} \\
			\ie \partial_{x_1} & \omega \upmu_0 - \frac{k^2}{\omega \epsilon_1 }
		\end{pmatrix}, \  D(\widetilde{T}_{k,\omega}) := H^1(\R)^2, \  
		\widetilde{\bs{v}} := \begin{pmatrix}
			v_2 \\ v_3 
		\end{pmatrix}, \ 
		\widetilde{\bs{f}} := \begin{pmatrix}
			f_{2} \\ f_{3} - \frac{k}{\omega \epsilon_1} f_{1} 
		\end{pmatrix}.
	\end{equation*}
	Note that 
	\begin{equation*}
		(v_2,v_3)^\top \in N(\widetilde{T}_{k,\omega}) \ \Longleftrightarrow \ \left(-\frac{k}{\omega \epsilon_1} v_3, v_2, v_3 \right)^\top \in N(T_{k,\omega})
	\end{equation*}
	and hence
	$$
	\widetilde{\bs{f}} \in N(\widetilde{T}_{k,\omega})^\perp \ \Longleftrightarrow \ \bs{f} \in N(T_{k,\omega})^\perp.
	$$
	We also obtain that $0 \in \sigma (\widetilde{T}_{k,\omega})$ if and only if $0 \in \sigma\left(T_{k,\omega} \right)$.
	
	Since $\widetilde{T}_{k,\omega}$ is self-adjoint, the result will follow from the closedness of the range of $\widetilde{T}_{k,\omega}$ and the closed range theorem, see e.g. \cite[Section 7]{yosida1980functional}.
	We check the closedness by showing that $\widetilde{T}_{k,\omega}$ is Fredholm. To this aim, we rewrite the problem 
	as the linear ordinary differential equation
	\begin{equation*}
		\partial_{x_1} \widetilde{\bs{v}} = A(x_1) \widetilde{\bs{v}} + \bs{g}
	\end{equation*}
	with 
	\begin{equation*}
		\bs{g} := - \ie \begin{pmatrix}
			f_{3} - \frac{k}{\omega \epsilon_1} f_{1} \\
			f_2
		\end{pmatrix},
	\end{equation*}
   \begin{equation*}
		A(x_1) := \begin{pmatrix}
			0 & \ie \left( \omega \upmu_0 - \frac{k^2}{\omega \epsilon_1(x_1)} \right) \\
			\ie\omega \epsilon_1(x_1) & 0 
		\end{pmatrix} =: \begin{cases}
			A_-(x_1), & x_1 < 0, \\
			A_+(x_1), & x_1 > 0.
		\end{cases}
	\end{equation*}
	Theorem 1.2 in \cite{ben1992dichotomy} says that $\widetilde{T}_{k,\omega}$ is Fredholm if and only if the ODEs
	\begin{align}
		\partial_{x_1} \widetilde{\bs{v}}^- &= A_-(x_1) \widetilde{\bs{v}}^-, \qquad x_1<0,
		\label{E:reduced_problem_ode_minus}\\
		\partial_{x_1} \widetilde{\bs{v}}^+ &= A_+(x_1) \widetilde{\bs{v}}^+, \qquad x_1>0,
		\label{E:reduced_problem_ode_plus}
	\end{align}
	have exponential dichotomies. We only show the dichotomy for \eqref{E:reduced_problem_ode_plus} as
	\eqref{E:reduced_problem_ode_minus} can be treated analogously. First, the problem
	\begin{equation*}
		\partial_{x_1} \bs{w} = A_{+\infty} \bs{w}
	\end{equation*}
	with the constant coefficient matrix
	\begin{equation*}
		A_{+ \infty}:= A(x_1 \rightarrow \infty) = \begin{pmatrix}
			0 & \ie \omega \upmu_0 - \frac{\ie k^2}{\omega \epsilon_1^{+ \infty}} \\
			\ie \omega \epsilon_1^{+ \infty} & 0
		\end{pmatrix}
	\end{equation*}
	has an exponential dichotomy since the eigenvalues 
	\begin{equation*}
		\lambda_{1,2} = \pm \sqrt{k^2 - \omega^2 \epsilon_1^{+\infty}\upmu_0}
	\end{equation*}
	of $A_{+ \infty}$ are real with different signs for $k^2 > \omega^2 \epsilon_1^{+\infty}\upmu_0$. Then
	Proposition 1 in Chapter 4 and the discussion starting on page 13 of \cite{coppel1971dichotomies} imply that also
	\begin{equation*}
		\partial_{x_1} \widetilde{\bs{v}}^+ = A_+(x_1) \widetilde{\bs{v}}^+ = \left(A_{+\infty} + \left(A_+(x_1) - A_{+ \infty}\right)\right) \widetilde{\bs{v}}^+
	\end{equation*}
	has an exponential dichotomy, because $A_+(x_1) - A_{+ \infty}$ tends to 0 as $x_1\to\infty$.
\end{proof}

Using the spectral information obtained above, we next show that the eigenvalues $\omega(k)$ and the corresponding eigenfunctions $w(k)$ are smooth in $k$.

\begin{kor} \label{cor:bsw}
	Let \eqref{E:ass-eps1}, \eqref{E:ass-eps1-conv}, \eqref{E:ass-eigvexist} and \eqref{E:ass-esssp} be true. Then for some $\delta > 0$ the eigenvalues and the corresponding eigenfunctions of problem \eqref{E:ev-prob} satisfy
	\begin{equation*}
		\omega \in C^\infty\left((k_0 - \delta, k_0 + \delta), \R \right) \text{ and } 
		\bs{w} \in C^\infty((k_0 - \delta, k_0 + \delta), L^2(\R) \times H^1(\R) \times H^1(\R)).
	\end{equation*}
\end{kor}
\begin{proof}
	We translate our problem into standard perturbation theory of spectra, as discussed in \cite{kato1995perturbation}.
	By the assumptions, $\nu_0$ is a simple eigenvalue of $-\Lambda^{-1} L(k_0)$ with eigenfunction $\bs{w}(k_0)$, 
	and there are no other eigenvalues nearby. As shown in the proof of Lemma~\ref{L:inhom-prob}, $\omega$ belongs to the 
	resolvent set of $-\Lambda^{-1} L(k_0)$ if and only if $\widetilde{T}_{k_0,\omega}$ is invertible. For 
	$\omega \approx \nu_0$ we can write $\widetilde{T}_{k_0,\omega}=\widetilde{T}_{k_0,\nu_0} + R$ with a 
	perturbation $R:L^2(\R)\to L^2(\R)$, whose norm is bounded by $c\,|\omega-\nu_0|$. In the proof of Lemma \ref{L:inhom-prob} we have seen that $\widetilde{T}_{k_0,\nu_0}$ is a Fredholm operator, hence $0\notin \sigma_{\rm ess}(\widetilde{T}_{k_0,\nu_0})$ and the same is true for 
	$\widetilde{T}_{k_0,\omega}$ if $\omega$ is close to $\nu_0$. If 0 was an eigenvalue of 
	$\widetilde{T}_{k_0,\omega}$, the number $\omega\neq \nu_0$ would be an eigenvalue of $-\Lambda^{-1} L(k_0)$ 
	which is impossible in a small enough neighborhood of $\nu_0$ by assumption \eqref{E:ass-eigvexist}. As a result, $0$ is contained in $\rho(\widetilde{T}_{k_0,\omega})$ and thus $\nu_0$ is an isolated simple 
	eigenvalue of $-\Lambda^{-1} L(k_0)$.
	
	For $k\approx k_0$, Theorem~1.8 in \S VII.1 of \cite{kato1995perturbation} now shows that $-\Lambda^{-1} L(k)$
	has a simple eigenvalue $\omega(k)$ smoothly depending on $k$. Also the projection $P(k)$ onto the eigenspace
	is smooth in $k$. Hence, the mapping $k \mapsto P(k)\bs{w}(k_0)$ is a smooth family of eigenfunctions of \eqref{E:ev-prob} if $k$ is close to $k_0$.
\end{proof}

The next lemma improves the regularity of solutions to \eqref{E:general_inhomo_system} if the right-hand side is smooth enough. 

\begin{lem}
	\label{lemma_regularity_of_eigenfunction}
	Assume \eqref{E:ass-eps1} and \eqref{E:ass-esssp}. Let $k, \omega \in \R$, $\bs{f} := \left(f_1,f_2,f_3 \right)^\top$ with $f_1 \in \mathcal{H}^3(\R)$ and $f_2,f_3 \in \mathcal{H}^2(\R)$. If $\bs{v} \in L^2(\R)^3$ is a solution of \eqref{E:general_inhomo_system}, then $\bs{v} \in \mathcal{H}^3(\R)^3$.
\end{lem}
\begin{proof}
	We start by showing that $\bs{v} \in \mathcal{H}^1(\R)$. From \eqref{E:general_inhomo_system} we know that
	\begin{equation}
		\left\{
		\begin{aligned}
			\ie \partial_{x_1} v_2 &= f_3 - k v_1 - \upmu_0 \omega v_3, \\
			\ie \partial_{x_1} v_3 &= f_2 - \epsilon_1 \omega v_2.
		\end{aligned}
		\right.
		\label{regularity_eigenvalue_problem_first_part}
	\end{equation}
	The right-hand sides in \eqref{regularity_eigenvalue_problem_first_part} belong to $L^2(\R)$ and therefore $v_2,v_3 \in \mathcal{H}^1(\R)$.
	The assumptions on $\epsilon_1$ imply that $\partial_{x_1}\left(\epsilon_1^{-1}\right) = -\epsilon_1^{-2} \partial_{x_1} \epsilon_1 \in L^\infty(\R_\pm)$. Now $v_1 \in \mathcal{H}^1(\R)$ is a direct consequence of 
	\begin{equation}
		v_1 = \frac{1}{\epsilon_1 \omega}\left(f_1 - k v_3 \right).
		\label{regularity_eigenvalue_problem_second_part}
	\end{equation}
	We can now iterate this process since $\epsilon_1^{-1} \in \mathcal{W}^{3,\infty}(\R)$. Equations \eqref{regularity_eigenvalue_problem_first_part} and \eqref{regularity_eigenvalue_problem_second_part} yield that $\bs{v} \in \mathcal{H}^2(\R)^3$ 
	if one knows that $\bs{v} \in \mathcal{H}^1(\R)^3$. This fact then implies that $\bs{v} \in \mathcal{H}^3(\R)^3$.
\end{proof}

\section{Envelope approximation of wave-packets; amplitude equation}\label{S:asymp}

The aim of this section is to make the residual
\begin{equation}
	\Res(\Uans) := \begin{pmatrix}
		\partial_t \mathcal{D}_1(\bs{U}_{\mathrm{ans},E}) - \partial_{x_2} \Uansthree \\
		\partial_t \mathcal{D}_2(\bs{U}_{\mathrm{ans},E}) + \partial_{x_1} \Uansthree \\
		- \partial_{x_2} \Uansone + \partial_{x_1} \Uanstwo + \upmu_0 \partial_t \Uansthree
	\end{pmatrix}
	\label{E:res}
\end{equation}
of \eqref{E:Uans} in the Maxwell problem \eqref{E:Maxw-red}
small enough for the subsequent justification of the asymptotics, i.e., for the proof of Theorem \ref{T:main}. As we will see, this requires an extension of the ansatz \eqref{E:Uans}.
In the propagation direction $x_2$ we mostly work in Fourier variables applying the Fourier transform 
$$\widehat{f}(k) = \mathcal{F}(f)(k) := (2\uppi)^{-1/2} \int_\R f(x) \e^{-\ie kx}\,\mathrm{d} x.$$
The corresponding inverse transform is given by 
$$\mathcal{F}^{-1}(f)(x) := (2\uppi)^{-1/2} \int_\R f(k) \e^{\ie kx}\,\mathrm{d} k.$$
Indeed, it is $\mathcal{F}\circ \mathcal{F}^{-1}=\mathcal{F}^{-1}\circ \mathcal{F}=\text{Id}:L^2(\R)\to L^2(\R)$ after the standard extension of the transforms from $L^1(\R)$ to $L^2(\R)$.

For the wave-packet $\Uans$ we compute
\begin{equation}
	\Uanshat\left(x_1,k, t\right) :=
	\begin{pmatrix}
		\widehat{\mathcal{E}}_{{\rm ans},1}\left(x_1,k, t\right) \\ \widehat{\mathcal{E}}_{{ \rm ans},2}\left(x_1,k, t\right) \\ \widehat{\mathcal{H}}_{{ \rm ans},3}\left(x_1,k, t\right) \end{pmatrix} = \widehat{A}\left(\frac{k-k_0}{\varepsilon},\varepsilon^2t \right) \bs{m}\left(x_1\right) \e^{-\ie\left( \nu_0 + \left(k-k_0\right)\nu_1\right)t} + \widehat{\cc},
	\label{E:uans_hat}
\end{equation}
where $\widehat{\cc}(\widehat{f}) = \widehat{\cc(f)}$. Maxwell's equations with the reduction \eqref{E:reduce} transform to
\begin{equation}
	\left\{
	\begin{aligned}
		\partial_t \widehat{\mathcal{D}}_1 - \ie k \widehat{\mathcal{H}}_3 &= 0, \\
		\partial_t \widehat{\mathcal{D}}_2 + \partial_{x_1}\widehat{\mathcal{H}}_3 &= 0, \\
		-\ie k \widehat{\mathcal{E}}_1 + \partial_{x_1} \widehat{\mathcal{E}}_2 + \upmu_0 \partial_t \widehat{\mathcal{H}}_3 &= 0,
	\end{aligned} \right.
	\label{E:Maxw-reduce-FT} 
\end{equation}
with
\begin{equation*}
	\begin{aligned}
		\widehat{\bs{\mathcal{D}}}(\bs{\mathcal{E}}) &= \epsilon_1 \widehat{\bs{\mathcal{E}}} + \epsilon_3 \left(\left(\bs{\mathcal{E}}\cdot \bs{\mathcal{E}}\right) \bs{\mathcal{E}}\right)^\land\\ 
	\end{aligned}.
\end{equation*}
In what follows we use the notations $E_1 := \e^{-\ie\left(\nu_0 + \left(k-k_0\right)\nu_1\right)t}$, $F_1 := \e^{\ie\left(k_0x_2 - \nu_0t \right)}$, $K := \eps^{-1}(k-k_0)$, $T := \varepsilon^2 t$ and $X_2 := \varepsilon\left(x_2 - \nu_1t\right)$,
and we will suppress the arguments of $\bs{m} = \bs{m}(x_1)$ and $\widehat{A} = \widehat{A}(K,T)$ and their derivatives if they are obvious. 

We start our formal asymptotic analysis by writing out the nonlinear term, where we employ the notation $\bs\cD=\bs\cD_{{\rm lin}}+\bs\cD_{{\rm nl}}$ with
$$
\bs{\cD}_{{\rm lin}}:= \epsilon_1 \bs{\mathcal{E}}, \quad 
\bs{\cD}_{{\rm nl}}:= \epsilon_3 (\bs{\mathcal{E}}\cdot \bs{\mathcal{E}})\bs{\mathcal{E}}.
$$
In the physical variables we get
\begin{align*}
	\partial_t \mathcal{D}_{{\rm nl},1}(\bs{U}_{\mathrm{ans},E}) = &-\varepsilon^3 3 \ie \epsilon_3 \nu_0 F_1^3 A^3 \left(m_1^3 + m_1 m_2^2 \right)\\
	&-\varepsilon^3 \ie \epsilon_3 \nu_0 F_1 |A|^2A \left(3m_1^3- m_1 m_2^2\right) + \mathcal{O}(\varepsilon^4) + \cc \quad (\varepsilon \rightarrow 0),
\end{align*}
using that $m_{1}$ is real and $m_2$ is imaginary. As one easily checks,
\begin{equation*}
	\begin{aligned}
		\int_\R \e^{\ie\left(k_0x_2 - \nu_0t\right)} |A\left(X_2,T\right)|^2A\left(X_2,T\right)\e^{-\ie k x_2} \,\mathrm{d}x_2 &= \tfrac{1}{\sqrt{2 \uppi}} \varepsilon^{-1}E_1 \left(\widehat{A} * \widehat{\overline{A}} * \widehat{A}\right)\left(K,T\right),\\
		\int_\R \e^{3\ie(k_0x_2-\nu_0t)}A^3(X_2,T)\e^{-\ie kx_2}\,\mathrm{d}x_2 &=\tfrac{1}{\sqrt{2 \uppi}} \varepsilon^{-1} E_3 \left(\widehat{A} * \widehat{A} * \widehat{A}\right)\left(\Ktil,T\right),
	\end{aligned}
\end{equation*}
with $\Ktil:=\frac{k-3k_0}{\eps}$, $E_3 := \e^{-\ie(3\nu_0+(k-3k_0)\nu_1)t}$, and the convolution
\begin{equation*}
	(f * g)(K) = \int_\R f(K-s)g(s)\,\mathrm{d}s.
\end{equation*}
Hence, as $\eps \to 0$ we have
\begin{equation}\label{E:Dnl-Uans}
\begin{aligned}
	\partial_t \widehat{\mathcal{D}}_{{\rm nl},1}(\bs{U}_{\mathrm{ans},E}) = &-\varepsilon^2 \tfrac{\ie}{\sqrt{2 \uppi}} \epsilon_3 \nu_0 E_1 \left(3m_1^3- m_2^2 m_1 \right)\left(\widehat{A} * \widehat{\overline{A}} * \widehat{A}\right)(K,T)\\
	&-\varepsilon^2 \tfrac{3\ie}{\sqrt{2 \uppi}}\epsilon_3 \nu_0 E_3 \left(m_1^3 + m_1 m_2^2\right) \left(\widehat{A} * \widehat{A} * \widehat{A}\right)(\Ktil,T) + \mathcal{O}(\varepsilon^3) + \widehat{\cc}.
\end{aligned}
\end{equation}
The second component $\pa_t\widehat{\mathcal{D}}_{{\rm nl},2}(\bs{U}_{\mathrm{ans},E})$ is obtained analogously and the third component $\widehat{\mathcal{D}}_{{\rm nl},3}(\bs{U}_{\mathrm{ans},E})$ obviously vanishes.

Below we use the Taylor expansion of $\omega(k)$, of the corresponding eigenfunction $\bs{w}(k)$, see Corollary~\ref{cor:bsw}, 
and of the operator $L(k)$. Recalling \eqref{E:nus} and $k = k_0 + \varepsilon K$, we obtain
\begin{align*}
	\omega\left(k\right) &= \omega\left(k_0 + \varepsilon K\right) = \nu_0 + \varepsilon K \nu_1 + \frac{1}{2}\varepsilon^2 K^2 \nu_2 + \mathcal{O}(\varepsilon^3),\\
	\bs{w}\left(k\right) &= \bs{w}\left(k_0 + \varepsilon K\right) = \bs{m} + \varepsilon K \partial_k \bs{w}(k_0) + \frac{1}{2}\varepsilon^2 K^2 \partial_k^2 \bs{w}(k_0) + \mathcal{O}(\varepsilon^3),\\
	L\left(k\right) &= L\left(k_0 + \varepsilon K\right) = L(k_0)+\eps K \pa_k L(k_0)=L_0+\eps K L_1
\end{align*}
as $\varepsilon \rightarrow 0$, where 
\begin{equation*}
	L_1 \bs{m}:= \left(\partial_k L(k_0)\right)\bs{m} = \begin{pmatrix}
		m_3 \\ 
		0\\
		m_1 
	\end{pmatrix}
\end{equation*} 
and all higher derivatives in $k$ of $L$ vanish since it is linear in $k$.
Differentiation of the linear eigenvalue problem \eqref{E:ev-prob} then produces the equations
\begin{align}
	\left(L_0 + \nu_0 \Lambda\right)\bs{m}&=\bs{0},\label{E:ev-prob-Tayl1}\\
	K \left(L_1 + \nu_1 \Lambda\right)\bs{m}+K \left(L_0 + \nu_0 \Lambda\right) \partial_k \bs{w}(k_0)&=\bs{0},\label{E:ev-prob-Tayl2}\\
	K^2\nu_2 \Lambda \bs{m}+2K^2\left(L_1 + \nu_1 \Lambda\right)\partial_k \bs{w}(k_0)+ K^2(L_0+\nu_0 \Lambda)\partial_k^2 \bs{w}(k_0)&=\bs{0}.\label{E:ev-prob-Tayl3}
\end{align}

The residual is obtained by inserting \eqref{E:uans_hat} in the left-hand side of \eqref{E:Maxw-reduce-FT}. We obtain at $\mathcal{O}(\varepsilon^0)$ the expression $- \ie \widehat{A}E_1(L_0+\nu_0 \Lambda)\bs{m}$ which vanishes due to \eqref{E:ev-prob-Tayl1}. 
At $\mathcal{O}(\varepsilon^1)$ we obtain $-\ie K \Ahat E_1 (L_1 + \nu_1 \Lambda)\bs{m}$. In order to annihilate
the residual also at $\mathcal{O}(\eps^1)$, equation \eqref{E:ev-prob-Tayl2} dictates that we need to extend the ansatz $\Uanshat$ by the term $\eps\Ahat K\pa_k\bs{w}(k_0)E_1$. 

At $\mathcal{O}(\eps^2)$ terms proportional to $E_1$ and those proportional to $E_3$ (as obtained in \eqref{E:Dnl-Uans}) appear in the residual. The latter terms can be removed by introducing a further correction term to $\Uanshat$, namely $\varepsilon^2 (2\uppi)^{-1} \left(\widehat{A} * \widehat{A} * \widehat{A}\right)\bs{h}(x_1) E_3,$ where $\bs{h}$ solves
\begin{equation}
	(L(3k_0) +3\nu_0 \Lambda) \bs{h}=- 3\nu_0\epsilon_3 \begin{pmatrix} m_1^3+m_1m_2^2\\ 
		m_2^3+m_2m_1^2\\
		0
	\end{pmatrix}.
	\label{E:system_h}
\end{equation}
The non-resonance assumption \eqref{E:nres-cond} guarantees that a solution $\bs{h}$ exists. Indeed, $L(3k_0)+3\nu_0 \Lambda$ is injective by \eqref{E:nres-cond} 
and hence the closed range theorem implies $R(L(3k_0)+3\nu_0 \Lambda)=N(L(3k_0)+3\nu_0 \Lambda)^\perp =L^2(\R)^3$,
thanks to Lemma \ref{L:inhom-prob} (case ii) with $k=3k_0$ and $\omega=3\nu_0.$ Here Assumption \eqref{E:ass-esssp} is used. That the right-hand side in \eqref{E:system_h} is in $L^2(\R)^3$ follows from $\bs{m} \in \mathcal{H}^3(\R)^3$, see \eqref{E:ass2-Amwk} and Lemma \ref{L:regularity_ef}.

In summary, the residual is in $\mathcal{O}(\eps^2)$ and contains only terms proportional to $E_1$ (and their complex conjugates) if we modify $\Uanshat$ to 
\begin{equation}\label{eq:umod}
	\begin{aligned}
		\Umodhat\left(x_1,k,t\right) &:= \widehat{A}\left(\frac{k-k_0}{\varepsilon},\varepsilon^2t \right) \left(\bs{m}\left(x_1\right) + \varepsilon K \partial_k \bs{w}\left(x_1,k_0\right)\right) \e^{-\ie\left(\nu_0 + \left(k-k_0\right)\nu_1 \right)t} \\
		&\ +\varepsilon^2 \tfrac{1}{\sqrt{2 \uppi}} \left(\widehat{A} * \widehat{A} * \widehat{A}\right)\left(\frac{k-3k_0}{\varepsilon},\varepsilon^2t \right)\bs{h}(x_1) \e^{-\ie(3 \nu_0 + (k-3k_0)\nu_1)t} + \widehat{\cc}.
	\end{aligned}
\end{equation}
As a result, the $\mathcal{O}(\eps^2)$-terms in the residual of $\Umodhat$ are 
$$\eps^2\Reshat^{(2,E_1)}E_1+\widehat{\cc},$$ 
where
\begin{equation}\label{E:resE1}
	\begin{aligned}
		\Reshat^{(2,E_1)}(x_1,k&,t) :=-\ie \left[K^2\Ahat(K,T)(L_1+\nu_1 \Lambda)\pa_k\bs{w}(x_1,k_0) +\ie \pa_T \Ahat(K,T) \Lambda\bs{m}(x_1) \right.\\
		&\left. +\nu_0\epsilon_3(x_1)(2\uppi)^{-1}(\Ahat*\Abarhat*\Ahat)(K,T)\begin{pmatrix}3m_1^3(x_1)-m_1(x_1)m_2^2(x_1)\\ -3m_2^3(x_1)+m_1^2(x_1)m_2(x_1)\\0 \end{pmatrix}\right],
	\end{aligned}
\end{equation}
again recalling that $m_{1,3}$ are real and $m_2$ is imaginary. 

To derive the amplitude equation for the envelope approximation it will be sufficient that the $L^2$-projection of $\Reshat^{(2,E_1)}$ onto $N(L_0+\nu_0 \Lambda) = \operatorname{span}\{\bs{m}\}$ vanishes, i.e.,
$$P_{\bs{m}}\bs{f}\left(x_1\right) := \langle \bs{f},\bs{m} \rangle_{L^2\left(\R\right)^3} \bs{m}\left(x_1\right)= \int_\R \bs{f}\left(\xi_1 \right) \cdot \overline{\bs{m}} \left(\xi_1\right)\,\mathrm{d}\xi_1 \; \bs{m}\left(x_1\right) = 0$$
with $\bs{f} = \Reshat^{(2,E_1)}$. This is equivalent to the condition that the envelope $A$ satisfies a certain nonlinear Schr\"odinger equation, as we show now. Note that for the complete removal of $\Reshat^{(2,E_1)}$ a further extension of the ansatz is necessary, see Section \ref{S:res}. We
use \eqref{E:ev-prob-Tayl3} and replace $K^2(L_1+\nu_1 \Lambda)\partial_k \bs{w}(k_0)$ in \eqref{E:resE1} by 
\[-\tfrac{1}{2}K^2\nu_2 \Lambda \bs{m}-\tfrac{1}{2}K^2(L_0+\nu_0 \Lambda)\partial_k^2 \bs{w}(k_0).\]
The self-adjointness of $L_0+\nu_0 \Lambda$ implies
$$P_{\bs{m}}\left(-\frac{1}{2}K^2\nu_2 \Lambda \bs{m}-\frac{1}{2}K^2(L_0+\nu_0 \Lambda)\partial_k^2 \bs{w}(k_0)\right)=-\frac{1}{2}K^2\nu_2 P_{\bs{m}}(\Lambda \bs{m})=-\frac{1}{2}K^2\nu_2 \bs{m}$$
due to the normalization of $\bs{m}$. Altogether, $P_{\bs{m}}$ of \eqref{E:resE1} is zero if $\widehat{A}$ satisfies
\begin{equation}
	\begin{aligned}
		0 &= \ie \partial_T \widehat{A} - \frac{1}{2} K^2 \nu_2 \widehat{A} + (2\uppi)^{-1} \nu_0 \int_\R \epsilon_3 \left( 3 m_1^4 - 2 m_1^2 m_2^2 + 3 m_2^4 \right)\,\mathrm{d}x_1 \left( \widehat{A} * \widehat{\overline{A}} * \widehat{A}\right).
	\end{aligned} 
	\label{E:NLS-FT}
\end{equation}
In other words, $A$ has to solve the nonlinear Schr\"odinger equation
\begin{equation}
	\ie \partial_T A = - \frac{1}{2} \nu_2 \partial_{X_2}^2 A + \kappa |A|^2 A
	\label{E:NLS}
\end{equation}
with
\begin{equation*}
	\kappa := - \nu_0 \int_\R \epsilon_3 \left(3 m_1^4 - 2 m_1^2 m_2^2 + 3 m_2^4 \right)\,\mathrm{d}x_1.
\end{equation*}
Here \eqref{E:w-normaliz2} has been used. Equation \eqref{E:NLS} is the so-called effective amplitude equation. Note that for Theorem \ref{T:main} we need smooth solutions $A \in \bigcap_{k=0}^4 C^{4-k}([0,T_0],H^{3+k}(\R))$ of \eqref{E:NLS}. Such solutions are provided by Proposition 3.8 and Remark 3.9 in \cite{tao2006nonlinear}.

\section{Estimation of the residual}\label{S:res}

In this section we estimate the residual in the $\mathcal{H}^3(\R^2)^3$-norm rigorously under the assumption that $A$ solves \eqref{E:NLS}. Here $\Res$ is obtained from $\Reshat$ by applying the inverse Fourier transformation. The modified wave-packet $\Umod$ 
from \eqref{eq:umod} has the residual $\Reshat(\Umodhat) =\eps^2(I-P_m)\Reshat^{(2,E_1)}E_1+\mathcal{O}(\eps^3)$ as shown in Section \ref{S:asymp}. 
Since $\Reshat^{(2,E_1)}$ only depends on $K=\eps^{-1}(k-k_0)$ (and not directly on $k$), see \eqref{E:resE1}, we have formally $\Res^{(2,E_1)}=\mathcal{O}(\eps)$ and therefore formally $\Res(\Umod) = \mathcal{O}(\varepsilon^3)$. The $L^2$-norm is bounded by 
\begin{equation*}
	\|\Res(\Umod)(\cdot,\cdot,t)\|_{L^2(\R^2)^3}\leq C\eps^{\frac{5}{2}}
\end{equation*}
due to the presence of functions depending on $X_2=\eps(x_2-\nu_1 t)$. The loss of the half power of $\eps$ is clear from the substitution $X_2=\eps(x_2-\nu_1 t)$ in the integral of the $L^2$-norm.

It turns out that for the error analysis the residual has to be bounded in the $\mathcal{H}^3(\R^2)^3$-norm by $C \varepsilon^{7/2}$ for all $t\in [0,T_0\eps^{-2}]$, see the estimates of the residual in Section \ref{S:pf-main} especially \eqref{E:estimate_integral_residual}. This requires a further extension of the ansatz. 
We introduce the final modification of the wave-packet by
\begin{align*}
	\Uexthat\left(x_1,k,t\right) &:= \widehat{A}\left(\frac{k-k_0}{\varepsilon},\varepsilon^2t \right) \left(\bs{m}\left(x_1\right) + \varepsilon K \partial_k \bs{w}\left(x_1,k_0\right)\right) \e^{-\ie\left(\nu_0 + \left(k-k_0\right)\nu_1 \right)t} \\
	&\quad \ +\varepsilon^2 (2\uppi)^{-1} \left(\widehat{A} * \widehat{A} * \widehat{A}\right)\left(\frac{k-3k_0}{\varepsilon},\varepsilon^2t \right)\bs{h}(x_1) \e^{-\ie(3 \nu_0 + (k-3k_0)\nu_1)t} \\
	&\quad \ + \varepsilon^2 \widehat{A}\left(\frac{k-k_0}{\varepsilon},\varepsilon^2t \right) \frac{1}{2} K^2 \partial_k^2 \bs{w}\left(x_1,k_0\right) \e^{-\ie\left(\nu_0 + \left(k-k_0\right)\nu_1 \right)t} \\	
	&\quad \ + \varepsilon^2 (2\uppi)^{-1} \left(\widehat{A} * \widehat{\overline{A}} * \widehat{A} \right)\left(\frac{k-k_0}{\varepsilon},\varepsilon^2 t \right) \bs{p}(x_1) \e^{-\ie \left(\nu_0 + (k-k_0)\nu_1 \right)t} + \widehat{\cc},
\end{align*}
where $\bs{p}$ will be chosen such that formally $\Reshat (\Uexthat) = \mathcal{O}(\varepsilon^3)$. 
To determine $\bs{p}$, we calculate
\begin{align*}
	\Reshat(\Uexthat) &= \eps^2E_1(I-P_{\bs{m}})\Reshat^{(2,E_1)} - \frac{1}{2} \ie \varepsilon^2 E_1 K^2(L_0+\nu_0 \Lambda)\partial_k^2 \bs{w}(x_1,k_0)\widehat{A} \\
	&\quad\ - \ie (2 \uppi)^{-1} \varepsilon^2 E_1 \left(L_0 + \nu_0 \Lambda \right) \bs{p} \left(\widehat{A} * \widehat{\overline{A}} * \widehat{A} \right) + \widehat{\cc} + \mathcal{O}(\eps^3).
\end{align*}
\newpage
Since $\widehat{A}$ solves \eqref{E:NLS-FT}, we know that 
\begin{align*}
	(I-P_{\bs{m}})\Reshat^{(2,E_1)}& = \frac{1}{2} \ie K^2(L_0+\nu_0 \Lambda)\partial_k^2 \bs{w}(x_1,k_0) \widehat{A}\\
	&\ - \tfrac{\ie }{\sqrt{2 \uppi}} \left(\kappa \begin{pmatrix}
		\epsilon_1 m_1 \\
		\epsilon_1 m_2 \\
		\upmu_0 m_3
	\end{pmatrix} + \epsilon_3 \nu_0 \begin{pmatrix}
		3 m_1^3 - m_1 m_2^2 \\
		- 3 m_2^3 + m_1^2 m_2 \\
		0
	\end{pmatrix} \right) \left(\widehat{A} * \widehat{\overline{A}} * \widehat{A} \right),
\end{align*}
where \eqref{E:ev-prob-Tayl3} was used again. Therefore the terms of order $\varepsilon^2$ in $\Reshat(\Uexthat)$ vanish if $\bs{p}$ solves
\begin{equation}
	\left(L_0 + \nu_0 \Lambda\right) \bs{p} = - \kappa \begin{pmatrix}
		\epsilon_1 m_1 \\
		\epsilon_1 m_2 \\
		\upmu_0 m_3
	\end{pmatrix} - \epsilon_3 \nu_0 \begin{pmatrix}
		3 m_1^3 - m_1 m_2^2 \\
		- 3 m_2^3 + m_1^2 m_2 \\
		0
	\end{pmatrix}.
	\label{E:system_p}
\end{equation}
Such a function $\bs{p}$ exists since the right-hand side in \eqref{E:system_p} is orthogonal to $\bs{m}$ by the choice of $\kappa$
and it therefore lies in the range of $L_0 + \nu_0 \Lambda$ due to the closed range theorem. 
Here we use Lemma \ref{L:inhom-prob} (case i) with $k=k_0$ and $\omega=\nu_0.$ Here Assumption \eqref{E:ass-esssp} is employed. We now apply the inverse Fourier transformation to obtain
\begin{equation}\label{E:Uext}
	\begin{aligned}
		\Uext\left(x_1,x_2, t\right) &= \varepsilon A\left(X_2,T \right) \bs{m}\left(x_1\right)\e^{\ie\left(k_0 x_2 - \nu_0 t \right)} \\
		&\quad\ - \varepsilon^2 \ie \partial_{X_2} A(X_2,T) \partial_k \bs{w}\left(x_1,k_0\right)\e^{\ie\left(k_0 x_2 - \nu_0 t \right)} \\
		&\quad\ - \varepsilon^3 \frac{1}{2} \partial_{X_2}^2 A(X_2,T) \partial_k^2 \bs{w}\left(x_1,k_0\right)\e^{\ie\left(k_0 x_2 - \nu_0 t \right)} \\
		&\quad\ + \varepsilon^3 |A(X_2,T)|^2 A(X_2,T) \bs{p}(x_1)\e^{\ie\left(k_0 x_2 - \nu_0 t\right)} \\
		&\quad\ +\varepsilon^3 A^3(X_2,T)\bs{h}(x_1) \e^{3 \ie\left(k_0 x_2 - \nu_0 t\right)} + \cc,
	\end{aligned}
\end{equation}
recalling that $X_2=\eps(x_2-\nu_1 t)$ and $T=\eps^2 t$. Since $\Reshat(\Uexthat)$ is of order $\varepsilon^3$ 
after transformation $\Res(\Uext)$ is of order $\varepsilon^4$ formally. The terms of order $\varepsilon^4$ of $\Res(\Uext)$ can be found in Appendix \ref{S:resiudal_order_4}.

\begin{remark}\label{R:Uext-div-IFC}
	Note that our residual incorporates neither the divergence condition on $\bs{\cD}$ nor the interface conditions. This is because these quantities do not directly appear in the $\cG^3$-norm which we use to estimate the error. 
	
	Nevertheless, for $\bs{\cE}_{\rm{ext}} := \left(U_{\rm{ext},1}, U_{\rm{ext},2},0 \right)^\top$ one can show that $\nabla\cdot\bs{\cD}(\bs{\cE}_{\rm{ext}})=\mathcal{O}(\eps^2)$.
	Indeed, for the divergence condition we have 
	$$\pa_{x_1} \cD_1(\bs{\cE}_{\rm{ext}})+\pa_{x_2} \cD_2(\bs{\cE}_{\rm{ext}}) = \eps F_1(\pa_{x_1}(\epsilon_1m_1) +\epsilon_1 \ie k_0 m_2)A+\cc + \mathcal{O}(\eps^2)= \mathcal{O}(\eps^2)$$
	because $\pa_{x_1}(\epsilon_1m_1) +\epsilon_1 \ie k_0 m_2=0$, see \eqref{E:linear_divergence} at $k=k_0$.
	
	Regarding the interface conditions, $\Uans$ and $\Uext$ fulfill \eqref{E:IFC-red-E2H3} exactly
	since the second and third components of $\bs{m},$ $\partial_k \bs{w}(k_0),$ $\partial_k^2 \bs{w}(k_0),$ $\bs{h}$ and $\bs{p}$ 
	belong to $H^1(\R)$ and are therefore continuous at the interface. Moreover, the jumps of 
	$\cD_1(\bs{\cE}_{\rm{ans}})$ and $\cD_1(\bs{\cE}_{\rm{ext}})$ at $x_1=0$ are of order $\mathcal{O}(\eps^3)$, respectively $\mathcal{O}(\eps^4)$.
	Indeed, at $\mathcal{O}(\eps)$ condition \eqref{E:IFC-red-D1} holds exactly for $\Uans$ and $\Uext$ because 
	$\bs{m}$ satisfies the interface conditions \eqref{E:IFC-linear}. At $\mathcal{O}(\eps^2)$ there are no 
	contributions for $\Uans$. For $\Uext$ only a linear term involving $\partial_k \bs{w}(\cdot,k_0)$
	appears. As \eqref{E:IFC-linear} holds for each $k$, we have $\llbracket \epsilon_1 \partial_k w_1(\cdot,k_0) \rrbracket_{\text{1D}}=0$
	and also $\llbracket \epsilon_1 \partial_k^2 w_1(\cdot,k_0) \rrbracket_{\text{1D}}=0$, and hence
	$\Uext$ satisfies \eqref{E:IFC-red-D1} at $\mathcal{O}(\eps^2)$. 
	Finally, \eqref{E:IFC-red-D1} holds for $\Uext$ also at $\mathcal{O}(\eps^3)$ since $\bs{h} \in D(L(3k_0,3 \nu_0))$ and $\bs{p} \in D(L(k_0,\nu_0))$ solve 
	\eqref{E:system_h} and \eqref{E:system_p}, respectively. This fact and the continuity of $p_1$ and $h_1$ imply the jump equations
	\begin{align*}
		\llbracket \epsilon_1 p_1 \rrbracket_{\text{1D}} = -\llbracket \epsilon_3 \left(3 m_1^3 - m_1 m_2^2\right) \rrbracket_{\text{1D}} \quad\text{ and }\quad
		\llbracket \epsilon_1 h_1 \rrbracket_{\text{1D}} = -\left \llbracket \epsilon_3 \left(m_1^3 + m_1 m_2^2\right) \right \rrbracket_{\text{1D}},
	\end{align*}
	implying that the first nonlinear contribution in \eqref{E:IFC-red-D1} for $\Uext$ is canceled.
	
\end{remark}

As we explain next, $\Res(\Uext)(\cdot,\cdot,t)$ lies in $L^2(\R^2)^3$ for all $t\in [0,T_0\eps^{-2}]$ if, e.g., 
\beq\label{E:ass-Amwk}
\begin{aligned}
&A\in \bigcap_{k=0}^1 C^{1-k}([0,T_0], H^{2+k}(\R)) \quad \text{and} \quad {\bs m}, \pa_k{\bs w}(\cdot, k_0), \pa_k^2{\bs w}(\cdot, k_0),\bs{h}, \bs{p} \in L^2(\R)^3\cap L^\infty(\R)^3.
\end{aligned}
\eeq

So far we have used the $\eps$-orders in a formal sense, i.e., without specifying the norm. To determine the asymptotic order of the $L^2(\R^2)^3$-norm of the residual, 
we note that the summands of $\Res(\Uext)(\cdot,\cdot,t)$ of (the smallest) order $\eps^4$ have the form $g(\bs{x}):=f_1(x_1)f_2(\eps x_2)f_3(x_2)$ with $f_1, f_2\in L^2(\R)$ and $f_3\in L^\infty(\R)$.
Such products can be estimated by
\begin{equation*}
	\|g\|_{L^2(\R^2)}\leq\eps^{-\frac{1}{2}}\|f_1\|_{L^2(\R)}\|f_2\|_{L^2(\R)}\|f_3\|_{L^\infty(\R)}.
\end{equation*}
Terms of higher order in $\eps$ are of a similar form. The $x_1$-derivatives do not appear in the residual, 
since they only occur at low orders of $\eps$ and were canceled in the construction. Hence, in \eqref{E:ass-Amwk} we have to use 
Sobolev spaces only for $A$.
In the residual the derivatives of highest order are $\partial_{X_2}^3 A$ and $\partial_T\partial_{X_2}^2 A$. This can be seen from the form of $\Uext$ and the fact that the Maxwell equations are of first order. These terms are bounded in
$L^2(\R)^3$ uniformly in time according to \eqref{E:ass-Amwk}. Due to the embedding $H^1(\R)\hookrightarrow L^\infty(\R)$ the lower-order factors 
are bounded. The products appearing in nonlinear terms are estimated by (suppressing the time dependence)
\begin{align*}
	\||A|^2\pa_{X_2} A\|_{L^2(\R)} &\leq \|A\|_{L^\infty(\R)}^2\| \pa_{X_2} A\|_{L^2(\R)}\leq C\|A\|_{H^1(\R)}^3,\\
	\||\pa_{X_2}^2 A|^2\pa_{X_2}^2 A\|_{L^2(\R)} &\leq \|\pa_{X_2}^2 A\|_{L^\infty(\R)}^2\|\pa_{X_2}^2 A\|_{L^2(\R)}\leq C\|A\|_{H^3(\R)}^3,\\
	\||\pa_{X_2}^2 A|^2 \pa_T \pa_{X_2}^2 A\|_{L^2(\R)}&\leq \|\pa_{X_2}^2 A\|_{L^\infty(\R)}^2\|\pa_T \pa_{X_2}^2 A\|_{L^2(\R)}\leq C\|A\|_{H^3(\R)}^2 \|\pa_T A\|_{H^2(\R)},
\end{align*}
for instance. Using these principles, we easily deduce
\begin{equation*}
\|\Res(\Uext)(\cdot,\cdot,t)\|_{L^2(\R^2)^3}=\|\Reshat(\Uexthat)(\cdot,\cdot,t)\|_{L^2(\R^2)^3}\leq C \eps^{\frac{7}{2}}
\end{equation*}
under condition \eqref{E:ass-Amwk}. The constant $C$ depends on the norms of $A,$ ${\bs m},$ $\pa_k{\bs w}(\cdot, k_0),$ $\pa_k^2{\bs w}(\cdot, k_0),$ $\bs{h}$, and 
$\bs{p}$ appearing in \eqref{E:ass-Amwk}.

For the error analysis in Section \ref{S:pf-main} we have to estimate the residual in $\mathcal{G}^3(\R^2 \times [0,T_0\varepsilon^{-2}])^3$ and not only in $L^2(\R^2)^3$ for all $t \in [0,T_0 \varepsilon^{-2}]$. To do this, we impose the stricter conditions
\beq\label{E:ass2-Amwk}
A \in \bigcap_{k=0}^4 C^{4-k}([0,T_0],H^{2+k}(\R)) \quad \text{and} \quad {\bs m}, \pa_k{\bs w}(\cdot, k_0), \pa_k^2{\bs w}(\cdot, k_0), \bs{h}, \bs{p} \in \mathcal{H}^3(\R)^3.
\eeq 
Since we want to estimate the derivatives up to order three of the residual, it is clear that the regularity of the envelope $A$ has also to increase by three orders in space and time. Since no $x_1$-derivative appears in the residual we can use the algebra property of $\mathcal{H}^3(\R)$ to control the appearing nonlinear terms.
With the same arguments as before it is now possible to bound the residual by
\beq\label{E:Res-est}
\|\Res(\Uext)\|_{\mathcal{G}^3(\R^2\times(0,T_0\eps^{-2}))^3}\leq C \eps^{\frac{7}{2}}
\eeq
under the above conditions \eqref{E:ass2-Amwk}. The constant $C$ depends on the norms of $A,$ ${\bs m},$ $\pa_k{\bs w}(\cdot, k_0),$ $\pa_k^2{\bs w}(\cdot, k_0),$ 
$\bs{h}$, and $\bs{p}$ appearing in \eqref{E:ass2-Amwk}. Furthermore, under condition \eqref{E:ass2-Amwk} we have
\begin{equation*}
	\Uext \in \bigcap_{k=0}^4 C^{4-k}([0,T_0\eps^{-2}],\mathcal{H}^{\min \{3;k\}}(\R^2))^3. 
\end{equation*}

In the bootstrapping argument of Section \ref{S:bootstrap} we need somewhat stronger regularity properties of $\Uext$. They follow
from the structure of $\Uext$, which is a sum of products of functions in $(x_2,t)$ and in $x_1$, where the latter only appear linearly.
So we can use the Sobolev embedding $\mathcal{H}^1(\R) \hookrightarrow L^\infty(\R)$ in both space dimensions separately, 
thus avoiding the less favorable embedding $\mathcal{H}^2(\R^2) \hookrightarrow L^\infty(\R^2)$. 
Take multi-indices $\bs{\alpha}=(\alpha_1,\alpha_2,\alpha_t)^\top$ with $|\bs{\alpha}|\le 3$ and $\alpha_1 \le 2$, as well as
$|\bs{\beta}|= 3$ with $\beta_1\in\{1,2\}$. For $\partial^{\bs{\alpha}} := \partial_{x_1}^{\alpha_1}\partial_{x_2}^{\alpha_2}\partial_t^{\alpha_t}$ condition \eqref{E:ass2-Amwk}
implies that 
\beq\label{E:regularity_Uext}
\|\partial^{\bs{\alpha}}\Uext\|_{L^\infty(\R^2\times (0,T_0\eps^{-2}))^3},\|\partial^{\bs{\beta}} \partial_t \Uext\|_{L^\infty(\R^2\times (0,T_0\eps^{-2}))^3}\le C\eps,
\eeq 
and
\beq\label{E:regularity_Uext2}
\int_{\R} \sup_{x_2\in\R} |\partial_{x_1}^3 \partial_t^k \Uext(x_1,x_2,t)|^2\,\mathrm{d} x_1 \le C\eps^2, \qquad k \in \{0,1\}, 
\eeq
for all $t\in [0,T_0\eps^{-2}]$. For $|\bs{\gamma}|= 3$ with $\gamma_1=0$ we have $\partial^{\bs{\gamma}} \partial_t \Uext = \mathcal{A} + \mathcal{B}$ with
\begin{align*}
	\mathcal{A}(x_1,x_2,t) &:= \varepsilon A\left(X_2,T \right) \bs{m}\left(x_1\right)\partial^{\bs{\gamma}} \partial_t \left(\e^{\ie\left(k_0 x_2 - \nu_0 t \right)}\right),\\
	\mathcal{B}(x_1,x_2,t) &:= \partial^{\bs{\gamma}} \partial_t \Uext(x_1,x_2,t) - \mathcal{A}(x_1,x_2,t),
\end{align*}
and one derives the estimates
\beq\label{E:regularity_Uext3}
\|\mathcal{A}\|_{L^\infty(\R^2\times (0,T_0\eps^{-2}))^3}\le C\eps, \quad \int_{\R} \sup_{x_1\in\R} |\mathcal{B}(x_1,x_2,t)|^2\,\mathrm{d} x_2 \le C\eps^2.
\eeq

An application of Lemma \ref{lemma_regularity_of_eigenfunction} gives us the necessary regularity of ${\bs m}$, $\pa_k{\bs w}(\cdot, k_0)$, $\pa_k^2{\bs w}(\cdot, k_0)$, $\bs{h}$, and $\bs{p}$ under our assumptions on $\epsilon_1$ and $\epsilon_3$.

\begin{lem}\label{L:regularity_ef}
	Let $\bs{m}, \partial_k \bs{w}\left(k_0\right), \partial_k^2 \bs{w}\left(k_0\right), \bs{h}, \bs{p} \in L^2(\R)^3$ be defined as before. Assume that $\epsilon_1, \epsilon_{3} \in \mathcal{W}^{3,\infty}(\R)$. 
	Then $\bs{m}, \partial_k \bs{w}\left(k_0\right), \partial_k^2 \bs{w}\left(k_0\right),\bs{p}, \bs{h} \in \mathcal{H}^3(\R)^3$.
\end{lem} 
\begin{proof}
	Since $\left(L(k_0) +\nu_0 \Lambda\right) \bs{m} = \bs{0}$, Lemma \ref{lemma_regularity_of_eigenfunction} with $\bs{f} = \bs{0}$ 
	shows that $\bs{m} \in \mathcal{H}^3(\R)^3$. Next, differentiating $\left(L(k) + \omega \Lambda\right) \bs{w} = \bs{0}$ in $k$, 
	we see that $\partial_k \bs{w}\left(k_0\right)$ and $\partial_k^2 \bs{w}\left(k_0\right)$ solve
	\begin{align*}
		\left(L(k_0) +\nu_0 \Lambda\right) \partial_k \bs{w}\left(k_0\right) = &- \left(\partial_k L(k_0) +\partial_k\omega(k_0) \Lambda\right)\bs{m}, \\
		\left(L(k_0) +\nu_0 \Lambda\right) \partial_k^2 \bs{w}\left(k_0\right) = &- 2 \left( \partial_k L(k_0) +\partial_k\omega(k_0) \Lambda\right)\partial_k \bs{w}\left(k_0\right) - \left(\partial_k^2 L(k_0) +\partial_k^2\omega(k_0)\Lambda\right) \bs{m}.
	\end{align*}
	As $\bs{m} \in \mathcal{H}^3(\R)^3$, the functions
	\begin{equation*}
		\left(\partial_k L(k_0) +\partial_k\omega(k_0) \Lambda\right)\bs{m} = \begin{pmatrix}
			\epsilon_1 \nu_1 m_1 + m_3 \\
			\epsilon_1 \nu_1 m_2 \\
			m_1 + \upmu_0 \nu_1 m_3 
		\end{pmatrix}, 
		\end{equation*}
		\begin{equation*}
\left(\partial_k^2 L(k_0) +\partial_k^2\omega(k_0)\Lambda\right) \bs{m} = \begin{pmatrix}
			\epsilon_1 \nu_2 m_1 \\
			\epsilon_1 \nu_2 m_2 \\
			\upmu_0 \nu_2 m_3 
		\end{pmatrix}
	\end{equation*}
	belong to $\mathcal{H}^3(\R)^3$. Therefore from Lemma \ref{lemma_regularity_of_eigenfunction} we infer $\partial_k \bs{w}\left(k_0\right) \in \mathcal{H}^3(\R)^3$.
	This fact implies $\partial_k^2 \bs{w}\left(k_0\right) \in \mathcal{H}^3(\R)^3$ in the same way.
	To treat $\bs{h}$ and $\bs{p}$ from \eqref{E:system_h} and \eqref{E:system_p}, we note that the right-hand sides in
	\begin{align*}
		\left(L(k_0) +\nu_0 \Lambda\right) \bs{p} &= - \kappa \begin{pmatrix}
			\epsilon_1 m_1 \\
			\epsilon_1 m_2 \\
			\upmu_0 m_3
		\end{pmatrix} - \epsilon_3 \nu_0 \begin{pmatrix}
			3 m_1^3 - m_1 m_2^2 \\
			- 3 m_2^3 + m_1^2 m_2 \\
			0
		\end{pmatrix},\\
	(L(3k_0) +3\nu_0 \Lambda) \bs{h} &= - 3\nu_0\epsilon_3 \begin{pmatrix} m_1^3+m_1m_2^2\\ 
		m_2^3+m_2m_1^2\\
		0
	\end{pmatrix}.
	\end{align*}
	are also contained in $\mathcal{H}^3(\R)^3$ since $\bs{m}\! \in \!\mathcal{H}^3(\R)^3$. Hence, the statement follows as before.
\end{proof}

\newpage
\section{Local existence theory}\label{S:exist}

We employ local existence results of \cite{schnaubelt2018local} for linear and quasilinear hyperbolic problems. 
We first define some additional function spaces. For any open $\Omega \subset \R^2$ and $J\subset \R$ we use 
\begin{align*}
	F^{m,n}(\Omega \times J) &:= \left\{\! A \in W^{1,\infty}(\Omega \times J)^{n \times n} \mv \partial^{\bs{\alpha}} A \in L^\infty(J,H^{m-|\bs{\alpha}|}(\Omega))^{n \times n}\! \right.\\
	&\ \textcolor{white}{:=\Big\{} + W^{m-|\bs{\alpha}|,\infty}(\Omega \times J)^{n \times n}  \mathrm{~for~all~} \bs{\alpha} \in \N_0^3 \mathrm{~with~} 1 \leq |\bs{\alpha}| \leq m \Big\}, \\
	\norm{A}_{F^{m,n}(\Omega \times J)} &:= \max \!\left\{\!\norm{A}_{W^{1,\infty}(\Omega \times J)^{n \times n}};\! \max_{1 \leq |\bs{\alpha}| \leq m} \norm{\partial^{\bs{\alpha}} A}_{L^\infty(J,H^{m-|\bs{\alpha}|}(\Omega))^{n \times n} + W^{m-|\bs{\alpha}|,\infty}(\Omega \times J)^{n \times n}}\!\right\}\!,\\
	\mathcal{F}^{m,n}(\R^2\times J) &:= \left\{ A \in \mathcal{W}^{1,\infty}(\R^2 \times J)^{n \times n} \mv  A^\pm \in F^{m,n}(\R_\pm^2\times J) \right\},\\
	\norm{A}_{\mathcal{F}^{m,n}(\R^2\times J)} &:= \max \left\{\norm{A^-}_{F^{m,n}(\R_-^2 \times J)};\norm{A^+}_{F^{m,n}(\R_+^2 \times J)}\right\},
\end{align*}
with the usual definition
\begin{align*}
	&\norm{A}_{L^\infty(J,H^{m-|\bs{\alpha}|}(\Omega)) + W^{m-|\bs{\alpha}|,\infty}(\Omega \times J)} := \inf\! \left\{\!\norm{B}_{L^\infty(J,H^{m-|\bs{\alpha}|}(\Omega))} \!+\! \norm{C}_{W^{m-|\bs{\alpha}|,\infty}(\Omega \times J)} \mv \right. \\
	&\quad \textcolor{white}{:=\inf \{}  A = B + C, B \in L^\infty(J,H^{m-|\bs{\alpha}|}(\Omega)), C \in W^{m-|\bs{\alpha}|,\infty}(\Omega \times J) \Big\}.
\end{align*}
For a fixed time instant we use the spaces
\begin{align*}
	F^{m,n}_0(\Omega) &:= \left\{\!A \in L^\infty (\Omega)^{n \times n} \mv \partial^{\bs{\alpha}} A \in H^{m-|\bs{\alpha}|}(\Omega)^{n \times n} + W^{m-|\bs{\alpha}|,\infty}(\Omega)^{n\times n}\right.\\ 
	&\textcolor{white}{:=\big\{}\left.\mathrm{~for~all~} \bs{\alpha} \in \N_0^2 \mathrm{~with~} 1 \leq |\bs{\alpha}| \leq m \right\}, \\
	\norm{A}_{F_0^{m,n}(\Omega)} &:= \max\left\{\norm{A}_{L^{\infty}(\Omega)^{n \times n}}; \max_{1 \leq |\bs{\alpha}| \leq m} \norm{\partial^{\bs{\alpha}} A}_{H^{m-|\bs{\alpha}|}(\Omega)^{n \times n} + W^{m-|\bs{\alpha}|,\infty}(\Omega)^{n \times n}}\right\},\\
	\mathcal{F}_0^{m,n}(\R^2) &:= \left\{ A \in L^\infty(\R^2)^{n \times n} \mv A^- \in F_0^{m,n}(\R_-^2), A^+ \in F_0^{m,n}(\R_+^2) \right\},\\
	\norm{A}_{\mathcal{F}_0^{m,n}(\R^2)} &:= \max \left\{\norm{A^-}_{F_0^{m,n}(\R_-^2)}; \norm{A^+}_{F_0^{m,n}(\R_+^2)}\right\}.
\end{align*}
Finally, we define subspaces of $\mathcal{F}^{m,n}$ with the properties of positive definiteness, or convergence to a constant matrix for $|(\bs{x},t)|\to \infty$. For $\eta>0$ we set
\begin{align*}
	\mathcal{F}^{m,n}_\eta(\R^2\times J) &:= \left\{A \in \mathcal{F}^{m,n}(\R^2\times J) \mv A = A^\top, \bs{v}^\top A \bs{v} \geq \eta |\bs{v}|^2 \mathrm{~for~all~} \bs{v} \in \R^n \right\},\\
	\mathcal{F}^{m,n}_{{\rm cv}}(\R^2\times J) &:= \left\{ A \in \mathcal{F}^{m,n}(\R^2\times J) \mv \exists \widetilde{A}\in \R^{n\times n} : \lim_{|(\bs{x},t)| \rightarrow \infty}A(\bs{x},t) =\widetilde{A}\right\}, \\
	\mathcal{F}^{m,n}_{\eta,{\rm cv}}(\R^2\times J) &:= \mathcal{F}^{m,n}_{\eta}(\R^2\times J) \cap \mathcal{F}^{m,n}_{{\rm cv}}(\R^2\times J).
\end{align*}

In the linear setting, as in (4.1) of \cite{schnaubelt2018local} we will use a system of the form
\begin{equation}
	\left\{
	\begin{aligned}
		A_t(\bs{x},t)\partial_t \bs{U}^\pm + \sum_{j=1}^2 A_j \partial_{x_j} \bs{U}^\pm +M(\bs{x},t) \bs{U}^\pm &= \bs{f}^\pm, & \bs{x} &\in \R_\pm^2, & t &\in J, \\
		B_\Gamma \begin{pmatrix} \bs{U}^+ \\ \bs{U}^- \end{pmatrix} &= \bs{0}, & \bs{x} &\in \Gamma, & t &\in J, \\
		\bs{U}(0) &= \bs{U}^{(0)}, & \bs{x} &\in \R^2
	\end{aligned}
	\right.
	\label{E:u-syst-lin}
\end{equation}
on the interval $J:=(0,T')$ with some $T'>0$, where $M:\R^2\times J\to \R^{3\times 3}$, $A_t :\R^2\times J\to \R^{3\times 3}$, 
$A_t(\bs{x},t)$ is symmetric for all $(\bs{x},t)$, and
\begin{align*}
	A_1 &:= \begin{pmatrix}
		0 & 0 & 0 \\
		0 & 0 & 1\\
		0 & 1 & 0
	\end{pmatrix}, \quad A_2 := \begin{pmatrix}
		0 & 0 & -1 \\
		0 & 0 & 0\\
		-1 & 0 & 0
	\end{pmatrix}, \qquad
	B_\Gamma := \begin{pmatrix}
		0 & 1 & 0 & 0 & -1 & 0 \\
		0 & 0 & 1 & 0 & 0 & -1
	\end{pmatrix},
\end{align*}
see \eqref{E:rbeta-syst} in Section \ref{S:bootstrap}. Clearly, $B_\Gamma ( \bs{U}^+, \bs{U}^-)^\top$ encodes the interface conditions $\llbracket U_2 \rrbracket = \llbracket U_3 \rrbracket = 0$ on $\Gamma$. Note that we are not going to use \eqref{E:u-syst-lin} in order to study the linear part of \eqref{E:Maxw-red} but rather to study a fixed point problem in the bootstrapping argument for the nonlinear system in Section \ref{S:bootstrap}. Hence, we need the inhomogeneous term $ \bs{f}$ as well as the linear term $M\bs{U}$ in \eqref{E:u-syst-lin}.

\begin{defi}[Weak Solution of the Linear Hyperbolic Problem]
    By a weak solution of \eqref{E:u-syst-lin} we mean a function $\bs{U} \in C(\overline{J}, L^2(\R^2))^3$ that satisfies
	\begin{align*}
		&\int_J \int_{\R^2} \bs{f} \cdot \bs{\varphi}\,\mathrm{d}\bs{x}\,\mathrm{d}t = -\int_J \int_{\R^2} \left(\bs{U} \cdot \partial_t(A_t\bs{\varphi}) + \bs{U} \cdot \partial_{x_1} (A_1 \bs{\varphi}) + \bs{U} \cdot \partial_{x_2}(A_2 \bs{\varphi}) - \bs{U} \cdot M^\top \bs{\varphi} \right)\,\mathrm{d}\bs{x}\,\mathrm{d}t
	\end{align*}
for all
	\begin{align*}
		\bs{\varphi} \in  \left\{\bs{\psi} \mv \bs{\psi}^+ \in H^1_0\left(\R_+^2\times J\right)^3,\,\bs{\psi}^- \in H^1_0\left(\R_-^2\times J\right)^3\right\},
	\end{align*}
	$\operatorname{Tr}_{\Gamma\times J} \left(B_\Gamma \left(\bs{U}^+,\bs{U}^-\right)^\top \right) = \bs{0}$, \ and \ $\bs{U}(0) = \bs{U}^{(0)}$. 
\end{defi}
For the trace in the above definition, note that for a weak solution $\bs{U}$ of \eqref{E:u-syst-lin} the space-time divergence of 
$(A_1\varphi_1, A_2\varphi_2,A_t\varphi_t)$ belongs to $L^2(\R^2_{\pm}\times J)^3$ and thus has traces in $H^{-1/2}(\partial (\R^2_{\pm}\times J))^3$. 
This fact and the properties of $A_j$ allow us to define the initial and interface conditions in $H^{-1/2}(\R^2\times \{0\})$ and 
$H^{-1/2}(\Gamma\times J)^3$, respectively. See Chapter 2.1 in \cite{spitz2017local} for an in depth discussion of this trace.

The following existence result is a consequence of Theorem 3.1 in \cite{schnaubelt2018local}.
\begin{them}[Existence result for the linear system \eqref{E:u-syst-lin}]\label{T:lin-exist}	
	Let $\eta,T',r > 0$, $m \in\{0,1,2,3\}$, and $J=(0,T')$. Take coefficients $A_t \in \mathcal{F}_{\eta,{\rm cv}}^{3,3}(\R^2\times J)$ and $M \in \mathcal{F}_{{\rm cv}}^{3,3}(\R^2\times J)$ with
	\begin{align*}
		\norm{A_t}_{\mathcal{F}^{3,3}(\R^2\times J)}, \norm{A_t(\cdot,0)}_{\mathcal{F}_0^{2,3}(\R^2)}, \norm{\partial_t^j A_t(\cdot,0)}_{\mathcal{H}^{2-j}(\R^2)^{3 \times 3}} &\leq r, \\
		\norm{M}_{\mathcal{F}^{3,3}(\R^2\times J)}, \norm{M(\cdot,0)}_{\mathcal{F}_0^{2,3}(\R^2)}, \norm{\partial_t^j M(\cdot,0)}_{\mathcal{H}^{2-j}(\R^2)^{3\times 3}} &\leq r,
	\end{align*}
	for all $j \in \{1,2\}$. Choose $\bs{f} \in \mathcal{H}^m(\R^2\times J)^3$ and $\bs{U}^{(0)} \in \mathcal{H}^m(\R^2)^3$ such that the linear compatibility conditions of order $m$ are satisfied, see Definition \ref{D:compat}.
	
	Then there is a unique weak solution $\bs{U}$ of \eqref{E:u-syst-lin} in $\mathcal{G}^m(\R^2\times J)^3$ and a constant $C_m=C_m(\eta,r,T')\geq 1$ such that
	\begin{equation}
		\norm{\bs{U}}_{\mathcal{G}^m(\R^2\times J)^3}^2 \leq C_m \Big(\norm{\bs{U}^{(0)}}_{\mathcal{H}^m(\R^2)^3}^2 + \norm{\bs{f}}_{\mathcal{H}^m(\R^2\times J)^3}^2 + \sum_{j=0}^{m-1} \norm{\partial_t^j \bs{f}(\cdot,0)}_{\mathcal{H}^{m-1-j}(\R^2)^3}^2\Big)
		\label{E:est-lin}
	\end{equation}
	where, as usual, the sum is empty if $m=0$.
\end{them}
\begin{remark}\label{rem:spitz}
	Theorem 3.1 of \cite{schnaubelt2018local} deals with spatial domains in $\R^3$ instead of $\R^2$ and the solution vector takes values in $\R^6$ instead of $\R^3$, 
	but the above case can be treated in an analogous and simpler way. We do not need the lengthy localization process discussed in
	\cite{schnaubelt2018local}. As in this paper, one reduces the interface problem on $\R^2\setminus \Gamma$
	to a boundary value problem on $\R^2_+$ and the latter can again be solved in $L^2$ by means of general results from 
	\cite{eller2012symmetric}. To obtain solutions in $\mathcal{G}^m$ for appropriate data, one first needs a priori estimates.
	These can be shown as in Section~\ref{S:bootstrap} below which uses ideas from \cite{schnaubelt2018local}. The regularity of 
	solutions can then be shown by approximation arguments which are simplified versions of those in \cite{schnaubelt2018local}.
\end{remark}

In the bootstrapping argument in Section \ref{S:bootstrap} we need the next approximation result, involving the space
\[\mathcal{D}(\R^2)^3 := \left\{\bs{\varphi} \mv \bs{\varphi}^+ \in C^\infty(\R_+^2)^3,\,\bs{\varphi}^- \in C^\infty (\R_-^2)^3,\, \operatorname{supp} \bs{\varphi} \subset \R^2 \mathrm{~compact}\right\}.\]
\begin{lem}
	\label{L:approx-arg}
	Let $T'>0$, $J=(0,T')$, $\bs{U}^{(0)} \in L^2(\R^2)^3$, $A_t \in \mathcal{F}^{3,3}_{\eta,{\rm cv}}(\R^2\times J)$, $M \in \mathcal{F}^{3,3}_{{\rm cv}}(\R^2\times J)$ and $\bs{f} \in \mathcal{G}^0(\R^2\times J)^3$. Take a weak solution $\bs{U} \in \mathcal{G}^0(\R^2\times J)^3$ of \eqref{E:u-syst-lin} for the data $\big(\bs{f},\bs{U}^{(0)}\big)$. Then the following statements are true.
	\begin{enumerate}[label=\roman*)]
		\item There are sequences $\big(\bs{U}^{(0)}_{n}\big)_n \subset \mathcal{D}(\R^2)^3$ and $(\bs{f}_n)_n \subset \mathcal{H}^1(\R^2\times J)^3$ such that $\bs{U}^{(0)}_{n} \rightarrow \bs{U}^{(0)}$ in $L^2(\R^2)^3$ and $\bs{f}_n \rightarrow \bs{f}$ in $L^2(\R^2\times J)^3$ for $n \rightarrow \infty$ and the linear compatibility conditions of order $1$ are satisfied, i.e., $B_\Gamma \left(\bs{U}^{(0),+}_{n},\bs{U}^{(0),-}_{n}\right)^\top = 0$.
		\item There exists a sequence $(\bs{U}_n)_n \subset \mathcal{G}^1(\R^2\times J)^3$ such that for all $n \in \N$ the function $\bs{U}_n$ solves \eqref{E:u-syst-lin} for the data $\big(\bs{f}_n,\bs{U}^{(0)}_{n}\big)$ and $\bs{U}_n \rightarrow \bs{U}$ in $\mathcal{G}^0(\R^2\times J)^3$ for $n \rightarrow \infty$.
	\end{enumerate}
\end{lem}
\begin{proof}
	For \textit{i)} we use that $C^\infty_c(\Omega)$ and $H^1(\Omega)$ are dense in $L^2(\Omega)$ for any domain $\Omega$. Therefore we can choose sequences $(\bs{w}_n)_n \subset \mathcal{D}(\R^2)^3$ and $(\bs{f}_n)_n \subset \mathcal{H}^1(\R^2 \times J)^3$ with $\bs{w}_n \rightarrow \bs{U}^{(0)}$ in $L^2(\R^2)^3$ and $\bs{f}_n \rightarrow \bs{f}$ in $L^2(\R^2\times J)^3$ for $n \rightarrow \infty$. 
	To guarantee the compatibility condition, we introduce the characteristic function $\chi_{M_n}$ with $M_n := \R^2 \setminus \left\{\bs{x} \in \R^2 \mv x_1 \in \left[-\tfrac{1}{n},\tfrac{1}{n}\right]\right\}$.
	By dominated convergence, one easily sees that $\bs{w}_n \chi_{M_n} \rightarrow \bs{U}^{(0)}$ in $L^2(\R^2)^3$. Since $\Gamma \cap M_n = \emptyset$ we also get that $\operatorname{Tr}_\Gamma \left(B_\Gamma \bs{w}_{n}\chi_{M_n}\right) = \bs{0}$. Now we mollify $\bs{w}_{n}\chi_{M_n}$ to produce functions $\bs{U}^{(0)}_{n} \in \mathcal{D}(\R^2)^3$ with the stated properties.
	
	The existence of $(\bs{U}_n)_n$ in assertion \textit{ii)} is a direct consequence of Theorem \ref{T:lin-exist}. To show the convergence we use that \eqref{E:u-syst-lin} is a linear problem, consequently $\bs{U}_n - \bs{U}$ is a weak solution of \eqref{E:u-syst-lin} for the data $\big(\bs{f}_n - \bs{f}, \bs{U}^{(0)}_{n} - \bs{U}^{(0)}\big)$. Estimate \eqref{E:est-lin} thus yields
	\begin{equation*}
		\norm{\bs{U}_n - \bs{U}}_{\mathcal{G}^0(\R^2\times J)^3} \leq C \Big( \norm{\bs{U}^{(0)}_{n} - \bs{U}^{(0)}}_{L^2(\R^2)^3} + \norm{\bs{f}_n - \bs{f}}_{L^2(\R^2 \times J)^3}\Big).
	\end{equation*}
	The convergence properties of $(\bs{f}_n)_n$ and $\big(\bs{U}^{(0)}_{n}\big)_n$ complete the proof. 
\end{proof}

For the quasilinear term we have to define spaces for functions whose domain of definition is not the full $\R^2 \times \R^3$, namely
\begin{align*}
	&\mathcal{ML}^{m,k}\left(\R^2,\Omega_\pm\right) := \Big\{S: \left(\R_+^2 \times \Omega_+\right) \cup \left(\R_-^2 \times \Omega_-\right) \rightarrow \R^{k \times k} \mv S^\pm \in C^m\left(\R_\pm^2 \times \Omega_\pm,\R^{k \times k}\right), \\
	& \qquad \qquad \qquad \sup_{(\bs{x},\bs{u}) \in \R_\pm^2 \times \mathcal{U}_\pm} |\partial^{\bs{\alpha}} S(\bs{x},\bs{u})| < \infty \mathrm{~for~all~compact~} \mathcal{U}_\pm \subset \Omega_\pm \mathrm{~and~} \bs{\alpha} \in \N_0^5 \mathrm{~with~} |\bs{\alpha}| \leq m \Big\},
\end{align*}
where $\Omega_\pm \subset \R^3$ are open and $S^+$ and $S^-$ are the restrictions of $S$ to $\R_+^2 \times \Omega_+$ and $\R_-^2 \times \Omega_-$, respectively. As for the spaces $\mathcal{F}$ we use the subscripts $\eta$ and ${\rm cv}$ to denote the additional conditions that the matrices in  $\mathcal{ML}^{m,k}$ are symmetric and positive definite respectively convergent. 

The reduced nonlinear Maxwell system \eqref{E:Maxw-red}, \eqref{E:IC}, \eqref{E:IFC-red-E2H3} is a special case of equation (1.7) in \cite{schnaubelt2018local} and can be written as
\begin{equation}
	\left\{
	\begin{aligned}
		\mathcal{S}(\bs{x},\bs{U}^\pm) \partial_t \bs{U}^\pm + \sum_{j=1}^2 A_j \partial_{x_j} \bs{U}^\pm &= \bs{0}, & \bs{x} &\in \R_\pm^2, t \in J, \\
		B_\Gamma \begin{pmatrix} \bs{U}^+ \\ \bs{U}^- \end{pmatrix} &= \bs{0}, & \bs{x} &\in \Gamma, t \in J, \\
		\bs{U}(0) &= \bs{U}^{(0)}, & \bs{x} &\in \R^2,
	\end{aligned}
	\right.
	\label{E:u-syst-NL}
\end{equation}
where for $\bs{v} \in \R^n$ we set
\begin{align}\label{eq:tildeS}
	\mathcal{S}(\bs{x},\bs{v})&:=\Lambda(x_1)+\epsilon_3(x_1)\theta(\bs{v}),\\
	\Lambda(x_1) &= \begin{pmatrix}
		\epsilon_1(x_1) & 0 & 0 \\
		0 & \epsilon_1(x_1) & 0\\
		0 & 0 & \upmu_0
	\end{pmatrix}, \quad \theta(\bs{v}) :=
	\begin{pmatrix}
		3v_1^2+v_2^2& 2v_1v_2 & 0 \\
		2v_1v_2 & v_1^2+3v_2^2 & 0\\
		0 & 0 & 0
	\end{pmatrix}.\notag
\end{align}

\begin{defi}[Solution of the Nonlinear Hyperbolic Problem]
	By a solution of \eqref{E:u-syst-NL} we mean a function $\bs{U} \in \mathcal{G}^1(\R^2\times J)^3 \cap L^\infty(\R^2 \times J)^3$ with $\overline{\operatorname{im} \bs{U}^\pm} \subset \Omega_\pm$ that satisfies
	\begin{align*}
		\mathcal{S}(\bs{x},\bs{U}) \partial_t \bs{U} + \sum_{j=1}^2 A_j \partial_{x_j} \bs{U} &= 0,
	\end{align*}
	for almost all $\bs{x} \in \R^2\setminus \Gamma$ and for all $t \in J$,
	$\operatorname{Tr}_{\Gamma\times J} \left(B_\Gamma \left(\bs{U}^+,\bs{U}^-\right)^\top\right) = \bs{0}$, and $\bs{U}(0) = \bs{U}^{(0)}$.
\end{defi}
\begin{remark}
	Note that a solution $\bs{U}$ of \eqref{E:u-syst-NL} in $\mathcal{G}^3(\R^2\times J)^3$ is a classical solution of \eqref{E:u-syst-NL} because of the Sobolev embedding $H^{3}(\R^2_{\pm}) \hookrightarrow C^{1}(\R^2_\pm)$. 
\end{remark}

The following local existence result for the general quasilinear system \eqref{E:u-syst-NL} follows from Proposition 7.1 and Theorem 7.1 of \cite{schnaubelt2018local}. It requires that the coefficient $\mathcal{S}$ of the quasilinear term lies in $\mathcal{ML}_{\eta,{\rm cv}}^{3,3}(\R^2,\Omega_\pm)$. This will be checked for the Maxwell system in Section \ref{S:pf-main}.
\begin{them}[Nonlinear Existence Result]
	\label{T:NL-exist} 
	Let $\eta>0$, $\Omega_\pm \subset \R^3$, and $\mathcal{S} \in \mathcal{ML}_{\eta,{\rm cv}}^{3,3}(\R^2,\Omega_\pm)$. Assume that $\bs{U}^{(0)} \in \mathcal{H}^3(\R^2)^3$ satisfies the nonlinear compatibility conditions of order $3$, see \eqref{nonlinear_compatibility_conditions}, and $\operatorname{im} \bs{U}^{(0),\pm} \subset \Omega_\pm$ with
	\begin{equation}\label{eq:kappa}
		\operatorname{dist} \left(\overline{\operatorname{im} \bs{U}^{(0),\pm}}, \partial \Omega_\pm \right) > \kappa
	\end{equation}
	for some $\kappa > 0$.	Then the following statements are true.
	\begin{enumerate}[label=\roman*)]
		\item There exists a unique solution $\bs{U} \in \mathcal{G}^3(\R^2 \times (0,t_M))^3$ of \eqref{E:u-syst-NL}, where $t_M > 0$ is the maximal existence time.
		\item If the maximal existence time is finite, then $\liminf_{t \nearrow t_M} \dist \left(\overline{\operatorname{im} \bs{U}^\pm(\cdot,t)}, \partial \Omega_\pm \right) = 0$ or $\lim_{t \nearrow t_M} \norm{\bs{U}(\cdot,t)}_{\mathcal{H}^3(\R^2)^3} = \infty$.
		\item Let $t^* \in (0,t_M)$. Then there is a constant $C > 0$ such that
		\begin{align*}
			\norm{\bs{U}}_{\mathcal{G}^3(\R^2 \times (0,t^*))^3} \leq C \norm{\bs{U}^{(0)}}_{\mathcal{H}^3(\R^2)^3}.
		\end{align*}
	\end{enumerate}
\end{them}
\begin{remark}
	As already explained in Remark~\ref{rem:spitz}, the results of \cite{schnaubelt2018local} treat a somewhat different but more
	difficult situation. Moreover, the above result does not contain the full local well-posedness and a refined blow-up condition 
	shown in \cite{schnaubelt2018local}. So Theorem~\ref{T:NL-exist} follows from Theorem~\ref{T:lin-exist} by rather standard 
	arguments, compare the proof of Theorem~6.1 of \cite{schnaubelt2018local}.
\end{remark}

\subsection*{Compatibility conditions}
For a smooth solution $\bs{U} \in \mathcal{G}^3(\R^2\times J)^3$ of \eqref{E:u-syst-NL} with $J:=(0,t_M)$ we can differentiate \eqref{E:u-syst-NL} two times in time and get new equations that are still satisfied for all $t \in J$. By continuity these new equations have to be satisfied at $t=0$. This gives us necessary conditions on the initial values for $\bs{U} \in \mathcal{G}^3(\R^2\times J)^3$.

If $\mathcal{S}(\bs{U})$ is positive definite, then $\mathcal{S}(\bs{U})$ is invertible and system \eqref{E:u-syst-NL} implies
\begin{align}
	\partial_t \bs{U} &= -\mathcal{S}(\bs{U})^{-1}\sum_{j=1}^2 A_j \partial_{x_j} \bs{U}=: \widetilde{\bs{V}}^{(1)}(\bs{U}),\label{E:function_first_cc}\\
	\llbracket U_2 \rrbracket &= \llbracket U_3 \rrbracket = 0.\notag
\end{align}
Differentiation in time gives us the following new equations:
\begin{equation}
	\partial_t^2 \bs{U} = -\mathcal{S}(\bs{U})^{-1}\bigg(\sum_{j=1}^2 A_j \partial_{x_j} \partial_t \bs{U} + \partial_t\mathcal{S}(\bs{U}) \partial_t \bs{U} \bigg) =: \widetilde{\bs{V}}^{(2)}(\bs{U},\partial_t \bs{U}), \label{E:function_second_cc}
\end{equation}	
\begin{equation*}
\llbracket \partial_t U_2 \rrbracket = \llbracket \partial_t U_3 \rrbracket = 0,
\end{equation*}
\begin{equation}
	\begin{aligned}
	\partial_t^3 \bs{U} &= -\mathcal{S}(\bs{U})^{-1}\bigg(\sum_{j=1}^2 A_j \partial_{x_j} \partial_t^2 \bs{U} + 2 \partial_t \mathcal{S}(\bs{U}) \partial_t^2 \bs{U} 
	+ \partial_t^2 \mathcal{S}(\bs{U})\partial_t \bs{U} \bigg) =: \widetilde{\bs{V}}^{(3)}(\bs{U},\partial_t \bs{U},\partial_t^2 \bs{U}), \label{E:function_third_cc}
	\end{aligned}
		\end{equation}
\begin{equation*}		
	\llbracket \partial_t^2 U_2 \rrbracket = \llbracket \partial_t^2 U_3 \rrbracket = 0.
\end{equation*}
We can now iteratively define $\bs{V}^{(1)}(\bs{U}) := \widetilde{\bs{V}}^{(1)}(\bs{U})$, $\bs{V}^{(2)}(\bs{U}) := \widetilde{\bs{V}}^{(2)}\left(\bs{U}, \bs{V}^{(1)}(\bs{U})\right)$ and $\bs{V}^{(3)}(\bs{U}) := \widetilde{\bs{V}}^{(3)}\left(\bs{U}, \bs{V}^{(1)}(\bs{U}),\bs{V}^{(2)}(\bs{U})\right)$ to get operators $\bs{V}^{(j)}$ that only contain space derivatives and no time derivatives. For the sake of completeness we also define $\bs{V}^{(0)}(\bs{U}) := \bs{U}$. 
The equations above imply that 
\begin{align*}
	\partial_t^j \bs{U}(\cdot,0) &= \bs{V}^{(j)}(\bs{U}(\cdot,0)) \quad\text{ and } \quad
	\left\llbracket V_2^{(j-1)}(\bs{U}(\cdot,0)) \right\rrbracket = \left\llbracket V_3^{(j-1)}(\bs{U}(\cdot,0)) \right\rrbracket = 0
\end{align*}
for $j \in \{1,2,3\}$. Hence, the initial values $\bs{U}^{(0)}$ have to satisfy the necessary conditions
\begin{equation}
	\left\llbracket V_2^{(j)}\left(\bs{U}^{(0)}\right) \right\rrbracket = \left\llbracket V_3^{(j)}\left(\bs{U}^{(0)}\right) \right\rrbracket = 0
	\label{nonlinear_compatibility_conditions}
\end{equation}
for $j \in \{0,1,2\}$. Note that for higher regularity additional compatibility conditions are necessary, but we will focus our analysis on solutions in $\mathcal{G}^3(\R^2\times J)^3$. 
\begin{defi}[Nonlinear Compatibility Conditions]\label{D:compat}
	Let $m \in \{1,2,3\}$. We say that an initial value $\bs{U}^{(0)} \in \mathcal{H}^m\left(\R^2\right)^3$ satisfies the nonlinear compatibility conditions of order $m$ for \eqref{E:u-syst-NL} if \eqref{nonlinear_compatibility_conditions} is true for $j \in \{0,\dots,m-1\}$. 
\end{defi}
\begin{remark}
	The compatibility conditions for the linear problem \eqref{E:u-syst-lin} can be derived analogously. In comparison to \eqref{E:function_first_cc}, \eqref{E:function_second_cc} and \eqref{E:function_third_cc} we have to replace $\mathcal{S}(\bs{U})$ by $A_t$ and include the additional terms $M(\bs{x},t)\bs{U}$, $\bs{f}$ and their time derivatives: 
	\begin{align*}
		\widetilde{\bs{V}}^{(1)}_{\mathrm{lin}}(\bs{U}) &= -A_t^{-1}\bigg( \sum_{j=1}^2 A_j \partial_{x_j} \bs{U} + M \bs{U} -\bs{f}\bigg),\\
		\widetilde{\bs{V}}^{(2)}_{\mathrm{lin}}(\bs{U},\partial_t \bs{U}) &= -A_t^{-1}\bigg( \sum_{j=1}^2 A_j \partial_{x_j} \partial_t \bs{U} + \partial_t A_t \partial_t \bs{U} + \partial_t\left(M \bs{U}\right) - \partial_t \bs{f}\bigg),\\
		\widetilde{\bs{V}}^{(3)}_{\mathrm{lin}}(\bs{U},\partial_t \bs{U},\partial_t^2 \bs{U}) &= -A_t^{-1}\bigg( \sum_{j=1}^2 A_j \partial_{x_j} \partial_t^2 \bs{U} + \partial_t^2 A_t \partial_t \bs{U} +2\partial_t A_t \partial_t^2 \bs{U} + \partial_t^2\left(M \bs{U}\right) - \partial_t^2 \bs{f}\bigg).
	\end{align*} 
\end{remark}


\section{Proof of Theorem \ref{T:main}}\label{S:pf-main}

Let $a>1$ and $\Uext$ as in \eqref{E:Uext}. We start by expressing the equation for the error 
$\eps^a\bs{R}:=\bs{U}-\Uext.$
Substituting 
\beq\label{E:Uform}
\bs{U}=\Uext+\eps^a\bs{R}
\eeq 
in \eqref{E:u-syst-NL}, one obtains
\begin{equation}
	\left\{
	\begin{aligned}
		S(\bs{x},t,\bs{R}^\pm) \partial_t \bs{R}^\pm + \sum_{j=1}^2 A_j \partial_{x_j} \bs{R}^\pm + W(\bs{x},t,\bs{R}^\pm)\bs{R}^\pm &= -\eps^{-a} \Res, & \bs{x} &\in \R_\pm^2, t \in J, \\
		B_\Gamma \begin{pmatrix} \bs{R}^+ \\ \bs{R}^- \end{pmatrix} &= \bs{0}, & \bs{x} &\in \Gamma, t \in J, \\
		\bs{R}(0) &= \bs{R}^{(0)}, & \bs{x} &\in \R^2,
	\end{aligned}
	\right.
	\label{E:R-syst}
\end{equation}
where we set $\bs{R}^{(0)}:=\eps^{-a}(\bs{U}^{(0)}-\Uext(\cdot,0))$ and, recalling \eqref{eq:tildeS},
\begin{align}
	\Res &:= \Res(\Uext) = \mathcal{S}(\cdot,\Uext)\pa_t\Uext + \sum_{j=1}^2 A_j \partial_{x_j}\Uext, \label{E: def_Res_Uext}\\
	S(\bs{x},t,\bs{R})&:=\mathcal{S}(\bs{x},\Uext(\bs{x},t)+\eps^a \bs{R})
	= \Lambda(\bs{x})+ \epsilon_3(x_1)\eps^{2a}\theta(\bs{R})+\varphi(\bs{R}), \label{E: def_S}\\
	\varphi(\bs{R}) &:= \epsilon_3(x_1)\eps^{a} 
	\begin{pmatrix} 
		6 U_{\text{ext},1}R_1 +2U_{\text{ext},2}R_2 & 2U_{\text{ext},1}R_2+2U_{\text{ext},2}R_1 & 0\\
		2U_{\text{ext},1}R_2 +2U_{\text{ext},2}R_1 & 2 U_{\text{ext},1}R_1 +6U_{\text{ext},2}R_2 &0 \\
		0 & 0&0
	\end{pmatrix} + \epsilon_3(x_1)\theta(\Uext), \nonumber\\
	W(\bs{x},t,\bs{R})\bs{R}&:=\eps^{-a}\big(S(\bs{x},t,\bs{R})-\mathcal{S}(x,\Uext(\bs{x},t))\big)\pa_t \Uext(\bs{x},t). \label{E: def_W}
\end{align}
One can check that the entries of the $3 \times 3$ matrix $W$ are
\begin{align*}
	W_{1,1}&=\epsilon_3 \left(\pa_t(3\Uextone^2+\Uexttwo^2)+3 \eps^a R_1\pa_t\Uextone+2 \eps^a R_2\pa_t\Uexttwo\right), \\ 
	W_{1,2}&= \epsilon_3 (\eps^a R_2\pa_t\Uextone+ 2 \pa_t(\Uextone\Uexttwo)), \\ 
	W_{2,1} &= \epsilon_3(\eps^a R_1\pa_t\Uexttwo+2 \pa_t(\Uextone\Uexttwo)),\\
	W_{2,2}&= \epsilon_3 \left(\pa_t(\Uextone^2+3\Uexttwo^2)+3 \eps^a R_2\pa_t\Uexttwo+2 \eps^a R_1\pa_t\Uextone\right),\\
	W_{1,3}&= W_{2,3}=W_{3,1}=W_{3,2}=W_{3,3}=0.
\end{align*}
The interface condition of \eqref{E:R-syst} is a consequence of $B_\Gamma(\Uext^+,\Uext^-)^\top=\bs{0}$, as explained in Remark \ref{R:Uext-div-IFC}.

For a fixed $\Uext$, systems \eqref{E:u-syst-NL} and \eqref{E:R-syst} are, of course, equivalent provided \eqref{E:Uform} holds. Our rough strategy is to use the local existence Theorem \ref{T:NL-exist} for \eqref{E:u-syst-NL} in order to get the existence of $\bs{R}$ on the time interval $(0,t_M)$ and then apply a bootstrapping argument on \eqref{E:R-syst} to show that $t_M=\mathcal{O}(\eps^{-2})$ and that the desired estimate
\begin{equation*}
    \norm{\bs{R}}_{\cG^3(\R^2\times (0,T_0\eps^{-2}))^3} = \varepsilon^{-a} \norm{\bs{U} - \Uext}_{\cG^3(\R^2\times (0,T_0\eps^{-2}))^3} \leq C
\end{equation*}
holds for all small enough $\eps>0$.

For the application of Theorem \ref{T:NL-exist} to \eqref{E:u-syst-NL} we need to find $\Omega_\pm \subset \R^3$ such that $\mathcal{S} \in \mathcal{ML}_{\eta,{\rm cv}}^{3,3}(\R^2,\Omega_\pm)$. For this, firstly, $\epsilon_1^\pm, \epsilon_3^\pm \in C^3(\R_\pm) \cap W^{3,\infty}(\R_\pm)$ is needed and $\epsilon_1, \epsilon_3$ have to converge for $|x_1| \rightarrow \infty$, which we have assumed in \eqref{E:ass-eps1}--\eqref{E:ass-eps3-conv}. 
Secondly, the symmetric matrix $\mathcal{S}(\bs{x},\bs{v})$ has to be positive definite for all $\bs{x}\in \R^2_\pm$ and $\bs{v}\in\Omega_\pm$, respectively.
It is easy to verify that $\mathcal{S}(\bs{v})$ has the three eigenvalues $\lambda_1 = \upmu_0$, $\lambda_2 = \epsilon_1 + \epsilon_3 \left( v_1^2 + v_2^2\right)$, and $\lambda_3 = \epsilon_1 + 3\epsilon_3 \left( v_1^2 + v_2^2\right)$. We now have to check when $\lambda_{1,2,3}\geq \eta > 0$. 

Recall the bounds on $\epsilon_1$ and $\epsilon_3$ in \eqref{E:ass-eps1} and \eqref{E:ass-eps3}. If $\epsilon_{3,m}^\pm \geq 0$, then clearly $\lambda_{2,3}>\epsilon_{1,m}^\pm$ and the choice $\eta:=\min\{\upmu_0;\epsilon_{1,m}^+;\epsilon_{1,m}^-\}$ and $\Omega_\pm :=\R^3$ is possible (and \eqref{eq:kappa} trivially holds). If $\epsilon_{3,m}^\pm < 0$, we impose
$$\epsilon_{1,m}^\pm+3\epsilon_{3,m}^\pm (v_1^2+v_2^2)>\eta>0 \ \text{ for all } \bs{v}\in \Omega_\pm.$$
Choosing $\eta\in \left(0,\min\{\upmu_0;\epsilon_{1,m}^+;\epsilon_{1,m}^-\}\right)$ and
\[
\Omega_\pm:=\begin{cases}
	\left\{\bs{v}\in\R^3 \mv v_1^2+v_2^2 < \frac{\eta-\epsilon_{1,m}^\pm}{3\epsilon_{3,m}^\pm} \right\}, & \epsilon_{3,m}^\pm < 0, \\
	\R^3, & \epsilon_{3,m}^\pm\geq 0,
\end{cases}
\]
we infer $\mathcal{S} \in \mathcal{ML}_{\eta,{\rm cv}}^{3,3}(\R^2,\Omega_\pm)$.

We now take a solution $A \in \bigcap_{k=0}^4 C^{4-k}([0,T_0],H^{2+k}(\R))$ of the effective nonlinear Schr\"odinger equation \eqref{E:NLS} for some $T_0>0$. Choose $\bs{R}^{(0)}\in \cH^3(\R^2)^3$ and $\eps_*>0$ small enough such that $\bs{U}^{(0)}:=\Uext(\cdot,0)+\eps_*^a\bs{R}^{(0)}$ satisfies $\overline{\operatorname{im} \bs{U}^{(0),\pm}} \subset \Omega_\pm$ (which implies \eqref{eq:kappa}) and the nonlinear compatibility conditions of order $3$, see Definition \ref{D:compat}. Then the local existence Theorem~\ref{T:NL-exist} yields a maximal existence time $t_M>0$ and a solution $\bs{U}\in \cG^3(\R^2\times (0,t_M))^3$ of \eqref{E:u-syst-NL}. For $t\in [0,t_M)$ we set 
$$z(t):=\sum_{k=0}^3\|\pa_t^k\bs{R}(\cdot,t)\|^2_{\cH^{3-k}(\R^2)^3}.$$
We have $\|\bs{R}(\cdot,t)\|_{L^\infty(\R^2)^3}\le c_Sz(t)^{1/2}$ for a constant $c_S\ge1$. For the application of Theorem \ref{T:NL-exist} we need that the values of the second argument of $\mathcal{S}$, i.e., $\Uext + \eps^a \bs{R}$, remain in $\Omega_\pm$. For this we choose $\varpi>0$ with
	\beq\label{def:varpi} 
	\varpi^2< \frac1{c_S^2}\min\left\{\frac{\eta-\epsilon_{1,m}^-}{3\min\{\epsilon_{3,m}^-;0\}}; \frac{\eta-\epsilon_{1,m}^+}{3\min\{\epsilon_{3,m}^+;0\}}\right\},
	\eeq
	where $\frac{c}{0}:=+\infty$.
The major part of the rest of the proof of Theorem \ref{T:main} is a bootstrapping argument to prove the statement
\begin{equation}
	\left\{
	\begin{aligned}
		&\exists \, 1\ge \rho>\rho_0>0 \ \exists \, \eps_0 = \eps_0(\rho) \in (0, \eps_*) \ \exists \, t^* \in (0,t_M) 
		\text{ such that for all } \eps\in (0,\eps_0) \\
		&\text{we have } \eps^a\rho + \|\Uext\|_{L^\infty(\R^2\times [0,t^*])} \le \varpi
		\text{ \ and if \ } z(0)\leq \rho_0^2 \text{ and } t^* \leq T_0\eps^{-2},\\
		&\text{then \ } z(t)\leq \rho^2 \ \text{ for all } t\in [0,t^*].
	\end{aligned}
	\right.
	\label{E:bootstrap-aim}
\end{equation}

Note that the first condition in \eqref{E:bootstrap-aim}, i.e., the smallness of $\eps^a\rho +\|\Uext\|_{L^\infty(\R^2\times [0,T_0 \eps^{-2}))^3} $, can be achieved by simply choosing $\eps_0=\eps_0(\rho)$ small enough. Together with the second condition, i.e., $z(t) \leq \rho^2$, these inequalities guarantee $\Uext(\cdot,t) + \eps^a \bs{R}(\cdot,t) \in \Omega_\pm$ for all $t\in [0,t^*]$.

To establish \eqref{E:bootstrap-aim}, we define for $1\ge\rho>\rho_0>0$
\begin{align}\label{def:Trho}
	&\begin{aligned}
	\Trho :=\sup \Big\{&t^* \ge0 : \eps^a\rho + \|\Uext\|_{L^\infty(\R^2\times [0,t^*])} \le \varpi, z(t)\leq \rho^2 \text{ for all } t\in [0,t^*),\, z(0)=\rho_0^2 \Big\}, \\
	\end{aligned}\\
	&\Jr :=[0, \Trho).\notag
\end{align}
On the time interval $\Jr$ the conditions
\begin{equation*}
	\forall\, t\in \Jr: \quad \operatorname{dist} \left(\overline{\operatorname{im} \bs{U}^{\pm}(\cdot,t)}, \partial \Omega_\pm \right) > \kappa> 0, \quad \|\bs{U}(\cdot,t)\|_{\cH^3(\R^2)^3}\le C< \infty 
\end{equation*}
are guaranteed.

We will prove in Section~\ref{S:bootstrap} that $z(t)\le \rho^2/2$ for $t\in \Jr$ and suitable $\eps_0$ and $\rho_0$ and hence $\Trho \geq T_0\eps^{-2}$
and \eqref{E:bootstrap-aim} is true. This together with the definition of $\bs{R}$ yields the estimate 
\beq\label{E:U-Uext-est}
\norm{\bs{U} - \Uext}_{\cG^3(\R^2\times (0,T_0\eps^{-2}))^3} \leq \rho \eps^{a}.
\eeq
Finally, to obtain \eqref{E:main-est} from Theorem \ref{T:main}
, it will only remain to show 
$$\norm{\Uans - \Uext}_{\cG^3(\R^2\times (0,T_0\eps^{-2}))^3}\leq C \eps^{a},$$ which is straightforward, see Section~\ref{sec:final}.

\subsection{Bootstrapping argument (proof of \eqref{E:bootstrap-aim})}\label{S:bootstrap}
We use the multi-index $\bs{\beta} := (\beta_1,\beta_2,\beta_t)^\top \in \N_0^3$, $|\bs{\beta}|\leq 3,$ and the abbreviation $\bs{r}_{\bs{\beta}} := \partial^{\bs{\beta}} \bs{R}$. Applying $\partial^{\bs{\beta}} = \partial_{x_1}^{\beta_1}\partial_{x_2}^{\beta_2}\partial_t^{\beta_t}$ to \eqref{E:R-syst} yields
\begin{equation}
	\left\{
	\begin{aligned}
		S(\bs{x},t,&\bs{R}) \partial_t \bs{r}_{\bs{\beta}} + \sum_{j=1}^2 A_j \partial_{x_j} \bs{r}_{\bs{\beta}} \\
		&= \bs{s}_{\bs{\beta}}(\bs{x},t,\bs{R}) + \bs{w}_{\bs{\beta}}(\bs{x},t,\bs{R}) - \varepsilon^{-a} \partial^{\bs{\beta}} \Res(\bs{x},t), & \bs{x} &\in \R^2\setminus\Gamma,t \in \Jr, \\
		\bs{r}_{\bs{\beta}}(\bs{x},0) &= \bs{r}^{(0)}_{\bs{\beta}}(\bs{x}) := \partial^{\bs{\beta}} \bs{R}(\bs{x},0), & \bs{x} &\in \R^2
	\end{aligned}\right.
	\label{E:rbeta-syst}
\end{equation}
with 
\begin{align*}
	\bs{s}_{\bs{\beta}}(\bs{x},t,\bs{R}) &:= - \sum_{\bs{0} < \bs{\gamma} \leq \bs{\beta}} \binom{\bs{\beta}}{\bs{\gamma}} \partial^{\bs{\gamma}} S(\bs{x},t,\bs{R}) \partial^{\bs{\beta} - \bs{\gamma}} \partial_t \bs{R},\\
	\bs{w}_{\bs{\beta}}(\bs{x},t,\bs{R}) &:= -\partial^{\bs{\beta}} \left( W(\bs{x},t,\bs{R}) \bs{R}\right).
\end{align*}
The time derivatives $\partial_t^k \bs{R}(\cdot,0)$ have to be interpreted as right-sided derivatives that satisfy
\begin{equation}
	\partial_t^j \bs{R}(\cdot,0) = \bs{V}^{(j)}(\bs{R}(\cdot,0)) = \bs{V}^{(j)}\left(\bs{R}^{(0)}\right),
	\label{E:R0-tderiv}
\end{equation}
with $\bs{V}^{(j)}$ as defined in Section \ref{S:exist}. Testing \eqref{E:rbeta-syst} with $\bs{r}_{\bs{\beta}}$ produces
\begin{equation}
	\begin{aligned}
		&\int_0^t \int_{\R^2} \Big(S(\bs{R}) \partial_t \bs{r}_{\bs{\beta}} \cdot \bs{r}_{\bs{\beta}} + \sum_{j=1}^2 A_j \partial_{x_j} \bs{r}_{\bs{\beta}} \cdot \bs{r}_{\bs{\beta}}\Big)\,\mathrm{d}\bs{x}\,\mathrm{d}s\\ 
		&= \int_0^t \int_{\R^2}\left( \bs{w}_{\bs{\beta}}(\bs{R}) \cdot \bs{r}_{\bs{\beta}} + \bs{s}_{\bs{\beta}}(\bs{R})\cdot \bs{r}_{\bs{\beta}}\right)\,\mathrm{d}\bs{x}\,\mathrm{d}s
		- \varepsilon^{-a} \int_0^t \int_{\R^2} \partial^{\bs{\beta}} \Res\cdot \bs{r}_{\bs{\beta}} \,\mathrm{d}\bs{x}\,\mathrm{d}s.
	\end{aligned}
	\label{E:rbeta-tested}
\end{equation}
The main steps of our bootstrapping argument are:
\begin{enumerate}[label=\Roman*.]
	\item Employ \eqref{E:R0-tderiv} to estimate $\norm{\bs{r}_{\bs{\beta}}^{(0)}}_{L^2(\R^2)^3}$ for all $\bs{\beta} \in \N_0^3$ with $|\bs{\beta}| \leq 3$.
	\item Based on \eqref{E:rbeta-tested}, estimate $\sum_{|\bs{\gamma}|\leq 3, \gamma_1=0}\|\pa^{\bs{\gamma}} \bs{R}(\cdot,t)\|_{L^2(\R^2)^3}^2$ using that $\int_{\R^2} \sum_{j=1}^2 A_j \partial_{x_j} \bs{r}_{\bs{\beta}} \cdot \bs{r}_{\bs{\beta}}\,\mathrm{d}\bs{x}=0$ if $\beta_1=0$.
	\item Rewrite \eqref{E:rbeta-syst} to analyze $\partial^{\bs{\beta}} R_2$ and $\partial^{\bs{\beta}} R_3$ for $\beta_1 = 1$, and then iterate the process for $\beta_1 = 2$ and $\beta_1 = 3$.
	\item Use $\nabla \cdot \partial_t \bs{\mathcal{D}}(\bs{U}_E) = 0$ to estimate $\partial^{\bs{\beta}} R_1$ for $\beta_1 = 1$, where we start with $\bs{\beta} = (1,0,0)^\top$ and then iterate to increase $\beta_t$ and $\beta_2$. Finally we have to iterate the process again for $\beta_1 = 2$ and $\beta_1 = 3$.
\end{enumerate}
Note that due to the interface at $x_1 = 0$ the interface conditions cannot be simply differentiated for all $\beta_1\neq0$, therefore the method of Step II cannot be used to estimate $x_1$-derivatives of $\bs{R}$ and Step III and IV are necessary.

Our basic strategy follows the proof of the local a priori estimates in \cite{schnaubelt2018local}. The main difference is that, using the structure of our ansatz,
we can derive the estimates on a large time interval $[0, T_0\eps^{-2})$ with the desired dependence on $\eps$. We let $t \in \Jr$.

\subsubsection*{Step I: Estimates of the initial values}
\label{S:est-IC}
In this section we estimate $\norm{\bs{r}_{\bs{\beta}}(\cdot,0)}_{L^2(\R^2)^3}$ for all $\bs{\beta} \in \N_0^3$ with $|\bs{\beta}|\leq 3$. For $\beta_t= 0$ we already have by the assumption $z(0)\leq \rho_0^2$ in \eqref{E:bootstrap-aim} that 
$$\norm{\bs{R}(\cdot,0)}_{\mathcal{H}^3(\R^2)^3} = \norm{\bs{R}^{(0)}}_{\mathcal{H}^3(\R^2)^3} < \rho_0.$$
If $\beta_t \neq 0$ we will use \eqref{E:R0-tderiv} to estimate $\partial_t^p \bs{R}(\cdot,0)$ in $\mathcal{H}^{3-p}(\R^2)^3$ for $p \in \{1,2,3\}$. 

Since $\bs{U}^{(0)}$ satisfies the nonlinear compatibility conditions of order $3$ we know from Section \ref{S:exist} that (suppressing the $\bs{x}$-dependence)
\begin{equation*}
	\partial_t^j \bs{U}(t) = \bs{V}^{(j)}\left( \bs{U}(t) \right)
\end{equation*}
for all $t \in [0,t_M)$ and $j \in \{0,1,2\}$. With $\bs{U} = \varepsilon^a \bs{R} + \Uext$, as in \eqref{E:R-syst} we rewrite these three equations as 
\begin{equation}
	\begin{aligned}
		\partial_t \bs{R} &= -\mathcal{S}(\bs{U})^{-1}\left(\sum_{j=1}^2 A_j \partial_{x_j} \bs{R} + \varepsilon^{-a} \Res+ \varepsilon^{-a} \left(\mathcal{S}(\bs{U}) - \mathcal{S}(\bs{U}_{{\rm ext}})\right)\partial_t \bs{U}_{{\rm ext}}\right) , \\
		\partial_t^2 \bs{R} &= -\mathcal{S}(\bs{U})^{-1}\left(\sum_{j=1}^2 A_j \partial_{x_j} \partial_t \bs{R} + \partial_t\mathcal{S}(\bs{U}) \partial_t \bs{R} + \varepsilon^{-a} \partial_t \Res\right) \\
		&\quad \ -\varepsilon^{-a} \mathcal{S}(\bs{U})^{-1}\left(\partial_t \left(\mathcal{S}(\bs{U}) - \mathcal{S}(\bs{U}_{{\rm ext}})\right)\partial_t \bs{U}_{{\rm ext}} + \left(\mathcal{S}(\bs{U}) - \mathcal{S}(\bs{U}_{{\rm ext}})\right)\partial_t^2 \bs{U}_{{\rm ext}} \right), \\
		\partial_t^3 \bs{R} &= -\mathcal{S}(\bs{U})^{-1}\!\left(\!\sum_{j=1}^2 A_j \partial_{x_j} \partial_t^2 \bs{R} + 2 \partial_t \left(\mathcal{S}(\bs{U})\right) \partial_t^2 \bs{R} + \partial_t^2 \left(\mathcal{S}(\bs{U})\right)\partial_t \bs{R} + \varepsilon^{-a} \partial_t^2 \Res \!\right)\\
		&\begin{aligned}\quad \ -\varepsilon^{-a} \mathcal{S}(\bs{U})^{-1} &\left(\partial_t^2 \left(\mathcal{S}(\bs{U}) - \mathcal{S}(\bs{U}_{{\rm ext}})\right)\partial_t \bs{U}_{{\rm ext}} + 2 \partial_t \left(\mathcal{S}(\bs{U}) - \mathcal{S}(\bs{U}_{{\rm ext}})\right)\partial_t^2 \bs{U}_{{\rm ext}}\right. \\
			& \left.\ + \left(\mathcal{S}(\bs{U}) - \mathcal{S}(\bs{U}_{{\rm ext}})\right)\partial_t^3 \bs{U}_{{\rm ext}}\right). 
		\end{aligned}	
	\end{aligned}
	\label{E:est-IC1}
\end{equation}
The following lemma collects some properties of the matrix function $\mathcal{S}$.
Using basic properties of Sobolev spaces and the mean value theorem, the lemma can be shown as Lemma~2.23, Lemma~7.1 and Corollary~7.2 in \cite{spitz2017local}. 
\begin{lem}\label{L:T-est}~
	Let $T', \eta_0, R >0$, $\Omega_\pm \subset \R^3$ and $\mathcal{S} \in \mathcal{ML}_{\eta_0,{\rm cv}}^{3,3}(\R^2,\Omega_\pm)$. Then for all $\bs{U}, \bs{V} \in B_R(0) \subset \mathcal{G}^3(\R^2\times [0,T'])^3$ with $\overline{\operatorname{im} \bs{U}^{\pm}}, \overline{\operatorname{im} \bs{V}^{\pm}} \subset \Omega_\pm$ there exists $C>0$ such that
	\begin{enumerate}[label=\roman*)]
		\item $\norm{\mathcal{S}(\bs{U}(t))^{-1}}_{\mathcal{W}^{2,\infty}(\R^2)^{3 \times 3}+\mathcal{H}^2(\R^2)^{3 \times 3}}\le C,$
		\item $\norm{\partial_t^k \mathcal{S}(\bs{U}(t))}_{ \mathcal{W}^{3-k,\infty}(\R^2)^{3 \times 3}+\mathcal{H}^{3-k}(\R^2)^{3 \times 3}}\le C$,
		\item $\norm{\partial_t^k \left(\mathcal{S}(\bs{U}(t)) - \mathcal{S}(\bs{V}(t)) \right)}_{\mathcal{H}^{2-k}(\R^2)^{3 \times 3}} \leq C \sum\limits_{j=0}^k \norm{\partial_t^j \bs{U}(t) - \partial_t^j \bs{V}(t)}_{\mathcal{H}^{2-k}(\R^2)^3}$ 
	\end{enumerate}
	for all $k \in \{0,1,2\}$ and $t \in [0,T']$.
\end{lem}

We can now go back to equations \eqref{E:est-IC1} and use Lemma \ref{L:products}, Lemma \ref{L:T-est} and $\varepsilon^a \bs{R} = \bs{U} - \bs{U}_{{\rm ext}}$ to show
\begin{align*}
	&\norm{\partial_t \bs{R}(\cdot,0)}_{\mathcal{H}^2(\R^2)^3} \\
	&\ \leq C\norm{\mathcal{S}(\bs{U}(0))^{-1}}_{\mathcal{W}^{2,\infty}(\R^2)^{3 \times 3} + \mathcal{H}^2(\R^2)^{3 \times 3}} \left( \norm{\bs{R}(\cdot,0)}_{\mathcal{H}^3(\R^2)^3} + \varepsilon^{-a} \norm{\Res(\cdot,0)}_{\mathcal{H}^2(\R^2)^3}\right)\\
	&\qquad \ + C \varepsilon^{-a} \norm{\varepsilon^a \bs{R}(\cdot,0)}_{\mathcal{H}^2(\R^2)^3}\norm{\partial_t \bs{U}_{{\rm ext}}(\cdot,0)}_{\mathcal{H}^2(\R^2)^3}\\
	&\ \leq C \left(\norm{\bs{R}^{(0)}}_{\mathcal{H}^3(\R^2)^3} + \varepsilon^{-a} \norm{\Res(\cdot,0)}_{\mathcal{H}^{2}(\R^2)^3}\right).
\end{align*}
The remaining two estimates follow analogously: 
\begin{align*}
	\norm{\partial_t^2 \bs{R}(\cdot,0)}_{\mathcal{H}^1(\R^2)^3} &\leq C \left(\norm{\bs{R}(\cdot,0)}_{\mathcal{H}^3(\R^2)^3} + \norm{\partial_t \bs{R}(\cdot,0)}_{\mathcal{H}^2(\R^2)^3} + \varepsilon^{-a} \norm{\partial_t \Res(\cdot,0)}_{\mathcal{H}^1(\R^2)^3}\right), \\
	\norm{\partial_t^3 \bs{R}(\cdot,0)}_{L^2(\R^2)^3} &\leq C \left(\norm{\bs{R}(\cdot,0)}_{\mathcal{H}^3(\R^2)^3} + \norm{\partial_t \bs{R}(\cdot,0)}_{\mathcal{H}^2(\R^2)^3}+ \norm{\partial_t^2 \bs{R}(\cdot,0)}_{\mathcal{H}^1(\R^2)^3} \right.\\
	&\qquad \left. + \varepsilon^{-a} \norm{\partial_t^2 \Res(\cdot,0)}_{L^2(\R^2)^3}\right). 
\end{align*}

Finally we use the recursive structure of the estimates and obtain
\begin{align*}
	\norm{\partial_t^p \bs{R}(\cdot,0)}_{\mathcal{H}^{3-p}(\R^2)^3} &\leq C \bigg(\norm{\bs{R}^{(0)}}_{\mathcal{H}^3(\R^2)^3} + \varepsilon^{-a} \sum_{j=0}^{p-1}\norm{\partial_t^j \Res(\cdot,0)}_{\mathcal{H}^{2-j}(\R^2)^3}\bigg)
\end{align*}
for all $p \in \{1,2,3\}$. With our estimate for the residual, see \eqref{E:Res-est}, we infer
\begin{equation}
	\norm{\bs{r}^{(0)}_{\bs{\beta}}}_{L^2(\R^2)^3} \leq C \left(\rho_0 + \varepsilon^{\frac{7}{2} -a}\right)
	\label{E:est-IC2}
\end{equation}
for all $\bs{\beta} \in \N_0^3$ with $|\bs{\beta}|\leq 3$.

\subsubsection*{Step II: Analysis of $\bs{\partial^{\beta} R}$ for $\bs{|\beta|\leq 3, \beta_1 = 0}$}
We first show an energy estimate for the $t$- and $x_2$-derivatives of $\bs{R}$.
\begin{lem}
	Let $\bs{R} \in \mathcal{G}^3(\R^2\times \Jr)^3$ be a solution of \eqref{E:R-syst}, let $\bs{\beta} \in \N_0^3, |\bs{\beta}|\leq 3, \beta_1=0$ and recall that $S$ as defined in \eqref{E: def_S} is positive definite with ellipticity constant $\eta$. Then $\bs{r}_{\bs{\beta}} = \partial^{\bs{\beta}}\bs{R}$ satisfies 
	\begin{equation}
		\begin{aligned}
			\frac{\eta}{2} \norm{\bs{r}_{\bs{\beta}}(\cdot,t)}_{L^2(\R^2)^3}^2 &\leq C \norm{\bs{r}^{(0)}_{\bs{\beta}}}_{L^2(\R^2)^3}^2+ \int_0^t \int_{\R^2} \Big(\bs{w}_{\bs{\beta}}(\bs{R}) \cdot \bs{r}_{\bs{\beta}} + \bs{s}_{\bs{\beta}}(\bs{R}) \cdot \bs{r}_{\bs{\beta}}  \\
			& \quad \ + \frac{1}{2} \partial_t S(\bs{R}) \bs{r}_{\bs{\beta}} \cdot \bs{r}_{\bs{\beta}} - \varepsilon^{-a} \partial^{\bs{\beta}} \Res \cdot \bs{r}_{\bs{\beta}}\Big)\, \mathrm{d}\bs{x}\,\mathrm{d}s
		\end{aligned} 
		\label{E:est_deriv_error}
	\end{equation}
	for every $t \in \Jr$. 
\end{lem}
\begin{proof}
	Step 1. Let us first study the case $|\bs{\beta}|<3.$ Since $\bs{R}\in \mathcal{G}^3(\R^2\times \Jr)^3$, we have $\bs{r}_{\bs{\beta}} = \partial^{\bs{\beta}} \bs{R} \in \mathcal{G}^1(\R^2\times \Jr)^3$.
	To employ \eqref{E:rbeta-tested}, we use that $S$ is symmetric and compute
	\begin{align*}
		\int_0^t \int_{\R^2} S(\bs{R}) \partial_t \bs{r}_{\bs{\beta}} \cdot \bs{r}_{\bs{\beta}}\,\mathrm{d}\bs{x}\,\mathrm{d}s &= \frac{1}{2}\int_0^t \partial_t \left(\int_{\R^2} S(\bs{R}) \bs{r}_{\bs{\beta}} \cdot \bs{r}_{\bs{\beta}}\,\mathrm{d}\bs{x}\right)\,\mathrm{d}s - \frac{1}{2} \int_0^t \int_{\R^2} \partial_t S(\bs{R}) \bs{r}_{\bs{\beta}} \cdot \bs{r}_{\bs{\beta}}\,\mathrm{d}\bs{x}\,\mathrm{d}s.
	\end{align*}
	Using that $S(\bs{R})$ is positive definite, we estimate $\int_{\R^2} S(\bs{R})(t) \bs{r}_{\bs{\beta}}(t) \cdot \bs{r}_{\bs{\beta}}(t)\,\mathrm{d}\bs{x}\geq \eta \norm{\bs{r}_{\bs{\beta}}(\cdot,t)}_{L^2(\R^2)^3}^2$. Moreover, we have $$\int_{\R^2} S(\bs{R})(0) \bs{r}_{\bs{\beta}}(0) \cdot \bs{r}_{\bs{\beta}}(0)\,\mathrm{d}\bs{x}\leq\norm{S\left(\bs{R}^{(0)}\right)}_{L^\infty(\R^2)^{3\times 3}} \norm{\bs{r}^{(0)}_{\bs{\beta}}}_{L^2(\R^2)^3}^2.$$ Since $\bs{R}^{(0)}, \bs{U}_{{\rm ext}}^{(0)} \in L^\infty(\R^2)^3$, this leads to
	\begin{align*}
		\frac{\eta}{2} \norm{\bs{r}_{\bs{\beta}}(\cdot,t)}_{L^2(\R^2)^3}^2 &\leq C 	\norm{\bs{r}^{(0)}_{\bs{\beta}}}_{L^2(\R^2)^3}^2 + \int_0^t \int_{\R^2} S(\bs{R}) \partial_t \bs{r}_{\bs{\beta}} \cdot \bs{r}_{\bs{\beta}}\,\mathrm{d}\bs{x}\,\mathrm{d}s \\
		& \quad +\frac{1}{2} \int_0^t \int_{\R^2} \partial_t S(\bs{R})\bs{r}_{\bs{\beta}} \cdot 	\bs{r}_{\bs{\beta}}\,\mathrm{d}\bs{x}\,\mathrm{d}s .
	\end{align*}
	An integration by parts yields
	\begin{align*}
		&\int_{\R^2} \sum_{j=1}^2 A_j \partial_{x_j} \bs{r}_{\bs{\beta}}\cdot \bs{r}_{\bs{\beta}}\,\mathrm{d}\bs{x} = \int_{\R^2} \left(-\partial_{x_2} r_{\bs{\beta},3} r_{\bs{\beta},1} +\partial_{x_1} r_{\bs{\beta},3} r_{\bs{\beta},2} - \partial_{x_2} r_{\bs{\beta},1} r_{\bs{\beta},3} + \partial_{x_1} r_{\bs{\beta},2} r_{\bs{\beta},3}\right) \, \mathrm{d}\bs{x} = 0,
	\end{align*}
	employing differentiated interface conditions
	\begin{equation}\label{E:interface_cond_rbeta}
		\llbracket r_{\bs{\beta},2} \rrbracket = \llbracket r_{\bs{\beta},3} \rrbracket = 0
	\end{equation}
	in the $x_1$-integral. The interface conditions can be differentiated since $\beta_1 = 0$. Now \eqref{E:est_deriv_error}
	is a consequence of \eqref{E:rbeta-tested} and the above formulas.
	
	Step 2. Next, we consider the remaining case $|\bs{\beta}| = 3, \beta_1=0$. Let $\bs{f} := \bs{s}_{\bs{\beta}}(\bs{R}) + \bs{w}_{\bs{\beta}}(\bs{R}) - \varepsilon^{-a} \partial^{\bs{\beta}} \Res$. The differential equation in \eqref{E:rbeta-syst} becomes
	\begin{equation}
		S(\bs{R}) \partial_t \bs{r}_{\bs{\beta}} + \sum_{j=1}^2 A_j \partial_{x_j} \bs{r}_{\bs{\beta}} = \bs{f}, \quad \bs{x} \in \R^2\setminus\Gamma, \ t \in \Jr.
		\label{E:appr-prob}
	\end{equation}
	Since $S(\bs{R}) \in \mathcal{F}^{3,3}_{\eta, {\rm cv}}(\R^2 \times \Jr)$ and $\bs{f} \in \mathcal{G}^0(\R^2 \times \Jr)^3$ (which can easily be shown), we can apply Lemma \ref{L:approx-arg} to \eqref{E:appr-prob} (setting $A_t := S(\bs{R})$ and $M:=0$). Because $\bs{r}_{\bs{\beta}}$ is a weak solution of \eqref{E:appr-prob} with the initial conditions $\bs{r}_{\bs{\beta}}(\bs{x},0) = \partial^{\bs{\beta}} \bs{R}(\bs{x},0)$ and the interface conditions \eqref{E:interface_cond_rbeta}, the lemma provides sequences $\left(\bs{r}_{\bs{\beta},n}^{(0)}\right)_n \subset \mathcal{D}(\R^2)^3$, $(\bs{f}_n)_n \subset \mathcal{H}^1(\R^2 \times \Jr)^3$ and $(\bs{r}_{\bs{\beta},n})_n \subset \mathcal{G}^1(\R^2\times \Jr)^3$ with $\bs{r}^{(0)}_{\bs{\beta},n} \rightarrow \bs{r}^{(0)}_{\bs{\beta}}$ in $L^2(\R^2)^3$, $\bs{f}_n \rightarrow \bs{f}$ in $L^2(\R^2 \times \Jr)^3$ and $\bs{r}_{\bs{\beta},n} \rightarrow \bs{r}_{\bs{\beta}}$ in $\mathcal{G}^0(\R^2\times \Jr)^3$ for $n \rightarrow \infty$, and $\bs{r}_{\bs{\beta},n}$ is a weak solution of \eqref{E:appr-prob} with data $\left(\bs{f}_n,\bs{r}^{(0)}_{\bs{\beta},n}\right)$ for all $n \in \N$.
	The same calculation as in Step 1 shows that 
	\begin{align*}
		&\frac{\eta}{2} \norm{\bs{r}_{\bs{\beta},n}(\cdot,t)}_{L^2(\R^2)^3}^2 \leq C \norm{\bs{r}^{(0)}_{\bs{\beta},n}}_{L^2(\R^2)^3}^2 + \int_0^t \int_{\R^2} \left(\bs{f}_n \cdot \bs{r}_{\bs{\beta},n} + \frac{1}{2} \partial_t S(\bs{R}) \bs{r}_{\bs{\beta},n} \cdot \bs{r}_{\bs{\beta},n}\, \mathrm{d}\bs{x}\right)\,\mathrm{d}s . 
	\end{align*}
	The Cauchy-Schwarz inequality then implies the statement. 
\end{proof}

We now have to estimate each part of the right-hand side in \eqref{E:est_deriv_error}. Let
$$\widetilde{z}(t):=\sum_{\stackrel{\bs{\gamma}\in \N_0^3}{|\bs{\gamma}|\leq 3, \gamma_1=0}}\|\pa^{\bs{\gamma}}\bs{R}(\cdot,t)\|^2_{L^2(\R^2)^3}.$$
From Step I we know that
\begin{equation*}
	\norm{\bs{r}^{(0)}_{\bs{\beta}}}_{L^2(\R^2)^3} \leq C \left(\rho + \varepsilon^{\frac{7}{2} -a}\right).
\end{equation*}
For $\int_0^t \int_{\R^2} \frac{1}{2} \partial_t S(\bs{R}) \bs{r}_{\bs{\beta}} \cdot \bs{r}_{\bs{\beta}}\,\mathrm{d}\bs{x}\,\mathrm{d}s$, recalling \eqref{E: def_S}, we first have
\begin{equation*}
	\varepsilon^{2a} \epsilon_3 \int_0^t \int_{\R^2} \frac{1}{2} \partial_t \theta(\bs{R}) \bs{r}_{\bs{\beta}} \cdot \bs{r}_{\bs{\beta}}\, \mathrm{d}\bs{x}\,\mathrm{d}s \leq C \varepsilon^{2a} \int_0^t \norm{\bs{r}_{\bs{\beta}}(\cdot,s)}_{L^2(\R^2)^3}^2\,\mathrm{d}s \leq C \varepsilon^{2a} \int_0^t \widetilde{z}(s)\,\mathrm{d}s
\end{equation*}
since $\bs{R},\partial_t \bs{R} \in L^\infty(\R^2 \times \Jr)^3$. Similarly, using also that 
\[\|\Uext\|_{ L^\infty(\R^2\times\Jr)^3}, \|\pa_t \Uext\|_{L^\infty(\R^2 \times \Jr)^3} \leq C\varepsilon,\] 
cf.\ \eqref {E:regularity_Uext}, we derive
\begin{align*}
	\int_0^t \int_{\R^2} \frac{1}{2} \partial_t \varphi(\bs{R}) \bs{r}_{\bs{\beta}} \cdot \bs{r}_{\bs{\beta}}\, \mathrm{d}\bs{x}\,\mathrm{d}s 
	&\leq C \left(\varepsilon^{2} + \varepsilon^{1+a}\right) \int_0^t \norm{\bs{r}_{\bs{\beta}}(\cdot,s)}_{L^2(\R^2)^3}^2\,\mathrm{d}s \\
	&\leq C \varepsilon^{2}\int_0^t \widetilde{z}(s)\,\mathrm{d}s.
\end{align*}
For the residual term, \eqref{E:Res-est} yields $\norm{\partial^{\bs{\beta}} \Res(\cdot,t)}_{L^2(\R^2)^3} \leq C \varepsilon^{7/2}$. The Cauchy-Schwarz inequality, $t \leq T_0 \varepsilon^{-2}$ and $\norm{\bs{r}_{\bs{\beta}}(\cdot,t)}_{L^2(\R^2)^3}^2 \leq z(t) \leq \rho\le 1$ for $t \in \Jr$ then give us
\begin{equation}
\begin{aligned}
	\int_0^t \int_{\R^2} \varepsilon^{-a} \partial^{\bs{\beta}} \Res \cdot \bs{r}_{\bs{\beta}}\, \mathrm{d}\bs{x}\,\mathrm{d}s 
	&\leq \int_0^t \varepsilon^{-a} \norm{\partial^\beta \Res(\cdot,s)}_{L^2(\R^2)^3} \norm{\bs{r}_{\bs{\beta}}(\cdot,s)}_{L^2(\R^2)^3}\,\mathrm{d}s\\
	& \leq C \varepsilon^{\frac{3}{2}-a}.
\end{aligned} 
\label{E:estimate_integral_residual}
\end{equation}

The remaining terms $\int_0^t \int_{\R^2}\left(\bs{w}_{\bs{\beta}} \cdot \bs{r}_{\bs{\beta}} + \bs{s}_{\bs{\beta}} \cdot \bs{r}_{\bs{\beta}}\right) \, \mathrm{d}\bs{x}\,\mathrm{d}s$ mainly consist of integrals of the type
\begin{equation}
	I_1 := \int_0^t \int_{\R^2} \partial^{\bs{a}} f(\bs{x},s) \partial^{\bs{b}} g(\bs{x},s) \partial^{\bs{c}}h(\bs{x},s)k(\bs{x},s)\,\mathrm{d}\bs{x}\,\mathrm{d}s,
	\label{general_quartic_term}
\end{equation} 
where $f,g,h \in \mathcal{G}^3(\R^2\times J)$, $k \in \mathcal{G}^0(\R^2\times J)$ and $\bs{a},\bs{b}, \bs{c} \in \N_0^3$ with $|\bs{a}|,|\bs{b}|,|\bs{c}| < 4$ and $s := |\bs{a}| + |\bs{b}| + |\bs{c}| \leq 4$. For $s = 4$ we only have integrals
where at least one time-derivative is present, i.e., $a_t = b_t = c_t = 0$ is not possible. The case where four derivatives fall an a component of $\Uext$ also appears and will be discussed separately. The following two classes of estimates are needed.
\begin{enumerate}[label=\roman*)]
	\item \underline{$|\bs{a}| \leq 3$, $|\bs{b}| \leq 1$, $|\bs{c}| \leq 1$:} Here $\partial^{\bs{a}} f(\cdot,t), k(\cdot,t)\in L^2(\R^2)$ and $\partial^{\bs{b}} g(\cdot,t), \partial^{\bs{c}} h(\cdot,t) \in L^\infty(\R^2)$. With the Cauchy-Schwarz inequality we obtain
	\begin{equation*}
		\hspace{-0.6cm}I_1 \leq C \int_0^t \norm{\partial^{\bs{a}}f(\cdot,s)}_{L^2(\R^2)} \norm{\partial^{\bs{b}}g(\cdot,s)}_{L^\infty(\R^2)} \norm{\partial^{\bs{c}}h(\cdot,s)}_{L^\infty(\R^2)} \norm{k(\cdot,s)}_{L^2(\R^2)}\,\mathrm{d}s.
	\end{equation*}
	\item \underline{$|\bs{a}| \leq 2$, $|\bs{b}| \leq 2$, $|\bs{c}| = 0$:} Now $k(\cdot,t)\in L^2(\R^2)$, $\partial^{\bs{c}} h(\cdot,t)\in L^\infty(\R^2)$, and $\partial^{\bs{a}} f(\cdot,t)$,   $\partial^{\bs{b}} g(\cdot,t)\in L^p(\R^2)$ for all $p \in [1,\infty)$. This follows from the Sobolev embedding $H^1(\R_\pm^2) \hookrightarrow L^p(\R_\pm^2)$ for all $1 \leq p<\infty$. 
	H\"older's inequality then yields
	\begin{equation*}
	\hspace{-0.8cm}	I_1 \leq C \int_0^t \norm{\partial^{\bs{a}}f(\cdot,s)}_{L^3(\R^2)} \norm{\partial^{\bs{b}}g(\cdot,s)}_{L^6(\R^2)} \norm{\partial^{\bs{c}}h(\cdot,s)}_{L^\infty(\R^2)} \norm{k(\cdot,s)}_{L^2(\R^2)}\,\mathrm{d}s.
	\end{equation*}
\end{enumerate}
The role of the function $k$ in \eqref{general_quartic_term} will always be played by a component of $\bs{r}_{\bs{\beta}} = \partial^{\bs{\beta}} \bs{R}$.

Recall that $S(\bs{R})=\Lambda + \epsilon_3\eps^{2a}\theta(\bs{R})+\varphi(\bs{R})$. Hence, to estimate $\int_0^t \int_{\R^2}\bs{s}_{\bs{\beta}} \cdot \bs{r}_{\bs{\beta}}\, \mathrm{d}\bs{x}\,\mathrm{d}s$, we first analyze $\eps^{2a}\epsilon_3 \partial^{\bs{\gamma}} \theta(\bs{R}) \partial^{\bs{\beta} - \bs{\gamma}} \partial_t \bs{R} \cdot \bs{r}_{\bs{\beta}}$ where $\beta_1=\gamma_1=0$. 
This sum consists of terms of the form
\begin{equation*}
	C \varepsilon^{2a} \partial^{\bs{\gamma}'}R_i \partial^{\bs{\gamma}''}R_j \partial^{\bs{\beta} - \bs{\gamma}} \partial_t R_k \partial^{\bs{\beta}} R_l
\end{equation*}
with $\bs{\gamma} = \bs{\gamma}' + \bs{\gamma}''$, $\gamma'_1=\gamma''_1=0$, and $i,j,k,l \in \{1,2\}$. The case i) above
applies if $|\bs{\beta}-\bs{\gamma}|=0$, where we may take $|\bs{\gamma}''| \leq 1$. We then estimate
\begin{align*}
	C \varepsilon^{2a} \int_0^t \int_{\R^2} \partial^{\bs{\gamma}'}R_i \partial^{\bs{\gamma}''}R_j \partial^{\bs{\beta} - \bs{\gamma}} \partial_t R_k \partial^{\bs{\beta}} R_l \,\mathrm{d}\bs{x}\mathrm{d}s &\leq C \varepsilon^{2a} \int_0^t \norm{\partial^{\bs{\beta}}R_l(\cdot,s)}_{L^2(\R^2)}\norm{\partial^{\bs{\gamma}'}R_i(\cdot,s)}_{L^2(\R^2)}\,\mathrm{d}s \\
	&\leq C \varepsilon^{2a}\int_0^t \widetilde{z}(s)\,\mathrm{d}s.
\end{align*}
A representative of type ii) is any term with $|\bs{\beta}-\bs{\gamma}| = 1$, $|\bs{\gamma}'|= 2$ and $|\bs{\gamma}''|= 0$, which is estimated via 
\begin{align*}
	C \varepsilon^{2a} &\int_0^t \int_{\R^2} \partial^{\bs{\gamma}'}R_i \partial^{\bs{\gamma}''}R_j \partial^{\bs{\beta} - \bs{\gamma}} \partial_t R_k \partial^{\bs{\beta}} R_l \,\mathrm{d}\bs{x}\,\mathrm{d}s \\
	&\leq C \varepsilon^{2a} \int_0^t \norm{\partial^{\bs{\beta}}R_l(\cdot,s)}_{L^2(\R^2)}\norm{\partial^{\bs{\gamma}'}R_i(\cdot,s)}_{L^6(\R^2)} \norm{\partial^{\bs{\beta} - \bs{\gamma}} \partial_t R_k(\cdot,s)}_{L^3(\R^2)}\,\mathrm{d}s \\
	&\leq C \varepsilon^{2a} \int_0^t \! \Big(\norm{\partial^{\bs{\beta}}R_l(\cdot,s)}_{L^2(\R^2)}^2 + \norm{\partial^{\bs{\gamma}'}R_i(\cdot,s)}_{L^6(\R^2)}^2 \norm{\partial^{\bs{\beta} - \bs{\gamma}} \partial_t R_k(\cdot,s)}_{L^3(\R^2)}^2\Big)\mathrm{d}s \\
	&\leq C \varepsilon^{2a}\int_0^t \left(\widetilde{z}(s) + \left(z(s)\right)^2\right)\,\mathrm{d}s \leq C \varepsilon^{2a} \int_0^t \widetilde{z}(s)\,\mathrm{d}s + C \rho^2\varepsilon^{2a}t,
\end{align*}
using again $H^1(\R_\pm^2) \hookrightarrow L^p(\R_\pm^2)$ for $1 \leq p<\infty$. The remaining cases can be treated similarly.

Next, we study $\partial^{\bs{\gamma}} \varphi(\bs{R}) \partial^{\bs{\beta}-\bs{\gamma}} \partial_t \bs{R} \cdot \partial^{\bs{\beta}}\bs{R}$ with $|\bs{\beta}-\bs{\gamma}|\leq 2$ and $\beta_1=\gamma_1=0$. We again use \eqref {E:regularity_Uext}
which provides the inequality $\norm{\pa^{\bs{\alpha}} \bs{U}_{{\rm ext}}}_{L^\infty(\R^2\times \Jr)^3} \leq C \varepsilon$ for all $|\bs{\alpha}|\leq 3$ with 
$\alpha_1\le2$. For terms quadratic in $\bs{U}_{{\rm ext}}$ case~i) applies:
\newpage
\begin{align*}
	\int_0^t \int_{\R^2} \!\partial^{\bs{\gamma}} (U_{{\rm ext},i}U_{{\rm ext},j}) \partial^{\bs{\beta} - \bs{\gamma}} \partial_t R_k \partial^{\bs{\beta}}R_l\,\mathrm{d}\bs{x}\,\mathrm{d}s &\leq C \varepsilon^2 \!\int_0^t \!\norm{\partial^{\bs{\beta}}R_l(\cdot,s)}_{L^2(\R^2)}\norm{\partial^{\bs{\beta} - \bs{\gamma}} \partial_t R_k(\cdot,s)}_{L^2(\R^2)}\mathrm{d}s \\
	&\leq C \varepsilon^2 \int_0^t \widetilde{z}(s)\,\mathrm{d}s.
\end{align*}
For terms linear in $\bs{U}_{{\rm ext}}$, i.e., $I_2:=\varepsilon^{a} \int_0^t \int_{\R^2} \partial^{\bs{\gamma}'}U_{{\rm ext},i} \partial^{\bs{\gamma}''}R_j \partial^{\bs{\beta} - \bs{\gamma}} \partial_t R_k \partial^{\bs{\beta}} R_l \,\mathrm{d}\bs{x}\,\mathrm{d}s$, we distinguish the three cases $|\bs{\beta} - \bs{\gamma}| = 0,1,$ and $2$. 
For $|\bs{\beta} - \bs{\gamma}| = 0$ we compute
\begin{align*}
	I_2 &\leq C \varepsilon^{1 + a} \int_0^t \norm{\partial^{\bs{\beta}}R_l(\cdot,s)}_{L^2(\R^2)}\norm{\partial^{\bs{\gamma}''}R_j(\cdot,s)}_{L^2(\R^2)}\,\mathrm{d}s
	\leq C \varepsilon^{2a}\int_0^t \widetilde{z}(s)\,\mathrm{d}s
\end{align*}
by means of the estimate type i) and the fact that $\pa_tR_k \in L^\infty(\R^2\times \Jr)$. For $|\bs{\beta} - \bs{\gamma}| = 1$ the estimate of type ii) applies and we have 
\begin{align*}
	\begin{aligned}
		I_2&\leq C \varepsilon^{1+a} \int_0^t \norm{\partial^{\bs{\beta}}R_l(\cdot,s)}_{L^2(\R^2)}\norm{\partial^{\bs{\gamma}''}R_j(\cdot,s)}_{L^6(\R^2)} \norm{\partial^{\bs{\beta} - \bs{\gamma}} \partial_t R_k(\cdot,s)}_{L^3(\R^2)}\,\mathrm{d}s \\
		&\leq C \varepsilon^{1+a}\int_0^t \widetilde{z}(s)\,\mathrm{d}s + C \varepsilon^{1+a} t
	\end{aligned}
\end{align*}
as $\partial^{\bs{\gamma}''}R_j, \partial^{\bs{\beta} - \bs{\gamma}} \partial_t R_k \in L^\infty(\Jr,\mathcal{H}^1(\R^2))$. Finally, for $|\bs{\beta} - \bs{\gamma}| = 2$ case i) again yields
\begin{align*}
	I_2&\leq C \varepsilon^{1+a} \int_0^t \norm{\partial^{\bs{\beta}-\bs{\gamma}}\partial_t R_k(\cdot,s)}_{L^2(\R^2)}\norm{\partial^{\bs{\beta}}R_l(\cdot,s)}_{L^2(\R^2)} \,\mathrm{d}s 
	\leq C \varepsilon^{1+a}\int_0^t \widetilde{z}(s)\,\mathrm{d}s
\end{align*}
where we have used $\pa^{\bs{\gamma}''}R_j\in L^\infty(\R^2\times \Jr)$ because $|\bs{\gamma}''|\leq 1$.	

At last, we treat $\partial^{\bs{\beta}} \left(W(\bs{R})\bs{R}\right) \cdot \partial^{\bs{\beta}}\bs{R}$. 
Terms quadratic in $\bs{U}_{{\rm ext}}$ are estimated as follows, where $\bs{\beta} = \bs{\beta}' + \bs{\beta}''$.
If $|\bs{\beta}'|<3$ or if not all three derivatives fall on $\partial_t U_{{\rm ext},j}$, we obtain 
\begin{align*}
	I_3 &:=\int_0^t \!\!\int_{\R^2}\!\! \partial^{\bs{\beta}'} (U_{{\rm ext},i} \partial_t U_{{\rm ext},j})\partial^{\bs{\beta}''} R_k \partial^{\bs{\beta}}R_l\,\mathrm{d}\bs{x}\,\mathrm{d}s\\
 &\leq C \varepsilon^2 \!\!\int_0^t \norm{\partial^{\bs{\beta}}R_l(\cdot,s)}_{L^2(\R^2)}\norm{\partial^{\bs{\beta}''} R_k(\cdot,s)}_{L^2(\R^2)} \mathrm{d}s\\
	&\leq C \varepsilon^2 \!\!\int_0^t \widetilde{z}(s)\,\mathrm{d}s
\end{align*}
as $\|\pa^{\bs{\beta}'} \left(U_{{\rm ext},i} \partial_t U_{{\rm ext},j}\right)\|_{L^\infty(\R^2 \times \Jr)}\leq C\eps^2$ by \eqref {E:regularity_Uext}. If $|\bs{\beta}'|=3$ and $\partial^{\bs{\beta}'}$ is only applied to $\partial_t U_{{\rm ext},j}$, we use \eqref {E:regularity_Uext3} with  $\partial^{\bs{\beta}'}\partial_t U_{{\rm ext},j}=\mathcal{A}_j+\mathcal{B}_j$.
Sobolev's embedding for $x_2\mapsto R_k(x_1,x_2,s)$ implies that
\begin{align}\label{eq:I3}
	I_3&\le \left|\int_0^t \!\!\int_{\R^2} U_{{\rm ext},i} \mathcal{A}_j\partial^{\bs{\beta}''} R_k \partial^{\bs{\beta}}R_l\,\mathrm{d}\bs{x}\,\mathrm{d}s\right| + \left|\int_0^t \!\!\int_{\R^2} U_{{\rm ext},i} \mathcal{B}_j \partial^{\bs{\beta}''} R_k \partial^{\bs{\beta}}R_l\,\mathrm{d}\bs{x}\,\mathrm{d}s\right|\\
	&\le C \varepsilon^2 \!\!\int_0^t \widetilde{z}(s)\,\mathrm{d}s + \int_0^t\|\partial^{\bs{\beta}}R_l(\cdot,s)\|_{L^2(\R^2)}
	\left[\int_{\R^2} |U_{{\rm ext},i} \mathcal{B}_j R_k|^2 \,\mathrm{d}\bs{x} \right]^\frac12 \,\mathrm{d} s\notag\\
	&\leq C \varepsilon^2 \!\!\int_0^t \widetilde{z}(s)\,\mathrm{d}s + C\eps^2\int_0^t\|\partial^{\bs{\beta}}R_l(\cdot,s)\|_{L^2(\R^2)}
	\left[\int_{\R} \sup_{x_2\in\R} |R_k(x_1,x_2,s)|^2\,\mathrm{d} x_1\right]^\frac12 \,\mathrm{d} s\notag\\
	&\le C \varepsilon^2 \!\int_0^t \widetilde{z}(s)\,\mathrm{d}s + C\eps^2\!\int_0^t\!\|\partial^{\bs{\beta}}R_l(\cdot,s)\|_{L^2(\R^2)} 
	\left[\int_{\R^2}\! \big(|R_k(x_1,x_2,s)|^2 + |\partial_{x_2}R_k(x_1,x_2,s)|^2 \big)\,\mathrm{d}\bs{x}\right]^\frac12 \!\mathrm{d} s\notag\\
	&\le C\eps^2\int_0^t \widetilde{z}(s)\,\mathrm{d}s \notag.
\end{align}
In the same way, we treat terms linear in $\bs{U}_{{\rm ext}}$. Let $\bs{\beta} = \bs{\beta}' + \bs{\beta}'' + \bs{\beta}'''$ and w.l.o.g.\ $|\bs{\beta}'''|\leq 1$.
If $|\bs{\beta'}|\le 2$, it follows
\begin{align*}
	I_4&:=\varepsilon^{a}\int_0^t \int_{\R^2} \partial^{\bs{\beta}'}\pa_t U_{{\rm ext},i} \partial^{\bs{\beta}''}R_j \partial^{\bs{\beta}'''}R_k \partial^{\bs{\beta}}R_l \,\mathrm{d}\bs{x}\,\mathrm{d}s \\
	&\leq C \varepsilon^{1 + a} \int_0^t \norm{\partial^{\bs{\beta}}R_l(\cdot,s)}_{L^2(\R^2)}\norm{\partial^{\bs{\beta}''}R_j(\cdot,s)}_{L^2(\R^2)}\,\mathrm{d}s\\
	&\leq C \varepsilon^{1 + a}\int_0^t \widetilde{z}(s)\,\mathrm{d}s,
\end{align*}
using $\pa^{\bs{\beta}'''}R_k \in L^\infty(\R^2\times \Jr)$ and \eqref {E:regularity_Uext}. If $|\bs{\beta'}|=3$, as above we estimate
\begin{align*}
	I_4 &\leq \eps^a \left|\int_0^t \!\!\int_{\R^2} \mathcal{A}_i R_j R_k \partial^{\bs{\beta}}R_l\,\mathrm{d}\bs{x}\,\mathrm{d}s\right| + \eps^a \left|\int_0^t \!\!\int_{\R^2} \mathcal{B}_i R_j R_k \partial^{\bs{\beta}}R_l\,\mathrm{d}\bs{x}\,\mathrm{d}s\right|\\
	&\le C \varepsilon^{1 + a}\int_0^t \widetilde{z}(s)\,\mathrm{d}s + C\varepsilon^{a}\int_0^t \|\partial^{\bs{\beta}}R_l(\cdot,s)\|_{L^2(\R^2)}
	\left(\int_{\R^2} | \mathcal{B}_i R_j R_k|^2 \,\mathrm{d}\bs{x}\right)^\frac12\,\mathrm{d}s\\
	&\le C \varepsilon^{1 + a}\int_0^t \widetilde{z}(s)\,\mathrm{d}s + C\eps^{1+a}\int_0^t\|\partial^{\bs{\beta}}R_l(\cdot,s)\|_{L^2(\R^2)}
	\left(\int_{\R} \sup_{x_2\in\R} |R_k(x_1,x_2,s)|^2\,\mathrm{d} x_1\right)^\frac12\,\mathrm{d} s\\ 
	&\leq C \varepsilon^{1 + a}\int_0^t \widetilde{z}(s)\,\mathrm{d}s.
\end{align*}
Note that these are the only cases where four derivatives can fall on one function in this step.

Collecting the above partial estimates, we finally get in \eqref{E:est_deriv_error} 
\begin{equation*}
	\widetilde{z}(t) \leq C \left(\rho_0^2 + \varepsilon^2 \int_0^t \widetilde{z}(s) \,\mathrm{d}s + \varepsilon^{1+a} t+ \varepsilon^{\frac{3}{2} -a}+\varepsilon^{7-2a}\right).
\end{equation*}
If $a\in (1,\tfrac{11}{2})$, the Gronwall's inequality yields
\begin{align*}
	\widetilde{z}(t) &\leq C \left(\rho_0^2 + \varepsilon^{\frac{3}{2} -a} + \varepsilon^{1+a}t \right)\e^{C \varepsilon^2 t}
	\leq C \left(\rho_0^2 + \varepsilon^{\frac{3}{2} -a} + \varepsilon^{a-1}\right)
\end{align*}
for all $t \in \Jr$ if $\|\bs{R}^{(0)}\|_{\mathcal{H}^3(\R^2)^3}\leq \rho_0$.

\subsubsection*{Step III: Analysis of $\bs{\partial^{\beta} R_{2,3}}$ for $\bs{|\beta|\leq 3, \beta_1\neq 0}$}

We first consider $\beta_1=1$. Setting $\bs{\alpha}:=(0,\beta_2,\beta_t)^\top$, we have $\bs{\beta} = (1,0,0)^\top+\bs{\alpha}$ and \eqref{E:rbeta-syst} implies 
\begin{equation}
	\left\{
	\begin{aligned}
		\partial^{\bs{\beta}} R_2 &= \partial_{x_2}\partial^{\bs{\alpha}} R_1 + \left(-S(\bs{R}) \partial_t \partial^{\bs{\alpha}}\bs{R} + \bs{s}_{\bs{\alpha}}(\bs{R}) + \bs{w}_{\bs{\alpha}}(\bs{R}) - \varepsilon^{-a} \partial^{\bs{\alpha}} \Res \right)_3,\\
		\partial^{\bs{\beta}} R_3 &= \left(-S(\bs{R}) \partial_t \partial^{\bs{\alpha}}\bs{R} + \bs{s}_{\bs{\alpha}}(\bs{R}) + \bs{w}_{\bs{\alpha}}(\bs{R}) - \varepsilon^{-a} \partial^{\bs{\alpha}} \Res \right)_2.
	\end{aligned}
	\right.
	\label{equations_second_third_component}
\end{equation}
Each term on the right-hand side has derivatives $\pa^{\bs{\gamma}}$ with $|\bs{\gamma}|\leq 3$ and $\gamma_1=0$ and hence can be bounded by Step II, for instance, 
$$
\begin{aligned}
	\|(S(\bs{R}) \partial_t \partial^{\bs{\alpha}}\bs{R})(\cdot,t)\|_{L^2(\R^2)^3}
&\leq \|S(\bs{R})(\cdot,t)\|_{L^\infty(\R^2)}\norm{(\partial_t \partial^{\bs{\alpha}}\bs{R})(\cdot,t)}_{L^2(\R^2)^3}\\
&\leq C \!\norm{\partial^{\bs{\gamma}}\bs{R}(\cdot,t)}_{L^2(\R^2)^3}.
\end{aligned}
$$
In summary, we get
$$\norm{\partial^{\bs{\beta}}R_{2,3}(\cdot,t)}_{L^2(\R^2)^3}^2\le C\widetilde{z}(t) \leq C \left(\rho_0^2 + \varepsilon^{\frac{3}{2} -a} + \varepsilon^{a-1}\right)$$
for all $|\bs{\beta}|\leq 3,$ $\beta_1=1$ and all $t \in \Jr$ if $a\in (1,\tfrac{11}{2})$ and $\|\bs{R}^{(0)}\|_{\mathcal{H}^3(\R^2)^3}\leq \rho_0$.

For larger values of $\beta_1$ we iterate the process. For $\beta_1=2$ we have \eqref{equations_second_third_component} with $\bs{\beta} = (1,0,0)^\top + \bs{\alpha}$ and $\bs{\alpha}:=(1,\beta_2,\beta_t)^\top$ and using the previous step, all terms in the right-hand side can be estimated in $L^2(\R^2)$. For $\beta_1=3$ the same process applies, with $\bs{\beta} = (1,0,0)^\top +\bs{\alpha}$ and $\bs{\alpha}=(2,0,0)^\top$. 
Altogether, we arrive at
\begin{equation}
\begin{aligned}
	&\sum_{\substack{|\bs{\beta}| \leq 3, \\ \beta_1=0}} \norm{\partial^{\bs{\beta}} \bs{R}(\cdot,t)}_{L^2(\R^2)^3}^2 + \sum_{|\bs{\beta}| \leq 3} \left(\norm{\partial^{\bs{\beta}} R_2(\cdot,t)}_{L^2(\R^2)}^2 + \norm{\partial^{\bs{\beta}} R_3(\cdot,t)}_{L^2(\R^2)}^2\right) \\
	& \qquad \leq C \left( \rho_0^2 + \varepsilon^{\frac{3}{2} -a} + \varepsilon^{a-1}\right)
	\label{estimate_all_easy_norms}
	\end{aligned}
\end{equation}
for all $t \in \Jr$ if $a\in (1,\tfrac{11}{2})$ and $\|\bs{R}^{(0)}\|_{\mathcal{H}^3(\R^2)^3}\leq \rho_0$.

\subsubsection*{Step IV: Analysis of $\bs{\partial^{\beta} R_1, |\beta|\leq 3, \beta_1\neq 0}$}

In this final step we exploit the divergence equation $\nabla \cdot \bs{\mathcal{D}}(\bs{U}_E) = \nabla \cdot \bs{\mathcal{D}}(\bs{U}^{(0)}_E)$ to estimate $\partial^{\bs{\beta}} R_1$. First, from the definitions \eqref{E: def_S} and \eqref{E: def_W} it follows that
\begin{align*}
    \varepsilon^{-a} \partial_t \widetilde{\bs{\mathcal{D}}}(\varepsilon^a \bs{R}_E &+ \bs{U}_{{\rm ext},E}) = \widetilde{S}(\bs{R})\partial_t \widetilde{\bs{R}} + \widetilde{W}(\bs{R}) \widetilde{\bs{R}} + \varepsilon^{-a} \partial_t \widetilde{\bs{\mathcal{D}}}(\bs{U}_{{\rm ext},E})\\
    &= \partial_t\left( \left(\epsilon_1 + \varepsilon^{2a} \epsilon_3 |\widetilde{\bs{R}}|^2\right) \widetilde{\bs{R}}\right) + \widetilde{\varphi}(\bs{R})\partial_t \widetilde{\bs{R}} + \widetilde{W}(\bs{R}) \widetilde{\bs{R}} + \varepsilon^{-a} \partial_t \widetilde{\bs{\mathcal{D}}}(\bs{U}_{{\rm ext},E}).
\end{align*}
With definition \eqref{E: def_Res_Uext} it follows now that for $\bs{\alpha}\in \N_0^3, |\bs{\alpha}|\leq 2$ and $\bs{r}_{\bs{\alpha}} = \partial^{\bs{\alpha}}\bs{R}$ we have 
\begin{equation}\label{E:Dtil-diff}
	\begin{aligned}
		\varepsilon^{-a}\partial^{\bs{\alpha}}&\partial_t \widetilde{\bs{\mathcal{D}}}(\varepsilon^a \bs{R}_E + \bs{U}_{{\rm ext},E}) = \partial_t\left( \left(\epsilon_1 + \varepsilon^{2a} \epsilon_3 |\widetilde{\bs{R}}|^2\right) \widetilde{\bs{r}}_{\bs{\alpha}}\right) + \partial^{\bs{\alpha}}\left(\widetilde{\varphi}(\bs{R})\partial_t \widetilde{\bs{R}}\right) \\
		&\quad + \partial^{\bs{\alpha}} \left( \widetilde{W}(\bs{R}) \widetilde{\bs{R}}\right)
		\ + \partial_t \left(\sum_{0 < \bs{\gamma} \leq \bs{\alpha}} \binom{\bs{\alpha}}{\bs{\gamma}} \partial^{\bs{\gamma}} \left(\epsilon_1 + \varepsilon^{2a}\epsilon_3 |\widetilde{\bs{R}}|^2\right) \partial^{\bs{\alpha} - \bs{\gamma}} \widetilde{\bs{R}}\right)\\
		&\quad  + \varepsilon^{-a} \partial^{\bs{\alpha}} \widetilde{\Res} + \varepsilon^{-a} \partial^{\bs{\alpha}} \begin{pmatrix}
			\partial_{x_2}U_{{\rm ext},3} \\ -\partial_{x_1}U_{{\rm ext},3}
		\end{pmatrix}
	\end{aligned}
\end{equation}
on $\R^2_\pm\times\Jr$, where $\widetilde{\cdot}$ of a $(3\times 3)$-matrix denotes the restriction to the upper left $(2 \times 2)$-submatrix and $\widetilde{\cdot}$ of a vector in $\R^3$ denotes the first two components of this vector. The calculation to obtain \eqref{E:Dtil-diff} uses that $\varphi(\bs{R})$ and $W(\bs{R})$ have a block structure.

An integration by parts yields
\begin{align*}
	&\int_0^t \partial^{\bs{\alpha}}\left(\widetilde{\varphi}(\bs{R})\partial_t \widetilde{\bs{R}}\right)\,\mathrm{d}s = \int_0^t \bigg(\widetilde{\varphi}(\bs{R}) \partial^{\bs{\alpha}} \partial_t \widetilde{\bs{R}} + \sum_{0 < \bs{\gamma} \leq \bs{\alpha}} \binom{\bs{\alpha}}{\bs{\gamma}} \partial^{\bs{\gamma}} \widetilde{\varphi}(\bs{R}) \partial^{\bs{\alpha} - \bs{\gamma}}\partial_t \widetilde{\bs{R}} \bigg)\,\mathrm{d}s \\
	&\quad =\int_0^t \bigg({-}\partial_t \widetilde{\varphi}(\bs{R}) \partial^{\bs{\alpha}} \widetilde{\bs{R}} + \sum_{0 < \bs{\gamma} \leq \bs{\alpha}} \binom{\bs{\alpha}}{\bs{\gamma}} \partial^{\bs{\gamma}} \widetilde{\varphi}(\bs{R}) \partial^{\bs{\alpha} - \bs{\gamma}}\partial_t \widetilde{\bs{R}} \bigg)\,\mathrm{d}s + \left[\widetilde{\varphi}(\bs{R}) \partial^{\bs{\alpha}}\widetilde{\bs{R}}\right]_0^t.
\end{align*} 
Integrating \eqref{E:Dtil-diff} in time, we then deduce
\begin{equation}\label{E:Dtil-diff-int}
	\begin{aligned}
		\big[\varepsilon^{-a}&\partial^{\bs{\alpha}} \widetilde{\bs{\mathcal{D}}}(\varepsilon^a \bs{R}_E + \bs{U}_{{\rm ext},E})\big]_0^t\\
		&= \bigg[\left(\epsilon_1 + \varepsilon^{2a} \epsilon_3 |\widetilde{\bs{R}}|^2\right) \widetilde{\bs{r}}_{\bs{\alpha}}+ \widetilde{\varphi}(\bs{R})\widetilde{\bs{r}}_{\bs{\alpha}} + \sum_{0 < \bs{\gamma} \leq \bs{\alpha}} \binom{\bs{\alpha}}{\bs{\gamma}} \partial^{\bs{\gamma}} \left(\epsilon_1 + \varepsilon^{2a}\epsilon_3 |\widetilde{\bs{R}}|^2\right) \partial^{\bs{\alpha} - \bs{\gamma}} \widetilde{\bs{R}}\bigg]_0^t\\
		&\quad \ +\int_0^t \bigg(- \partial_t \widetilde{\varphi}(\bs{R})\widetilde{\bs{r}}_{\bs{\alpha}} + \sum_{0 < \bs{\gamma} \leq \bs{\alpha}} \binom{\bs{\alpha}}{\bs{\gamma}} \partial^{\bs{\gamma}} \widetilde{\varphi}(\bs{R}) \partial^{\bs{\alpha} - \bs{\gamma}}\partial_t \widetilde{\bs{R}} + \partial^{\bs{\alpha}} \bigg( \widetilde{W}(\bs{R}) \widetilde{\bs{R}}\bigg) \bigg)\,\mathrm{d}s\\
		&\quad \ + \varepsilon^{-a} \int_0^t \left(\partial^{\bs{\alpha}} \widetilde{\Res} + \partial^{\bs{\alpha}} \begin{pmatrix}
			\partial_{x_2}U_{{\rm ext},3} \\ -\partial_{x_1}U_{{\rm ext},3}
		\end{pmatrix}\right)\,\mathrm{d}s
		.
	\end{aligned}
\end{equation}
Note that the divergence of the last term vanishes.

\smallskip

\noindent \textbf{Substep 1}: $\beta_1=1$.
We write $\bs{\beta} = (1,0,0)^\top + \bs{\alpha}$, where $\bs{\alpha}=(0,\beta_2,\beta_t)^\top$. We have that $\nabla \cdot \pa^{\bs{\alpha}}\bs{\widetilde{\mathcal{D}}}(\bs{U}_E) $ is constant in time because
$$\nabla \cdot \pa^{\bs{\alpha}}\bs{\widetilde{\mathcal{D}}}(\bs{U}_E) = \pa^{\bs{\alpha}}(\nabla \cdot \bs{\widetilde{\mathcal{D}}}(\bs{U}_E))= \pa^{\bs{\alpha}}\widetilde{\varrho}_0, \quad \widetilde{\varrho}_0:=\nabla \cdot \bs{\widetilde{\mathcal{D}}}(\bs{U}^{(0)}_E).$$
Note that  $\widetilde{\varrho}_0\in \cH^2(\R^2)$ because of the algebra property of $\cH^2(\R^2)$ and $\bs{U}^{(0)}\in \cH^3(\R^2)^3$. Hence, taking the divergence of \eqref{E:Dtil-diff-int}, the first term vanishes and we have
\begin{align*}
	&\left[ \left(\epsilon_1 + \varepsilon^{2a} \epsilon_3 |\widetilde{\bs{R}}|^2\right) (\partial_{x_1} r_{\bs{\alpha},1} + \partial_{x_2} r_{\bs{\alpha},2}) + \nabla \left(\epsilon_1 + \varepsilon^{2a} \epsilon_3 |\widetilde{\bs{R}}|^2\right) \cdot \widetilde{\bs{r}}_{\bs{\alpha}}\right]_0^t \\
	&= {-}\left[\nabla \cdot \left(\widetilde{\varphi}(\bs{R}) \partial^{\bs{\alpha}}\widetilde{\bs{R}} + \sum_{0 < \bs{\gamma} \leq \bs{\alpha}} \binom{\bs{\alpha}}{\bs{\gamma}} \partial^{\bs{\gamma}} (\epsilon_1 + \varepsilon^{2a}\epsilon_3 |\widetilde{\bs{R}}|^2) \partial^{\bs{\alpha} - \bs{\gamma}} \widetilde{\bs{R}} \right)\right]_0^t \\
	&\quad \ {-}\int_0^t \nabla \cdot \left( - \partial_t \widetilde{\varphi}(\bs{R}) \partial^{\bs{\alpha}} \widetilde{\bs{R}} + \partial^{\bs{\alpha}} \left( \widetilde{W}(\bs{R}) \widetilde{\bs{R}}\right) + \varepsilon^{-a} \partial^{\bs{\alpha}} \widetilde{\Res}\right)\,\mathrm{d}s \\
	&\quad \ {-} \int_0^t \nabla \cdot \left( \sum_{0 < \bs{\gamma} \leq \bs{\alpha}} \binom{\bs{\alpha}}{\bs{\gamma}} \partial^{\bs{\gamma}} \widetilde{\varphi}(\bs{R}) \partial^{\bs{\alpha} - \bs{\gamma}}\partial_t \widetilde{\bs{R}}\right)\,\mathrm{d}s .
\end{align*}

Because of $\widetilde{\bs{R}}\in C(\Jr,L^\infty(\R^2))^2$, there exists a number $\vartheta >0$ with 
$\vartheta \leq (\epsilon_1 + \varepsilon^{2a} \epsilon_3 |\widetilde{\bs{R}}|^2)(\bs{x},t)$ for small enough $\varepsilon$, all $t\in \Jr$ and almost all $\bs{x}\in\R^2$. Since $\partial_{x_1}\widetilde{\bs{R}}, \partial_{x_2} \widetilde{\bs{R}}\in C(\Jr,L^\infty(\R^2))^2$,
we can also estimate $$\norm{\nabla \left(\epsilon_1 + \varepsilon^{2a} \epsilon_3 |\widetilde{\bs{R}}|^2\right)(\cdot,t)}_{L^\infty(\R^2)^3} \le C, \quad \forall\, t \in \Jr.$$

These facts yield the central inequality of this step:
\begin{align}\label{E:DxR1-est}
	&\vartheta \norm{\partial_{x_1}r_{\bs{\alpha},1}(\cdot,t)}_{L^2(\R^2)} \\
	&\leq C \big(\!\norm{\partial_{x_2} r_{\bs{\alpha},2}(\cdot,t)}_{L^2(\R^2)}\!+\! \norm{\nabla \cdot \widetilde{\bs{r}}_{\bs{\alpha}}(\cdot,0)}_{L^2(\R^2)}\!+\! \norm{\widetilde{\bs{r}}_{\bs{\alpha}}(\cdot,t)}_{L^2(\R^2)^2}\!+\! \norm{\widetilde{\bs{r}}_{\bs{\alpha}}(\cdot,0)}_{L^2(\R^2)^2}\!\big) \notag\\
	&\quad \ + \norm{\left[\nabla \cdot \left(\widetilde{\varphi}(\bs{R}) \partial^{\bs{\alpha}}\widetilde{\bs{R}} + \sum_{0 < \bs{\gamma} \leq \bs{\alpha}} \binom{\bs{\alpha}}{\bs{\gamma}} \partial^{\bs{\gamma}} (\epsilon_1 + \varepsilon^{2a}\epsilon_3 |\widetilde{\bs{R}}|^2) \partial^{\bs{\alpha} - \bs{\gamma}} \widetilde{\bs{R}}\right)(\cdot,s)\right]_0^t }_{L^2(\R^2)} \notag\\
	&\quad \ + \norm{\int_0^t\nabla \cdot \left(- \partial_t \widetilde{\varphi}(\bs{R}) \partial^{\bs{\alpha}} \widetilde{\bs{R}} + \partial^{\bs{\alpha}} \left( \widetilde{W}(\bs{R}) \widetilde{\bs{R}}\right) + \varepsilon^{-a} \partial^{\bs{\alpha}} \widetilde{\Res}\right)(\cdot,s)\,\mathrm{d}s}_{L^2(\R^2)} \notag\\
	&\quad \ +\norm{\int_0^t\nabla \cdot \left( \sum_{0 < \bs{\gamma} \leq \bs{\alpha}} \binom{\bs{\alpha}}{\bs{\gamma}} \partial^{\bs{\gamma}} \widetilde{\varphi}(\bs{R}) \partial^{\bs{\alpha} - \bs{\gamma}}\partial_t \widetilde{\bs{R}}\right)(\cdot,s)\,\mathrm{d}s}_{L^2(\R^2)}\notag.
\end{align}
We next iterate over $\beta_t$ and $\beta_2$.

(i) $\bs{\alpha} = (0,0,0)^\top$. Here \eqref{E:DxR1-est} simplifies to
\begin{equation}
	\begin{aligned}
		&\vartheta \norm{\partial_{x_1}r_{\bs{\alpha},1}(\cdot,t)}_{L^2(\R^2)} \leq \varepsilon^{-a}\norm{\int_0^t \nabla \cdot \widetilde{\Res}(\cdot,s)\,\mathrm{d}s}_{L^2(\R^2)} \\
		&+ C \big(\!\norm{\partial_{x_2} r_{\bs{\alpha},2}(\cdot,t)}_{L^2(\R^2)} \!+\! \norm{\nabla \cdot \widetilde{\bs{r}}_{\bs{\alpha}}(\cdot,0)}_{L^2(\R^2)} \!+\! \norm{\widetilde{\bs{r}}_{\bs{\alpha}}(\cdot,t)}_{L^2(\R^2)^2} \!+\! \norm{\widetilde{\bs{r}}_{\bs{\alpha}}(\cdot,0)}_{L^2(\R^2)^2}\!\big) \\
		&+ \norm{\left[\nabla \cdot \left(\widetilde{\varphi}(\bs{R}) \widetilde{\bs{R}} \right)(\cdot,s)\right]_0^t }_{L^2(\R^2)} + \norm{\int_0^t \nabla \cdot \left(- \partial_t \widetilde{\varphi}(\bs{R}) \widetilde{\bs{R}} + \widetilde{W}(\bs{R}) \widetilde{\bs{R}}\right)(\cdot,s)\,\mathrm{d}s}_{L^2(\R^2)}.
	\end{aligned}
	\label{E:DxR-al-zero}
\end{equation}
The residual term on the right-hand side is bounded by $C\eps^{\frac{3}{2}-a}$ due to \eqref{E:Res-est}. The second and fourth term on the right-hand side of \eqref{E:DxR-al-zero} are estimated by \eqref{estimate_all_easy_norms} and the third and fifth term by \eqref{E:est-IC2}. In the first norm on the last line of \eqref{E:DxR-al-zero} all terms have been treated in Steps I, II or III except for those of the type
$\varepsilon^a \epsilon_3\partial_{x_1}r_{\bs{\alpha},1} R_j U_{{\rm ext},k}$ and $\epsilon_3\partial_{x_1}r_{\bs{\alpha},1} U_{{\rm ext},j} U_{{\rm ext},k}$. Using $\bs{R} \in L^\infty(\R^2\times \Jr)^3$ and $\norm{\bs{U}_{{\rm ext}}}_{L^\infty(\R^2\times \Jr)^3} \leq C \varepsilon$, we have
\begin{align*}
	\norm{\varepsilon^a \epsilon_3\left(\partial_{x_1}r_{\bs{\alpha},1} R_j U_{{\rm ext},k}\right)(\cdot,t)}_{L^2(\R^2)} &\leq C \varepsilon^{1+a}\norm{\partial_{x_1}r_{\bs{\alpha},1}(\cdot,t)}_{L^2(\R^2)}, \\
	\norm{\epsilon_3\left(\partial_{x_1}r_{\bs{\alpha},1} U_{{\rm ext},j} U_{{\rm ext},k}\right)(\cdot,t)}_{L^2(\R^2)} &\leq C \varepsilon^{2}\norm{\partial_{x_1}r_{\bs{\alpha},1}(\cdot,t)}_{L^2(\R^2)}.
\end{align*}
In the last norm of the right-hand side of \eqref{E:DxR-al-zero}, the terms which have not been estimated so far are of the type $\partial_{x_1}r_{\bs{\alpha},1} \pa_t U_{{\rm ext},j} U_{{\rm ext},k}$, $\varepsilon^a \partial_t(U_{{\rm ext},j} R_k) \partial_{x_1} r_{\bs{\alpha},1}$, and $\varepsilon^a U_{{\rm ext},j} R_k \partial_t \partial_{x_1} r_{\bs{\alpha},1}$ for $j,k\in \{1,2,3\}$. Using $\bs{R},\pa_t\bs{R} \in L^\infty(\R^2\times \Jr)^3$ and $\norm{\pa_t\bs{U}_{{\rm ext}}}_{L^\infty(\R^2\times \Jr)^3} \leq C \varepsilon$, we obtain
\begin{align*}
	\int_0^t \norm{\epsilon_3\left(\partial_{x_1}r_{\bs{\alpha},1} \pa_t U_{{\rm ext},j} U_{{\rm ext},k}\right)(\cdot,s)}_{L^2(\R^2)} \,\mathrm{d}s &\leq C \varepsilon^{2} \int_0^t \norm{\partial_{x_1}r_{\bs{\alpha},1}(\cdot,s)}_{L^2(\R^2)}\,\mathrm{d}s,\\
	\int_0^t \norm{\varepsilon^a \epsilon_3\left(\partial_{x_1}r_{\bs{\alpha},1} \pa_t (R_k U_{{\rm ext},j})\right)(\cdot,s)}_{L^2(\R^2)}\,\mathrm{d}s &\leq C \varepsilon^{1+a} \int_0^t \norm{\partial_{x_1}r_{\bs{\alpha},1}(\cdot,s)}_{L^2(\R^2)}\,\mathrm{d}s,
\end{align*}
and, integrating by parts in time, 
\begin{align*}
	&\Big\|\int_0^t \varepsilon^a\epsilon_3\left(U_{{\rm ext},j} R_k \partial_t \partial_{x_1} r_{\bs{\alpha},1} \right)(\cdot,s)\,\mathrm{d}s\Big\|_{L^2(\R^2)} \\
	&\leq \norm{\left[ \varepsilon^a\left(U_{{\rm ext},j} R_k \partial_{x_1} r_{\bs{\alpha},1} \right)(\cdot,s)\right]_0^t
		- \int_0^t \varepsilon^a\left(\partial_t(U_{{\rm ext},j} R_k) \partial_{x_1}r_{\bs{\alpha},1} \right)(\cdot,s)\,\mathrm{d}s}_{L^2(\R^2)}\\
	&\leq C\varepsilon^{1+a}\left(\norm{\partial_{x_1}r_{\bs{\alpha},1}(\cdot,t)}_{L^2(\R^2)} + \norm{\partial_{x_1}r_{\bs{\alpha},1}(\cdot,0)}_{L^2(\R^2)} + \int_0^t \norm{\partial_{x_1}r_{\bs{\alpha},1}(\cdot,s)}_{L^2(\R^2)}\,\mathrm{d}s \right) \\
	&\leq C \varepsilon^{1+a} \left( \rho_0 + \norm{\partial_{x_1}r_{\bs{\alpha},1}(\cdot,t)}_{L^2(\R^2)} + \int_0^t \norm{\partial_{x_1}r_{\bs{\alpha},1}(\cdot,s)}_{L^2(\R^2)}\,\mathrm{d}s \right).
\end{align*}

Combining the above inequalities, for $a \in (1,\tfrac{11}{2})$ and $0\leq t\leq \Trho \leq T_0 \varepsilon^{-2}$ we infer 
\begin{align*}
	\vartheta &\norm{\partial_{x_1}r_{\bs{\alpha},1}(\cdot,t)}_{L^2(\R^2)} \leq C \left(\rho_0 + \varepsilon^{\frac{1}{2}(\frac{3}{2}- a)} + \varepsilon^{\frac{1}{2}(a-1)} + \varepsilon^{2} \norm{\partial_{x_1} r_{\bs{\alpha},1}(\cdot,t)}_{L^2(\R^2)} + \varepsilon^{\frac{3}{2} -a}\right)\\
	&\quad \ + C \varepsilon^2\! \int_0^t \left(\rho_0 + \varepsilon^{\frac{1}{2}(\frac{3}{2}- a)}+ \varepsilon^{\frac{1}{2}(a-1)}\right)\,\mathrm{d}s + C\varepsilon^2 \!\int_0^t \norm{\partial_{x_1} r_{\bs{\alpha},1}(\cdot,s)}_{L^2(\R^2)}\,\mathrm{d}s.
\end{align*}
For $\varepsilon$ small enough and $a \in [\tfrac{5}{4},\tfrac{11}{2})$ (so that $\tfrac{3}{2}- a\leq a-1$) it follows
\begin{align*}
	\norm{\partial_{x_1}r_{\bs{\alpha},1}(\cdot,t)}_{L^2(\R^2)} &\leq C \left(\rho_0 + \varepsilon^{\frac{1}{2}(\frac{3}{2}- a)} + \varepsilon^2 \int_0^t \norm{\partial_{x_1} r_{\bs{\alpha},1}(\cdot,s)}_{L^2(\R^2)}\,\mathrm{d}s\right).
\end{align*}		
Finally, Gronwall's inequality yields
\begin{align*}
	\norm{\partial_{x_1}r_{\bs{\alpha},1}(\cdot,t)}_{L^2(\R^2)} 	& \leq C \left( \rho_0 + \varepsilon^{\frac{1}{2}(\frac{3}{2}- a)} \right) \e^{C\varepsilon^2 t}\leq C \left(\rho_0 + \varepsilon^{\frac{1}{2}(\frac{3}{2}- a)} \right).
\end{align*}

(ii) We iterate the process from (i) for higher $\alpha_2=\beta_2$ and $\alpha_t = \beta_t$ (keeping $\beta_1=1$). For instance, the following sequence of $\bs{\alpha}$'s can be chosen: $\bs{\alpha}=(0,1,0)^\top$, $(0,0,1)^\top$, $(0,2,0)^\top$, $(0,0,2)^\top$, $(0,1,1)^\top$.
Note that $|\bs{\alpha}| = \beta_t + \beta_2 \leq 2$ therefore we can always use integration by parts and Lemma \ref{L:products}. 
In the terms with $\widetilde{W}$ again three derivatives can fall on $\partial_t U_{\mathrm{ext},k}$. If $\partial_{x_1}$
is included, then one can proceed as above by means of \eqref{E:regularity_Uext}. Otherwise, one uses \eqref{E:regularity_Uext3} and argues as in \eqref{eq:I3}.

\smallskip

\noindent \textbf{Substep 2}: $\beta_1>1$.
In this last step we have to iterate over $\beta_1$ and increase it to $3$. For $\beta_1=2$ we set $\bs{\beta}=(1,0,0)^\top+\bs{\alpha}$ with $\bs{\alpha}=(1,\beta_2,\beta_t)^\top$. The estimates work like in Substep 1(i) since we have $|\bs{\alpha}| \leq 2$. 
Finally, for $\beta_1=3$ we have $\bs{\beta}=(3,0,0)^\top = (1,0,0)^\top + \bs{\alpha}$ with $\bs{\alpha} = (2,0,0)^\top$ and apply Substep 1(i) again. 
Here, factors $\partial_{x_1}^3\partial_t U_{\mathrm{ext},k}$ occur in the terms with $\widetilde{W}$, which are treated with \eqref{E:regularity_Uext2}.

We conclude that
\begin{equation}\label{estimate_R1_derivatives}
    \norm{\partial^{\bs{\beta}}R_{1}(\cdot,t)}_{L^2(\R^2)} \leq C \left(\rho_0 + \varepsilon^{\frac{1}{2}(\frac{3}{2}- a)} \right)
\end{equation}
for all $|\bs{\beta}|\leq 3, \beta_1\neq 0$.

In summary, combining \eqref{estimate_R1_derivatives} with \eqref{estimate_all_easy_norms}, one concludes 
\begin{equation*}
	z(t) \leq C \left(\rho_0^2 + \varepsilon^{\frac{3}{2}- a}\right)
\end{equation*}
for every $t \in\Jr$ and $\eps\in (0,\eps_0)$ if $a \in [\tfrac{5}{4},\tfrac{11}{2})$ and $\eps_0$ is small enough.

Next, we keep $\rho$ fixed, choose $a \in [\tfrac{5}{4},\tfrac{3}{2})$ and $\rho_0,\eps_0$ so small that $ C \left(\rho_0^2 + \eps_0^{3/2- a}\right)<\tfrac{1}{2}\rho^2$ and 
\[
\eps_0^a\rho + \|\Uext\|_{L^\infty(\R^2\times [0,T_0 \eps^{-2}))^3} \le \varpi 
\]
where we recall \eqref{def:varpi} and that $ \|\Uext\|_{L^\infty(\R^2\times [0,T_0 \eps^{-2}))^3} \leq C\eps \leq C\eps_0$. With this choice we have
\[
z(t)<\frac{1}{2}\rho^2 
\]
for every $t \in\Jr$ and $\eps\in (0,\eps_0)$ if $a \in [\tfrac{5}{4},\tfrac{3}{2})$. The definition \eqref{def:Trho} of $\Trho$
now implies that $\Trho=T_0\eps^{-2}<t_M$ and that \eqref{E:bootstrap-aim} holds with $t_*=T_0\eps^{-2}$.

\subsection{Final error estimate}\label{sec:final}

To finalize the proof of Theorem~\ref{T:main}, we first compare $\bs{U}_{{\rm ext}}$ from \eqref{E:Uext} and $\Uans$ from \eqref{E:Uans}. Similar as deducing 
$\Res(\Uext) \in \mathcal{G}^3(\R^2\times \Jr)^3$ from \eqref{E:ass2-Amwk}, 
one can show that the same condition yields 
\begin{align*}
	&\norm{\bs{U}_{{\rm ext}} - \Uans}_{\mathcal{G}^3(\R^2\times \Jr)^3}\\
	&\leq \norm{\left(- \varepsilon^2 \ie \partial_{X_2} A \partial_k \bs{w}(k_0) - \varepsilon^3 \frac{1}{2} \partial_{X_2}^2 A \partial_k^2 \bs{w}(k_0) + \varepsilon^3 |A|^2 A \bs{p}\right) F_1 + \cc}_{\mathcal{G}^3(\R^2\times \Jr)^3}\\
	& \quad + \norm{\varepsilon^3 A^3 \bs{h} F_1^3 + \cc}_{\mathcal{G}^3(\R^2\times \Jr)^3}\\
	&\leq C \varepsilon^{\frac{3}{2}}.
\end{align*}
Note that due to the fact that $A$ depends on  $X_2 = \varepsilon (x_2 - \nu_1 t)$, we again lose one half of the power of $\eps$ due to the substitution in the $L^2$-integral.
Second, we use \eqref{E:U-Uext-est} and the triangle inequality to conclude
\begin{align*}
	\norm{\bs{U} - \Uans}_{\mathcal{G}^3(\R^2\times \Jr)^3} &\leq \norm{\bs{U} - \bs{U}_{{\rm ext}}}_{\mathcal{G}^3(\R^2\times \Jr)^3} + \norm{\bs{U}_{{\rm ext}} - \Uans}_{\mathcal{G}^3(\R^2\times \Jr)^3} \\
	&\leq C \left(\varepsilon^a +\varepsilon^{\frac{3}{2}}\right) \leq C_\delta \varepsilon^{\frac{3}{2} - \delta}
\end{align*} 
for all $\delta >0$. \hfill $\Box$

\appendix
\section{Numerical method for the eigenvalue problem}
\label{S:numerical_eigenvalue}
The method described in this section can be found in \cite{dohnal2022quasilinear}. To solve \eqref{E:ev-prob} numerically we rewrite the problem as a second-order ordinary differential equation 
\begin{align*}
	\partial_{x_1}^2 w_3 &= \ie \partial_{x_1}\epsilon_1(x_1) \omega w_2 + \ie \epsilon_1(x_1) \omega \partial_{x_1} w_2 \\
	&= \frac{\partial_{x_1} \epsilon_1(x_1)}{\epsilon_1(x_1)} \partial_{x_1} w_3 - \epsilon_1(x_1) \omega \left( \upmu_0 \omega w_3 + k w_1\right) \\
	&= \frac{\partial_{x_1} \epsilon_1(x_1)}{\epsilon_1(x_1)} \partial_{x_1} w_3 - \epsilon_1(x_1) \upmu_0 \omega^2 w_3 + k^2 w_3
\end{align*}
on $\R\setminus\{0\}$.
The interface condition $\left\llbracket w_2 \right\rrbracket_{\text{1D}} = 0$ implies that $\left\llbracket \frac{\partial_{x_1} w_3}{\epsilon_1 } \right\rrbracket_{\text{1D}} = 0$. Now we have to solve the eigenvalue problem
\begin{equation}
	\left\{
	\begin{aligned}
		-\partial_{x_1}^2 w_3 + \frac{\partial_{x_1} \epsilon_1\left(x_1\right)}{\epsilon_1\left(x_1\right)} \partial_{x_1} w_3 + k^2 w_3 &= \epsilon_1\left(x_1\right) \upmu_0 \omega^2 w_3, & x_1 &\in \R\setminus \{0\},\\
		\left\llbracket w_3 \right\rrbracket_{\text{1D}} = \left\llbracket \frac{\partial_{x_1} w_3}{\epsilon_1 } \right\rrbracket_{\text{1D}} &= 0.
	\end{aligned} \right.
	\label{eigenvalue_problem_numerics}
\end{equation}
Note that we can use $w_1 = -\frac{k}{\epsilon_1 \omega} w_3$ and $w_2 = - \frac{\ie}{\epsilon _1 \omega} \partial_{x_1} w_3$ to calculate the remaining components of $\bs{w}$. We also see that the interface conditions $\llbracket \epsilon_1 w_1 \rrbracket_{\text{1D}} = \llbracket w_2 \rrbracket_{\text{1D}} = 0$ are satisfied if $w_3$ solves \eqref{eigenvalue_problem_numerics}. 

To simplify the numerics we write $w_3 = w_{3,\mathrm{r}} + w_{3,\mathrm{s}}$, with a smooth function $w_{3,\mathrm{r}}$ and a function $w_{3,\mathrm{s}}$ that has a discontinuous first derivative at $x_1 = 0$. For instance, we take
\begin{equation*}
	w_{3,\mathrm{s}}\left(x_1\right) = \begin{cases}
		w_{3,\mathrm{s}}^- = \mathrm{const.,} & x_1 \leq 0,\\
		w_{3,\mathrm{s}}^+\left(x_1\right), & x_1 >0,
	\end{cases}
\end{equation*}
and choose $w_{3,\mathrm{s}}^+(0) = w_{3,\mathrm{s}}^-$ so that $w_{3,\mathrm{s}}$ is continuous. Note that with this choice $w_3$ satisfies the first interface condition.
For the second interface condition we calculate $\partial_{x_1} w_3$ and get that
\begin{align*}
	\left\llbracket \frac{\partial_{x_1} w_3}{\epsilon_1 \omega} \right\rrbracket_{\text{1D}} = 0 \Longleftrightarrow&\quad \epsilon_1^-(0)\left(\partial_{x_1} w_{3,\mathrm{r}}(0) + \partial_{x_1}w_{3,\mathrm{s}}^+(0)\right) = \epsilon_1^+(0) \partial_{x_1} w_{3,\mathrm{r}}(0)\\
	\Longleftrightarrow&\quad \partial_{x_1} w_{3,\mathrm{s}}^+(0) = \frac{\epsilon_1^+(0) - \epsilon_1^-(0)}{\epsilon_1^-(0)} \partial_{x_1} w_{3,\mathrm{r}}(0) =: \widetilde{\epsilon} \partial_{x_1}w_{3,\mathrm{r}}(0).
\end{align*}
We now set
\begin{equation*}
	w_{3,\mathrm{s}}(x_1) = \left(\mathcal{L}w_{3,\mathrm{r}}\right)(x_1) := \begin{cases}
		- \sgn{\widetilde{\epsilon}} \partial_{x_1} w_{3,\mathrm{r}}(0), & x_1 < 0,\\
		- \sgn{\widetilde{\epsilon}} \partial_{x_1} w_{3,\mathrm{r}}(0) \e^{-|\widetilde{\epsilon}| x_1} , & x_1 \geq 0,
	\end{cases}
\end{equation*}
to satisfy the second interface condition. Thus, $w_{3,\mathrm{r}}$ has to solve
\begin{equation}
	\left\{
	\begin{aligned}
		\left(-\partial_{x_1}^2 + \frac{\partial_{x_1} \epsilon_1\left(x_1\right)}{\epsilon_1\left(x_1\right)} \partial_{x_1} + k^2 \right) (I + \mathcal{L})w_{3,\mathrm{r}}(x_1)
		= \epsilon_1(x_1) \upmu_0 \omega^2 &(I + \mathcal{L})w_{3,\mathrm{r}}\left(x_1\right), & x_1 &\in \R\setminus \{0\},\\
		\left\llbracket w_{3,\mathrm{r}} \right\rrbracket_{\text{1D}} = \left\llbracket \partial_{x_1} w_{3,\mathrm{r}}\right\rrbracket_{\text{1D}} &= 0.
	\end{aligned} \right.
	\label{eigenvalue_problem_numerics_2}
\end{equation}
We are interested in $H^1(\R)$-solutions, therefore we have at least the boundary conditions 
\begin{equation*}
	\lim_{x_0 \rightarrow -\infty} w_{3,\mathrm{r}}\left(x_1\right) = \sgn{\widetilde{\epsilon}} \partial_{x_1} w_{3,\mathrm{r}}(0), \qquad \lim_{x_1 \rightarrow \infty} w_{3,\mathrm{r}}\left(x_1\right) = 0.
\end{equation*}

To solve \eqref{eigenvalue_problem_numerics_2} numerically for a fixed $k \in \R$ we discretize the problem over a finite interval $[-d,d] \subset \R$ and apply a solver for a generalized eigenvalue problem, e.g. a solver based on a Krylov-Schur algorithm, see e.g. \cite{stewart2001krylov}.

To be more precise, we used $d$ ranging from $10^{2}$ to $10^{4}$ with the step size $h = 0.01$ in space. We used the second-order difference quotients with zero Dirichlet boundary conditions to discretize the derivatives. The generalized eigenvalue problem was then solved with the Matlab functions ``eigs'', where we calculated the first $10$ eigenvalues closest to $\nu_0$ with a convergence tolerance of $10^{-10}$. We then only selected solutions where the corresponding eigenfunctions were almost zero in a small neighborhood of the boundary of $[-d,d]$, i.e., the norm of $w_3$ on $[-d,-d + 100h] \cup [d- 100h,d]$ is smaller than $10^{-6}$.

\section{Residual of order $\bs{\eps^4}$}
\label{S:resiudal_order_4}
For $\Res := \Res(\Uext)$ and $F_1 = \e^{\ie\left(k_0x_2 - \nu_0t \right)}$ we have 
\begin{align*}
	&\operatorname{Res}_3 = F_1 \varepsilon^4 \left(\frac{1}{2} \partial_{X_2}^3 A \partial_k^2 w_1\left(k_0\right) - \left(2|A|^2 \partial_{X_2} A + A^2 \partial_{X_2} \overline{A}\right) p_1 \right) \\
	&\!\quad \!+ F_1 \varepsilon^4 \upmu_0 \left(\frac{\nu_1}{2} \partial_{X_2}^3 A \partial_k^2 w_3\left(k_0\right) - \ie \partial_T \partial_{X_2} A \partial_k w_3\left(k_0\right) - \nu_1 \left(2|A|^2 \partial_{X_2} A + A^2 \partial_{X_2} \overline{A}\right) p_3\right)\\
	&\!\quad \!- 3 F_1^3 \varepsilon^4 \left( \partial_{X_2} A A^2 h_1 + \upmu_0 \nu_1 A^2 \partial_{X_2} A h_3 \right) + \cc + \mathcal{O}\left(\eps^5 \right)
\end{align*}
and the parts of $\operatorname{Res}_1$ and $\operatorname{Res}_2$ that are linear in $\bs{U}_{\mathrm{ext}}$ are given by
\begin{align*}
	&\operatorname{Res}_{\mathrm{lin},1} = F_1 \varepsilon^4 \left(\frac{1}{2} \partial_{X_2}^3 A \partial_k^2 w_3\left(k_0\right) - \left(2|A|^2 \partial_{X_2} A + A^2 \partial_{X_2} \overline{A}\right) p_3 \right) \\
	&\!\quad \! + F_1 \varepsilon^4 \epsilon_1 \left(\frac{\nu_1}{2} \partial_{X_2}^3 A \partial_k^2 w_1\left(k_0\right) - \ie \partial_T \partial_{X_2} A \partial_k w_1\left(k_0\right) - \nu_1 \left(2|A|^2 \partial_{X_2} A + A^2 \partial_{X_2} \overline{A}\right) p_1\right)\\
	&\!\quad \! - 3F_1^3 \varepsilon^4 \left( \epsilon_1 \nu_1 A^2 \partial_{X_2} A h_1 + A^2 \partial_{X_2} A h_3\right) + \cc + \mathcal{O}\left(\eps^5 \right), \\
	&\operatorname{Res}_{\mathrm{lin},2} = F_1 \varepsilon^4 \epsilon_1 \left(\frac{\nu_1}{2} \partial_{X_2}^3 A \partial_k^2 w_2\left(k_0\right) - \ie \partial_T \partial_{X_2} A \partial_k w_2\left(k_0\right)\right)\\
	&\!\quad \! - F_1 \varepsilon^4 \epsilon_1 \nu_1 \left(2|A|^2 \partial_{X_2} A + A^2 \partial_{X_2} \overline{A}\right) p_2 - 3F_1^3 \varepsilon^4 \epsilon_1 \nu_1 A^2 \partial_{X_2} A h_2 
	+ \cc + \mathcal{O}\left(\eps^5 \right).
\end{align*}
For the nonlinear part of $\operatorname{Res}_1$ we get
\begin{align*}
	\operatorname{Res}_{\mathrm{nl},1} = -\varepsilon^4 \epsilon_3 &\Big(3 \nu_1 F_1^3 A^2 \partial_{X_2} A \left(m_1^3 + m_1 m_2^2\right) \\
	& + \left. \nu_0 F_1^3 A^2 \partial_{X_2} A \left(3 m_1^2 \partial_k w_1\left(k_0\right) + m_2^2 \partial_k w_1\left(k_0\right) + 2 m_1 m_2 \partial_k w_2\left(k_0\right) \right) \right.\\
	& + \left. \nu_1 F_1 A^2 \partial_{X_2} \overline{A} \left(3 |m_1|^2 m_1 + 2 m_1 |m_2|^2 + \overline{m}_1 m_2^2\right) \right. \\
	& + \left. \nu_0 F_1 A^2 \partial_{X_2} \overline{A} \left(3 m_1^2 \partial_k \overline{w}_1\left(k_0\right) + m_2^2 \partial_k \overline{w}_1\left(k_0\right) + 2 m_1 m_2 \partial_k \overline{w}_2\left(k_0\right) \right) \right.\\
	& + \left. 2 \nu_0 F_1 |A|^2 \partial_{X_2} A \left(3 |m_1|^2 \partial_k w_1\left(k_0\right) + |m_2|^2 \partial_k w_1\left(k_0\right) \right)\right.\\
	& + \left. 2 \nu_0 F_1 |A|^2 \partial_{X_2} A \left(\overline{m}_1 m_2 \partial_k w_2\left(k_0\right) + m_1 m_2 \partial_k w_2\left(k_0\right) \right)\right.\\
	& + 2 \nu_1 F_1 |A|^2 \partial_{X_2} A \left(3 |m_1|^2 m_1 + 2 m_1 |m_2|^2 + \overline{m}_1 m_2^2 \right)\!\Big) + \cc + \mathcal{O}\left(\varepsilon^5\right),
\end{align*}
and for $\operatorname{Res}_{\mathrm{nl},2}$ we simply have to change the indices of the components of $\bs{m}$ and $\partial_k \bs{w}\left(k_0\right)$ in $\operatorname{Res}_{\mathrm{nl},1}$.

\section{Calculus Lemma}\label{A:products}
\begin{lem}	\label{L:products}~
	Let $m_1,m_2 \in \N_0$ with $m_1 \geq m_2$ and $m_1 \geq 2$ and let $J\subset \R$ be an interval.
	\begin{enumerate}[label=(\roman*)]
		\item Let $j \in \{0,\dots,m_1\}$, $f \in \mathcal{H}^{m_1-j}(\R^2)$ and $g \in \mathcal{H}^j(\R^2)$. Then $fg \in L^2(\R^2)$ and
		\begin{equation*}
			\norm{fg}_{L^2(\R^2)} \leq C \norm{f}_{\mathcal{H}^{m_1-j}(\R^2)}\norm{g}_{\mathcal{H}^j(\R^2)}. 
		\end{equation*}
		\item Let $f \in \mathcal{H}^{m_1}(\R^2)$ and $g \in \mathcal{H}^{m_2}(\R^2)$. Then $fg \in \mathcal{H}^{m_2}(\R^2)$ and
		\begin{equation*}
			\norm{fg}_{\mathcal{H}^{m_2}(\R^2)} \leq C \norm{f}_{\mathcal{H}^{m_1}(\R^2)}\norm{g}_{\mathcal{H}^{m_2}(\R^2)}. 
		\end{equation*}
		\item Let $f \in \mathcal{F}_0^{m_1,1}(\R^2)$ and $g \in \mathcal{H}^{m_2}(\R^2)$. Then $fg \in \mathcal{H}^{m_2}(\R^2)$ and
		\begin{equation*}
			\norm{fg}_{\mathcal{H}^{m_2}(\R^2)} \leq C \norm{f}_{\mathcal{F}_0^{m_1,1}(\R^2)}\norm{g}_{\mathcal{H}^{m_2}(\R^2)}.
		\end{equation*}
		\item Let $f \in \mathcal{F}^{m_1,1}(\R^2 \times J)$ and $g \in \mathcal{G}^{m_2}(\R^2 \times J)$. Then $fg \in \mathcal{G}^{m_2}(\R^2 \times J)$ and
		\begin{equation*}
			\norm{fg}_{\mathcal{G}^{m_2}(\R^2 \times J)} \leq C \norm{f}_{\mathcal{F}^{m_1,1}(\R^2 \times J)}\norm{g}_{\mathcal{G}^{m_2}(\R^2 \times J)}.
		\end{equation*}
		\item Let $f \in \mathcal{F}_0^{m_1,1}(\R^2)$ and $g \in \mathcal{F}_0^{m_2,1}(\R^2)$. Then $fg \in \mathcal{F}_0^{m_2,1}(\R^2)$ and
		\begin{equation*}
			\norm{fg}_{\mathcal{F}_0^{m_2,1}(\R^2)} \leq C \norm{f}_{\mathcal{F}_0^{m_1,1}(\R^2)}\norm{g}_{\mathcal{F}_0^{m_2,1}(\R^2)}. 
		\end{equation*}
	\end{enumerate}
\end{lem}
\begin{proof}
	The proofs can be done analogously to Lemma 2.22 in \cite{spitz2017local}. 
\end{proof}

\section*{Acknowledgements}
The authors thank the referees for their comments and suggestions which led to an improvement of the presentation. Tom\'{a}\v{s} Dohnal and Daniel P. Tietz acknowledge funding by the Deutsche Forschungsgemeinschaft (DFG, German Research Foundation), Project-ID DO1467/4-1. Roland Schnaubelt acknowledges funding by the Deutsche Forschungsgemeinschaft (DFG, German Research Foundation) - Project-ID 258734477 - SFB 1173.

\bibliographystyle{plain}
\bibliography{references}

\begin{thebibliography}{10}

\bibitem{agrawal2013}
G.~Agrawal.
\newblock {\em Nonlinear Fiber Optics}.
\newblock Elsevier Science, 2013.

\bibitem{BF05}
A.~Babin and A.~Figotin.
\newblock {Nonlinear photonic crystals: IV. Nonlinear Schr\"odinger equation
  regime}.
\newblock {\em Waves Random Complex Media}, 15(2):145--228, 2005.

\bibitem{ben1992dichotomy}
A.~Ben-Artzi and I.l Gohberg.
\newblock Dichotomy of systems and invertibility of linear ordinary
  differential operators.
\newblock In {\em Time-Variant Systems and Interpolation}, volume~56 of {\em
  Oper. Theory Adv. Appl.}, pages 90--119. Birkh\"{a}user, Basel, 1992.

\bibitem{boyd2020nonlinear}
R.~Boyd.
\newblock {\em Nonlinear optics}.
\newblock Academic press, 2020.

\bibitem{BDPW21}
M.~Brown, T.~Dohnal, M.~Plum, and I.~Wood.
\newblock {Spectrum of the Maxwell Equations for a Flat Interface between
  Dispersive Media}.
\newblock 2022.
\newblock submitted, arXiv 2206.02037.

\bibitem{BSTU06}
K.~Busch, G.~Schneider, L.~Tkeshelashvili, and H.~Uecker.
\newblock {Justification of the nonlinear Schr\"odinger equation in spatially
  periodic media}.
\newblock {\em {Z. Angew. Math. Phys.}}, 57(6):905--939, 2006.

\bibitem{coppel1971dichotomies}
W.~Coppel.
\newblock {\em Dichotomies in Stability Theory}.
\newblock Springer-Verlag, Berlin, New York, 1978.

\bibitem{CO11}
S.~Crutcher and A.~Osei.
\newblock {Derivation of the Effective Nonlinear Schr\"odinger Equations for
  Dark and Power Law Spatial Plasmon-Polariton Solitons Using Nano
  Self-Focusing}.
\newblock {\em Progress In Electromagnetics Research B}, 29:83--103, 2011.

\bibitem{DL1990}
R.~Dautray and J.~Lions.
\newblock {\em Mathematical Analysis and Numerical Methods for Science and
  Technology: Volume 1 Physical Origins and Classical Methods}.
\newblock Springer-Verlag, 1990.

\bibitem{Davoyan:09}
A.~Davoyan, I.~Shadrivov, and Y.~Kivshar.
\newblock Self-focusing and spatial plasmon-polariton solitons.
\newblock {\em Opt. Express}, 17(24):21732--21737, Nov 2009.

\bibitem{Davydova:15}
M.~Davydova, D.~Dodonov, A.~Kalish, V.~Belotelov, and A.~Zvezdin.
\newblock Schr\"odinger plasmon solitons in kerr nonlinear heterostructures
  with magnetic manipulation.
\newblock {\em Opt. Lett.}, 40(23):5439--5442, 2015.

\bibitem{dohnal2022quasilinear}
T.~Dohnal, G.~Romani, and D.~Tietz.
\newblock A quasilinear transmission problem with application to maxwell
  equations with a divergence-free {D}-field.
\newblock {\em J. Math. Anal. Appl.}, 511(1):126067, 2022.

\bibitem{DR20}
T.~Dohnal and D.~Rudolf.
\newblock {NLS} approximation for wavepackets in periodic cubically nonlinear
  wave problems in $\mathbb{R}^d$.
\newblock {\em Appl.\ Anal.}, 99(10):1685--1723, 2020.

\bibitem{dull2018existence}
W.-P. D\"{u}ll and M.~He\ss.
\newblock Existence of long time solutions and validity of the nonlinear
  {S}chr\"{o}dinger approximation for a quasilinear dispersive equation.
\newblock {\em J. Differential Equations}, 264(4):2598--2632, 2018.

\bibitem{eller2012symmetric}
M.~Eller.
\newblock On symmetric hyperbolic boundary problems with nonhomogeneous
  conservative boundary conditions.
\newblock {\em SIAM J. Math. Anal.}, 44(3):1925--1949, 2012.

\bibitem{feynman1979feynman}
R.~Feynman, R.~Leighton, and M.~Sands.
\newblock {\em The Feynman Lectures on Physics, Volume 2: Mainly
  Electromagnetism and Matter}.
\newblock Addison-Wesley, 1979.

\bibitem{JAMOIS20031}
C.~Jamois, R.~Wehrspohn, L.~Andreani, C.~Hermann, O.~Hess, and U.~G\"osele.
\newblock Silicon-based two-dimensional photonic crystal waveguides.
\newblock {\em Photonics Nanostructures: Fundam.\ Appl.}, 1(1):1--13, 2003.

\bibitem{Kalyakin1989}
L.~Kalyakin.
\newblock Long wave asymptotics. integrable equations as asymptotic limits of
  non-linear systems.
\newblock {\em Russian Math.\ Surveys}, 44(1):3--42, 1989.

\bibitem{kato1995perturbation}
T.~Kato.
\newblock {\em Perturbation Theory for Linear Operators}.
\newblock Springer-Verlag, Berlin, 1995.
\newblock Reprint of the 1980 edition.

\bibitem{Kauranen2012}
M.~Kauranen and A.~Zayats.
\newblock Nonlinear plasmonics.
\newblock {\em Nature Photonics}, 6(11):737--748, Nov 2012.

\bibitem{KSM1992}
P.~Kirrmann, G.~Schneider, and A.~Mielke.
\newblock The validity of modulation equations for extended systems with cubic
  nonlinearities.
\newblock {\em Proc.\ Roy. Soc.\ Edinburgh A}, 122(1-2):85--91, 1992.

\bibitem{Kivshar:02}
Y.~Kivshar and S.~Mingaleev.
\newblock Nonlinear photonic crystals: waveguides, all-optical switching, and
  solitons.
\newblock In {\em Nonlinear Optics: Materials, Fundamentals and Applications},
  page ThD1. Optica Publishing Group, 2002.

\bibitem{LesSchnei2012}
V.~Lescarret and G.~Schneider.
\newblock Diffractive optics with harmonic radiation in 2d nonlinear photonic
  crystal waveguide.
\newblock {\em Z. Angew.\ Math.\ Phys.}, 63(3):401--427, Jun 2012.

\bibitem{Li:89}
G.~Li and S.~Seshadri.
\newblock Weakly nonlinear surface polariton.
\newblock {\em J. Opt. Soc. Am. B}, 6(6):1125--1137, 1989.

\bibitem{maier2007plasmonics}
S.~Maier.
\newblock {\em Plasmonics: Fundamentals and Applications}.
\newblock Springer, 2007.

\bibitem{MS2010}
A.~Marini and D.~Skryabin.
\newblock {Ginzburg-Landau} equation bound to the metal-dielectric interface
  and transverse nonlinear optics with amplified plasmon polaritons.
\newblock {\em Phys. Rev. A}, 81:033850, 2010.

\bibitem{schnaubelt2018local}
R.~Schnaubelt and M.~Spitz.
\newblock Local wellposedness of quasilinear {M}axwell equations with
  conservative interface conditions.
\newblock {\em Commun. Math. Sci.}, 20(8):2265--2313, 2022.

\bibitem{Schneider-95}
G.~Schneider.
\newblock Validity and limitation of the {N}ewell-{W}hitehead equation.
\newblock {\em Math.\ Nachr.}, 176(1):249--263, 1995.

\bibitem{SSZ15}
G.~Schneider, D.~Sunny, and D.~Zimmermann.
\newblock The {NLS} approximation makes wrong predictions for the water wave
  problem in case of small surface tension and spatially periodic boundary
  conditions.
\newblock {\em J. Dynam.\ Differential Equations}, 27(3):1077--1099, 2015.

\bibitem{SU03}
G.~Schneider and H.~Uecker.
\newblock Existence and stability of modulating pulse solutions in {M}axwell's
  equations describing nonlinear optics.
\newblock {\em Z. Angew.\ Math.\ Phys.}, 54(4):677--712, 2003.

\bibitem{spitz2017local}
M.~Spitz.
\newblock {\em Local wellposedness of nonlinear Maxwell equations}.
\newblock PhD thesis, Karlsruhe Institute of Technology, Karlsruhe, 2017.

\bibitem{stewart2001krylov}
G.~Stewart.
\newblock A {K}rylov-{S}chur algorithm for large eigenproblems.
\newblock {\em SIAM J. Matrix Anal. Appl.}, 23(3):601--614, 2001.

\bibitem{sulem1999nonlinear}
C.~Sulem and P.-L. Sulem.
\newblock {\em The nonlinear {S}chr\"{o}dinger equation: Self-focusing and wave
  collapse}, volume 139 of {\em Applied Mathematical Sciences}.
\newblock Springer-Verlag, New York, 1999.

\bibitem{tao2006nonlinear}
T.~Tao.
\newblock {\em Nonlinear Dispersive Equations}.
\newblock American Mathematical Society, Providence, RI, 2006.

\bibitem{PRL-2008}
Z.~Wang, Y.~Chong, J.~Joannopoulos, and M.~Solja\ifmmode \check{c}\else
  \v{c}\fi{}i\ifmmode~\acute{c}\else \'{c}\fi{}.
\newblock Reflection-free one-way edge modes in a gyromagnetic photonic
  crystal.
\newblock {\em Phys. Rev. Lett.}, 100:013905, 2008.

\bibitem{yosida1980functional}
K.~Yosida.
\newblock {\em Functional analysis}.
\newblock Springer-Verlag, Berlin, 1980.

\end{thebibliography}
\end{document}